\documentclass[reqno]{amsart}

\usepackage[bookmarks,colorlinks=true,citecolor=black,linkcolor=black,pagecolor=black,urlcolor=black,pagebackref=true]{hyperref}

\usepackage{amsmath,amssymb}
\usepackage{mathabx}
\usepackage{wasysym}
\usepackage{mathtools}

\usepackage[size=small,width=\linewidth]{caption}
\usepackage{fmtcount}

\usepackage{epsfig,psfrag,color}
\definecolor{darkgreen}{rgb}{0.00,0.50,0.10}
\definecolor{lightgreen}{rgb}{0.20,0.70,0.30}

\usepackage{enumitem}
\usepackage{booktabs}
\usepackage{stmaryrd}

\newtheorem{theorem}{Theorem}
\newtheorem{lemma}[theorem]{Lemma} 
\newtheorem{corollary}[theorem]{Corollary} 
\newtheorem{proposition}[theorem]{Proposition}

\newtheorem{definition}[theorem]{Definition}

\newtheorem{conjecture}[theorem]{Conjecture}

\newcommand{\Z}{\mathbb{Z}}

\newcommand{\upb}{\mathrm{b}}
\newcommand{\upc}{\mathrm{c}}

\newcommand{\upe}{\mathrm{e}}
\newcommand{\uph}{\mathrm{h}}

\newcommand{\upo}{\mathrm{o}}

\newcommand{\upw}{\mathrm{w}}

\newcommand{\upC}{\mathrm{C}}

\newcommand{\upE}{\mathrm{E}}

\newcommand{\upK}{\mathrm{K}}

\newcommand{\upM}{\mathrm{M}}
\newcommand{\upN}{\mathrm{N}}
\newcommand{\NCL}{\mathrm{NCL}}
\newcommand{\upP}{\mathrm{P}}

\newcommand{\upV}{\mathrm{V}}
\newcommand{\upX}{\mathrm{X}}

\newcommand{\upZ}{\mathrm{Z}}

\newcommand{\Nsi}{\mathrm{Nsi}}

\newcommand{\bw}{\mathrm{bw}}

\newcommand{\Supp}{\mathrm{Supp}}

\newcommand{\CL}{\mathrm{CL}}

\newcommand{\rk}{\mathrm{rk}}

\newcommand{\ev}{\mathrm{ev}}
\newcommand{\od}{\mathrm{od}}

\DeclareMathOperator{\Hom}{Hom}

\DeclareMathOperator{\Cay}{Cay}
\DeclareMathOperator{\Aut}{Aut}

{\begin{list}{}%
         {\setlength{\leftmargin}{#1}}%
         \item[]%
}
{\end{list}}

\title[Prisms, M{\"o}bius ladders, cycle space of dense graphs]{On prisms, 
M{\"o}bius ladders and the cycle space of dense graphs}

\author{Peter Heinig}

\address{Zentrum~Mathematik, M9, Technische~Universit{\"a}t~M{\"u}nchen, \newline
Boltzmannstra{\ss}e~3, D-85747 Garching~bei~M{\"u}nchen, Germany} 

\email{heinig@ma.tum.de}

\thanks{The author is partially supported by DFG grant TA 309/2-2, 
by the ENB graduate program TopMath and by TUM Graduate School. 
A large part of the present work was done while the author was 
partially supported by a scholarship from the Max Weber-Programm Bayern.}

\begin{document}

\begin{abstract}
For a graph $X$, let $f_0(X)$ denote its number of vertices, $\delta(X)$ its 
minimum degree and $\upZ_1(X;\Z/2)$ its cycle space in the standard 
graph-theoretical sense (i.e. $1$-dimensional cycle group in the sense of 
simplicial homology theory with $\Z/2$-coefficients). Call a 
graph \emph{Hamilton-generated} if and only if the set of all Hamilton circuits 
is a $\Z/2$-generating system for $\upZ_1(X;\Z/2)$. The 
main purpose of this paper is to prove the following: for every $\gamma>0$ there 
exists $n_0\in \Z$ such that for every graph $X$ with $f_0(X)\geq n_0$ vertices, \\
(1) if $\delta(X)\geq (\tfrac12 + \gamma) f_0(X)$ and $f_0(X)$ is odd, 
then $X$ is Hamilton-generated, \\
(2) if $\delta(X)\geq (\tfrac12 + \gamma) f_0(X)$ and $f_0(X)$ is even, 
then the set of all Hamilton circuits of $X$ generates a codimension-one 
subspace of $\upZ_1(X;\Z/2)$, and the set of all circuits of $X$ having 
length either $f_0(X)-1$ or $f_0(X)$ generates all of $\upZ_1(X;\Z/2)$, \\
(3) if $\delta(X)\geq (\tfrac14 + \gamma) f_0(X)$ and $X$ is square bipartite, 
then $X$ is Hamilton-generated. 

All these degree-conditions are essentially best-possible. 
The implications in (1) and (2) give an asymptotic affirmative answer to 
a special case of an open conjecture which according 
to [\textit{European~J.~Combin.}~4~(1983),~no.~3,~p.~246] originates with A.~Bondy.  

\end{abstract}

\maketitle 

\section{Introduction}\label{9877667556545454546565}

There exist investigations in which the set underlying a finite-dimensional vector 
space is not forgotten, but made to play a central part. One such investigation was 
begun thirty years ago by I.~B.-A.~Hartman and concerns the cycle space 
$\upZ_1(X;\Z/2)$ of a finite graph $X$ (whose vectors are the Eulerian 
subgraphs of $X$): under what conditions does $\upZ_1(X;\Z/2)$ admit a basis 
over $\Z/2$ consisting of \emph{long} graph-theoretical circuits only? Hartman proved 
\cite[Theorem~1]{MR725072} a theorem which guarantees that---barring the sole 
exception of $X$ being a complete graph with an even number of vertices---for 
every $2$-connected finite graph $X$, the set of all circuits of 
length \emph{at least} $\delta(X)+1$ generates $\upZ_1(X;\Z/2)$. 

The lower the minimum degree $\delta(X)$, the larger the set of 
cycle-lengths one has to allow in order to be guaranteed a generating system by 
Hartman's theorem. In particular, statements guaranteeing a generating system 
consisting entirely of \emph{Hamilton circuits} (a natural thing to ask for once 
the topic of long circuits has been broached) remain almost inaccessible via this 
theorem: one has to set $\delta(X) := f_0(X) - 1$, hence $X\cong \upK^{f_0(X)}$, and 
what remains of Hartman's general theorem is a rather special (albeit still 
non-obvious) statement about the complete graph. 

The property of $\upZ_1(X;\Z/2)$ being generated by the Hamilton circuits of $X$ 
seems to have been first studied by 
B.~Alspach, S.~C.~Locke and D.~Witte \cite{MR1057481}. They proved that $X$ has 
the property if $X$ is a connected Cayley graph on a finite abelian group and 
is either bipartite or has odd order (these hypotheses being mutually 
exclusive for connected Cayley graphs on finite abelian groups). 

Here, we will for the first time prove minimum degree conditions 
guaranteeing this property (Section~\ref{sec:concludingremarks} contains a 
short survey of the relevant literature.) We will accomplish this by way of a 
two-layered strategy which first harnesses theorems from extremal graph theory 
to prove the existence of certain spanning subgraphs which can be used to transfer 
the property to the entire ambient graph in a second step. 
The main purpose of the present paper is to prove 
the following previously unknown implications: 
\begin{theorem}[{sufficient conditions for a cycle space generated by 
Hamilton circuits; \ref{thm:minDegOneFourthPlusGammaImpliesHamiltonGeneratednessInBipartiteGraphs} had already been announced in \cite{MR2735919}}]\label{thm:mainresults}
For every $\gamma>0$ there exists $n_0\in \Z$ such that for every graph $X$ with 
$f_0(X)\geq n_0$, the following is true: 
\begin{enumerate}[label={\rm(I\arabic{*})}]
\item\label{thm:minDegOneHalfPlusGammaImpliesHamiltonGeneratedForOddOrder} 
if $\delta(X)\geq(\tfrac12 + \gamma) f_0(X)$ and $f_0(X)$ is odd, 
then $X$ is Hamilton-generated,
\item\label{thm:minDegOneHalfPlusGammaImpliesAllOneCanAskForWhenOrderIsEven} 
if $\delta(X)\geq(\tfrac12 + \gamma) f_0(X)$ and $f_0(X)$ is even, then the set of 
all Hamilton-circuits of $X$ generates a codimension-one subspace 
of $\upZ_1(X;\Z/2)$ and the set of all circuits of $X$ with lengths 
either $f_0(X)-1$ or $f_0(X)$ generates all of $\upZ_1(X;\Z/2)$,
\item\label{thm:minDegOneFourthPlusGammaImpliesHamiltonGeneratednessInBipartiteGraphs} if $\delta(X)\geq(\frac14 + \gamma) f_0(X)$ and $X$ is square bipartite, 
then $X$ is Hamilton-generated,
\item\label{thm:minDegTwoThirds} 
if in \ref{thm:minDegOneHalfPlusGammaImpliesHamiltonGeneratedForOddOrder} 
and \ref{thm:minDegOneHalfPlusGammaImpliesAllOneCanAskForWhenOrderIsEven} 
the condition `$\delta(X)\geq(\tfrac12 + \gamma) f_0(X)$' is replaced by 
`$\delta(X)\geq \frac23 f_0(X)$', then without further change to 
\ref{thm:minDegOneHalfPlusGammaImpliesHamiltonGeneratedForOddOrder} 
or \ref{thm:minDegOneHalfPlusGammaImpliesAllOneCanAskForWhenOrderIsEven} 
it suffices to take $n_0:=2\cdot 10^8$.
\end{enumerate}
Implication \ref{thm:minDegOneHalfPlusGammaImpliesHamiltonGeneratedForOddOrder} 
becomes false if `~$(\tfrac12 + \gamma) f_0(X)$~' is replaced by 
`~$\lfloor \tfrac{f_0(X)}{2}\rfloor$ and $X$ Hamilton-connected'. 
Implication 
\ref{thm:minDegOneFourthPlusGammaImpliesHamiltonGeneratednessInBipartiteGraphs} 
becomes false if `~$(\tfrac14 + \gamma) f_0(X)$~' is replaced by 
`~$\tfrac14 f_0(X)$ and $X$ hamiltonian'.  
\end{theorem}

A purely combinatorial way of phrasing the conclusions in 
Theorem~\ref{thm:mainresults} is to say that `every circuit in $X$ can be 
realized as a symmetric difference of some Hamilton circuits of $X$'. In 
this variant phrasing, talking only about graph-theoretical circuits (and not more 
generally about cycles in the sense of homology theory) does not lose any 
generality since for any graph $X$ and any cycle $c\in \upZ_1(X;\Z/2)$, the 
support $\Supp(c)$ is an edge-disjoint union of graph-theoretical circuits 
\cite[Proposition 1.9.2]{MR2159259}. Let us note in passing that the latter fact 
generalizes to locally-finite \emph{infinite} graphs 
\cite[Theorem~7.2, equivalence (i) $\Leftrightarrow$ (iii)]{MR2057684}, that 
it has been given a precise sense for arbitrary 
\emph{compact metric spaces} \cite[Theorem~1.4]{arXiv:1003.5115v3}, and, last but 
not least, that \emph{linear-algebraic properties of Hamilton circles} (in the 
sense of \cite{MR2443120}) \emph{in infinite graphs}---i.e. the role of 
infinite Hamilton circles vis-{\`a}-vis the cycle space (in the sense 
of \cite{MR2080038} \cite{MR2128083} \cite{MR2057684} \cite{MR2789733})---is an 
unexplored research topic.

Theorem~\ref{thm:mainresults}, the main result of the present paper, adds to the 
growing corpus of knowledge about the following phenomenon: 
when studying the set of Hamilton circuits as a function 
of the minimum degree $\delta(X)$, it pours if it rains---slightly below a 
sufficient threshold there still exist graphs which do not 
have \emph{any} Hamilton circuit, slightly above the threshold suddenly every 
graph contains not merely one but rather a plethora of Hamilton circuits 
satisfying many \emph{additional requirements}. This line of investigation 
appears to begin with C.~St.~J.~A.~Nash-Williams' 
proof \cite[Theorem~2]{MR0299517} \cite[Theorem~3]{MR0284366} that for every 
graph $X$ with $\delta(X)\geq \tfrac12 f_0(X)$ there exists not only one 
(Dirac's theorem \cite[Theorem~3]{MR0047308} \cite[Theorem~10.1.1]{MR2159259}) but 
at least $\lfloor \tfrac{5}{224} n \rfloor$ edge-disjoint Hamilton circuits. For 
sufficiently large graphs $X$ with $\delta(X)$ a little larger than 
$\tfrac12 f_0(X)$, Nash-Williams' theorem was improved by 
D.~Christofides, D.~K{\"u}hn and D.~Osthus \cite[Theorem~2]{arXiv:0908.4572v1} to 
the guarantee that there are at least $\tfrac18 n$ edge-disjoint Hamilton 
circuits---this being an asymptotically best-possible result in view of examples 
\cite[p.~818]{MR0299517} which show that in graphs $X$ with 
$\delta(X)\geq \tfrac12 f_0(X)$ and having a slightly irregular degree sequence, 
the number of edge-disjoint Hamilton circuits is bounded by $\tfrac18 n$. More can 
be achieved if besides a high minimum-degree, additional requirements are imposed 
on the host graph. Two aspects of this are (1) a regular degree sequence, (2) a 
random host graph. 

As to (1), if the host graph is required to be regular in 
advance, a still unsettled conjecture of B.~Jackson \cite[p.~13, l.~17]{MR527728} 
posits that a $d$-regular graph with $d \geq \tfrac{f_0(X)-1}{2}$ actually 
realizes the obvious upper bound $\lfloor \tfrac12 d \rfloor$ for the number of 
edge-disjoint Hamilton circuits. Christofides, K{\"u}hn and Osthus proved a 
theorem which in a sense comes arbitrarily close to the 
conjecture \cite[Theorem~5]{arXiv:0908.4572v1}. 

As to (2), A.~Frieze and M.~Krivelevich conjectured \cite[p.~222]{MR2430433} that 
for \emph{any} $0\leq p_n \leq 1$ an Erd{\H o}s--R{\'e}nyi random graph 
$\mathbb{G}_{n,p_n}$ a.a.s. attains the 
a priori maximum of $\lfloor \delta/2 \rfloor$ edge-disjoint Hamilton-circuits. 
(For $p_n$ which are low enough to a.a.s. imply $\delta \leq 1$ the 
conjecture claims nothing.) They proved \cite[Theorem~1]{MR2430433} the 
conjecture for $p_n \leq (1+o(1)) \tfrac{\log n}{n}$. 
In \cite[Theorem~2]{arXiv:1104.4412v1} F.~Knox, D.~K{\"u}hn and D.~Osthus proved 
the conjecture for a class of functions $p_n$ that sweeps a huge portion of the 
range $\tfrac{\log n}{n} \ll p_n \ll 1$. A remaining gap (starting at 
$\tfrac{\log n}{n}$) in the probability range heretofore covered was recently 
closed by M.~Krivelevich and W.~Samotij \cite{arXiv:1109.5341v1}. 
According to \cite[p.~2]{arXiv:1109.5341v1} the conjecture now remains open only 
for $p_n\geq 1-(\log(n))^{9}\cdot n^{-\frac14}$, i.e. for unusually dense 
Erd{\H o}s--R{\'e}nyi random graphs.  

One way to look at these results is as providing `extremely orthogonal' 
(i.e. no additive cancellation is involved in the vanishing of the standard 
bilinear form) sets of Hamilton circuits.  As they stand, these theorems are far 
from providing `orthogonal' Hamilton-circuit-bases for $\upZ_1(X;\Z/2)$: at the 
relevant minimum degrees, the dimension of $\upZ_1(X;\Z/2)$ is much 
higher than $\delta(X)/2$ (roughly, one has 
$\dim_{\Z/2} \upZ_1(X;\Z/2) \in \Theta_{f_0(X)\to\infty}(\delta(X)^2)$ ), so the sets 
of mutually disjoint Hamilton circuits are---while `very' orthogonal---far from 
being generating sets of $\upZ_1(X;\Z/2)$. Yet it does not seem unlikely that the 
above-mentioned theorems can be extended in a more algebraic vein by devising 
generalizations of `edge-disjoint' (e.g. `size of the intersection of the 
supports even') and thus be made to resonate with results like 
Theorem~\ref{thm:mainresults}. 

Further context for Theorem~\ref{thm:mainresults} is provided by the following 
open conjecture (thirty years ago, S.~C.~Locke proved 
\cite[Theorem~2 and Corollary~4]{MR821540} that Bondy's conjecture is true under 
the \emph{additional} assumption of `$X$ non-hamiltonian or $f_0(X)\geq 4d-5$'): 
\begin{conjecture}[{J.~A.~Bondy~1979;~\cite[p.~246]{MR725072} \cite[Conjecture~1]{MR821540} \cite[p.~256]{MR815581} \cite[Conjecture~1]{MR1174832} \cite[Conjecture~A]{MR1760293} \cite[p.~21]{MR2171084} \cite[p.~12]{MR2279163}}]\label{thm:BondyConjecture}
For every $d\in \Z$, in every vertex-$3$-connected graph $X$ with $f_0(X) \geq 2d$ 
and $\delta(X)\geq d$, the set of all circuits of length at least $2d-1$ is a 
$\Z/2$-generating system of $\upZ_1(X;\Z/2)$. 
\end{conjecture}

The present paper gives an asymptotic answer for two special cases of 
Conjecture~\ref{thm:BondyConjecture}: If `$\delta(X)\geq d$' is replaced by 
$\delta(X)\geq (1+\gamma) d$ for an arbitrary $\gamma>0$, and if $f_0(X)$ is 
sufficiently large, then 
\ref{thm:minDegOneHalfPlusGammaImpliesAllOneCanAskForWhenOrderIsEven} in 
Theorem~\ref{thm:mainresults} below says that in the case of `$f_0(X)\geq 2d$' 
holding as `$f_0(X) = 2d$', Bondy's conclusion is true; in case that 
`$f_0(X)\geq 2d$' holds as `$f_0(X) = 2d+1$', then 
\ref{thm:minDegOneHalfPlusGammaImpliesHamiltonGeneratedForOddOrder} in 
Theorem~\ref{thm:mainresults} says that of the three circuit lengths 
$f_0(X)-2$, $f_0(X)-1$ and $f_0(X)$ which Bondy allows as lengths of the 
generating circuits, $f_0(X)$ alone is enough. It seems likely that with the 
techniques of this paper it will be possible to make further inroads towards 
the full Conjecture~\ref{thm:BondyConjecture}. 

\addtocounter{footnote}{1}

\begin{center}
\begin{table}
    \begin{tabular}{ l l}
  \text{Aspects of Hamilton circuits}  & \text{Literature}  \\ \hline
\parbox{0.45\linewidth}{efficient algorithms for finding a copy} 
& \cite[Section~4]{MR0414429}, \cite{MR2510568} 
\\ 
  number of all copies & \cite{MR1969376}, \cite{MR2520274}, \cite{MR2520275} 
\\ 
number of mutually edge-disjoint copies  & \cite{MR0299517} \cite{MR0284366} 
\\ 
host graph is (in some sense) random & \cite{MR1943857} \cite{arXiv:1006.1268v1} \cite{arXiv:1101.3099v1} \cite{arXiv:1104.4412v1} \cite{arXiv:1108.2502v1} \cite{arXiv:1109.5341v1} \\  
linear algebraic properties  &  this paper \\ 
\end{tabular}
\caption{Some aspects of Hamilton circuits in graphs with high-minimum degree.}
\end{table}
\label{tab:aspectsofthesetofhamiltoncircuits}
\end{center}

\addtocounter{footnote}{-1}

\subsection*{Structure of the paper}\label{8i76tr565666g78g78}

There are four sections after the Introduction~\ref{9877667556545454546565}. 
In Section~\ref{section:presentationofmainresults} we develop a plan for proving 
Theorem~\ref{thm:mainresults}, in the process introducing all the auxiliary 
statements that we will later draw upon. 
In Section~\ref{677656565658899822222222222222221}, the plan 
is carried out in detail, in particular by giving proofs for all the auxiliary 
statements. Section~\ref{sec:alternativeargumentation} is logically superfluous 
but provides an alternative argumentation for a part of the proof 
of \ref{thm:minDegOneFourthPlusGammaImpliesHamiltonGeneratednessInBipartiteGraphs} in 
Theorem~\ref{thm:mainresults}. Section~\ref{sec:concludingremarks} surveys the 
literature relevant to Theorem~\ref{thm:mainresults} and mentions open problems. 

\section{Main results}\label{section:presentationofmainresults}

Let us first introduce terminology. We adopt the common convention that a 
$2$-set $\{v',v''\}$ can be abbreviated as $v'v''$. By `graph' we will mean 
`finite simple undirected graph', equivalently `$1$-dimensional simplicial complex'. 
If $X$ and $Y$ are graphs, then $Y\hookrightarrow X$ means that there exists an 
injective graph homomorphism $Y\to X$ (hence there is a subgraph of $X$ isomorphic 
to $Y$).  A path of \emph{length} (i.e. number of its edges) $\ell$ will be 
denoted by $\upP_\ell$ and a circuit of length $\ell$ by $\upC_\ell$. (As is done 
in e.g. \cite{MR0299517} and \cite{MR1373656} we reserve the word `cycle' for the 
homological meaning and use the more specific term `circuit' for 
`$2$-regular connected graph'.) For a graph $X$ we will write $\upV(X)$ for its 
vertex set, $\upE(X)$ for its edge set, $f_0(X):=\lvert \upV(X)\rvert$ and 
$f_1(X):=\lvert\upE(X)\rvert$. If $C$ is a circuit with 
$\upV(C) = \{v_0,v_1,v_2,\dotsc ,v_{\ell-1}\}$ and 
$\upE(C) = \{v_0v_1, v_1v_2, \dotsc, v_{\ell-1}v_0\}$, then we abbreviate 
$v_0v_1v_2\dotsc v_{\ell-1} v_0 := \upE(C)$. A subgraph $Y$ of a graph $X$ is 
called \emph{non-separating} if and only if the graph 
$X-Y:=(\upV(X)\setminus\upV(Y), \upE(X)\setminus \{e \in \upE(X)\colon 
e\cap\upV(Y)\neq\emptyset\} )$ is connected. A circuit $C$ in a graph $X$ is 
called \emph{non-separating induced} if and only if $C$ is non-separating and $C$ 
has no chords in $X$ (i.e. $\{ e\in\upE(X)\colon e\subseteq\upV(C)\} = \upE(C)$). 
We write $\Z/2:=\Z/2\Z$ and $\upc_e\in (\Z/2)^{\upE(X)}$ for the unique map 
with $\upc_e(e)=1\in\Z/2$ and $\upc_e(e')=0\in\Z/2$ for 
every $e\neq e'\in\upE(X)$. As usual, 
$\upC_1(X;\Z/2) := \bigwedge^2 \langle\upV(X)\rangle_{\Z/2}$ (second exterior 
power) denotes the $1$-dimensional chain group, where 
$\langle\upV(X)\rangle_{\Z/2}$ is the $\Z/2$-vector space freely generated 
by $\upV(X)$, and 
$\upZ_1(X;\Z/2):=\ker\left(\partial\colon\upC_1(X;\Z/2)\to\upC_0(X;\Z/2)\right)$ 
(standard boundary operator of simplicial homology theory) denotes the 
$1$-dimensional cycle group in the sense of simplicial homology with 
$\Z/2$-coefficients. This is the standard graph-theoretical \emph{cycle space} 
of a graph. It is a vector space over $\Z/2$ with 
$\dim_{\Z/2}\ \upZ_1(X;\Z/2) = f_1(X) - f_0(X) + 1 = \beta_1(X)$, 
the $1$-dimensional Betti number of $X$. Since $-1=1$ in $\Z/2$, 
in $\upC_1(X;\Z/2)$ we have $v_i\wedge v_j = v_j\wedge v_i$, hence we can 
write the standard basis of $\upC_1(X;\Z/2)$ as 
$\{$ $v_i\wedge v_j$ $\colon$ $v_iv_j\in \upE(X)$ $\}$, 
the latter notation being well-defined despite $v_iv_j = v_jv_i$. The 
notation $\mathcal{H}(X)$ denotes the set of all Hamilton 
circuits in $X$. For any set $\mathcal{M}$ of circuits 
in $X$ we say that `$\mathcal{M}$ generates $\upZ_1(X;\Z/2)$' if and only if 
$\{ \upc_C\colon C\in \mathcal{M}\}$ is a $\Z/2$-generating system of 
$\upZ_1(X;\Z/2)$, where $\upc_C$ is defined as the element of $\upC_1(X;\Z/2)$ 
having its support equal to $\upE(C)$. A bipartite graph is called \emph{square} 
if and only if its bipartition classes have equal size. If $X$ and $Y$ are 
graphs, we denote by $X \square Y$ the \emph{cartesian product} of $X$ and $Y$ 
(see e.g. \cite[Section~1.4]{MR1788124}). Moreover, if $X$ is a graph, then 
we write 
$\delta(X) := \min_{v\in\upV(X)} \lvert \upN_X(v)\rvert$ for the minimum degree, 
$\Delta(X) := \max_{v\in\upV(X)} \lvert \upN_X(v)\rvert$ for the maximum degree of $X$, 
and $\upN_X(v):=\{ w\in\upV(X)\colon \{v,w\}\in\upE(X)\}$ for every $v\in\upV(X)$. 
By \emph{$k$-connected} we mean the standard graph-theoretical notion of 
being `vertex-$k$-connected' (cf. e.g. \cite[Section~1.4]{MR2159259}).

\subsection{Plan of the proof of Theorem~\ref{thm:mainresults}}\label{8uy77y6676t6tg6tg6t6g7}

The proof of Theorem~\ref{thm:mainresults} will be broken into the following steps 
(the strategy is the same for 
\ref{thm:minDegOneHalfPlusGammaImpliesHamiltonGeneratedForOddOrder}--\ref{thm:minDegTwoThirds}, but the auxiliary spanning subgraphs used are different): 

\begin{enumerate}[label={\rm(St\arabic{*})}]
\item\label{it:proofstrategy:extremalgraphtheorystep} 
Prove the existence of suitably chosen spanning subgraphs $Y\hookrightarrow X$; 
for \ref{thm:minDegOneHalfPlusGammaImpliesHamiltonGeneratedForOddOrder} and 
\ref{thm:minDegOneHalfPlusGammaImpliesAllOneCanAskForWhenOrderIsEven} by using 
Theorem~\ref{thm:BoettcherSchachtTaraz2009}, for 
\ref{thm:minDegOneFourthPlusGammaImpliesHamiltonGeneratednessInBipartiteGraphs} 
by using Theorem~\ref{thm:BoettcherHeinigTaraz2010}, and for \ref{thm:minDegTwoThirds} by using Theorem~\ref{thm:KomlosSarkozySzemeredi1996} below. These 
graphs $Y$ serve as `scaffolds' in step \ref{it:proofstrategy:liftingstep} which 
help confer the desired properties to the ambient graph $X$. 
\item\label{it:proofstrategy:structuralgraphtheorystep} Prove that in each case 
the subgraph $Y$ itself has its cycle space generated by its Hamilton circuits, and 
moreover that $Y$ is Hamilton connected.\footnote{The weaker property 
`any two \emph{non-adjacent} vertices are connected by a Hamilton path' would suffice 
here, but we will work with the better-known property of being Hamilton-connected.}
\item\label{it:proofstrategy:liftingstep} By adapting a lemma of 
S.~C.~Locke~\cite[Lemma~1]{MR818599} argue that the properties proved 
in \ref{it:proofstrategy:structuralgraphtheorystep} transfer from the subgraph $Y$ 
to the ambient graph $X$, thereby proving Theorem~\ref{thm:mainresults}. 
\end{enumerate}
We now explain 
\ref{it:proofstrategy:extremalgraphtheorystep}---\ref{it:proofstrategy:liftingstep} 
in more detail. 

\subsubsection{Explanation of step \ref{it:proofstrategy:extremalgraphtheorystep}}\label{8iy8y877yh7y7y7y7h879}
The theorems mentioned in \ref{it:proofstrategy:extremalgraphtheorystep} are the 
following. As to terminology, the \emph{square} $Y^2$ of a graph $Y$ is the graph 
obtained from $Y$ by adding an edge between any two vertices having distance two 
in $Y$. A graph $Y$ has \emph{bandwidth at most $b$} if and only if there exists a 
bijection $\upb\colon \upV(Y) \to \{1,\dotsc,f_0(Y)\}$ such that 
if $vv'\in\upE(Y)$, then $\lvert \upb(v) - \upb(v')  \rvert \leq b$; any 
such bijection $\upb$ is called a \emph{bandwidth-$b$-labelling of $Y$}. 
Moreover, if $Y$ is a graph, $\upb\colon\upV(Y)\to\{1,\dotsc,f_0(Y)\}$ is a 
bijection and if $(c_1,c_2)\in \Z_{\geq 1}^2$ and $\rho\in\Z_{\geq 1}$, then a map 
$h\colon \upV(Y)\to\{0,\dotsc,\rho\}$ is called 
\emph{$(c_1,c_2)$-zero-free w.r.t. $\upb$} (cf.~\cite[p.~178]{MR2448444}) if and 
only if for for every $v'\in \upV(Y)$ there exists a 
$v''\in \upb^{-1}\bigl(\{ \upb(v'),\upb(v')+1,\dotsc,\min(f_0(Y),\upb(v')+c_1)\}\bigr)$ such that  $h(v''')\neq 0$ for every 
$v'''\in\upb^{-1}\bigl(\{\upb(v''),\upb(v'')+1,\dotsc,\min(f_0(Y),\upb(v'')+c_2)\}\bigr)$. As a tool for proving Theorem~\ref{thm:mainresults} we use:

\begin{theorem}[{B{\"o}ttcher--Schacht--Taraz~\cite[Theorem~2]{MR2448444}}]\label{thm:BoettcherSchachtTaraz2009}
For every $\gamma>0$ and arbitrary $\rho\in\Z_{\geq 2}$ and $\Delta\in\Z_{\geq 2}$ 
there exist numbers $\beta=\beta(\gamma,\Delta)>0$ and $n_0=n_0(\gamma,\Delta)$ 
such that the following is true: for every graph $X$ with $f_0(X) \geq n_0$ and 
$\delta(X) \geq (\tfrac{\rho-1}{\rho}+\gamma)f_0(X)$, and for every graph $Y$ having 
$f_0(X)=f_0(Y)$, $\Delta(Y)\leq \Delta$ and $\bw(Y) \leq \beta f_0(Y)$, and 
admitting a bandwidth-$\beta f_0(Y)$-labelling 
$\upb\colon\upV(Y)\to\{1,\dotsc,f_0(Y)\}$ and a $(\rho+1)$-colouring 
$h\colon\upV(Y)\to\{0,1,\dotsc,\rho\}$ which is 
$\bigl(8\rho\beta f_0(Y), 4\rho\beta f_0(Y)\bigr)$-zero-free w.r.t. $\upb$ and 
has $\lvert h^{-1}(0)\rvert \leq \beta f_0(Y)$, there is an 
embedding $Y\hookrightarrow X$. \hfill $\Box$
\end{theorem}

\begin{theorem}[{B{\"o}ttcher--Heinig--Taraz \cite[Theorem~3]{MR2735919}}]
\label{thm:BoettcherHeinigTaraz2010}
For every $\gamma > 0$ and every $\Delta\in\Z$ there exist numbers
$\beta=\beta(\gamma,\Delta) > 0$ and $n_0=n_0(\gamma,\Delta)\in \Z$ such 
that the following is true: for every square bipartite graph $X$ with 
$f_0(X) \geq n_0$ and $\delta(X) \geq (\frac14 + \gamma) f_0(X)$, and for 
every square bipartite graph $Y$ with $f_0(Y) = f_0(X)$, $\Delta(Y)\leq \Delta$ 
and $\bw(Y)\leq \beta f_0(Y)$, there is an embedding 
$Y\hookrightarrow X$.\hfill $\Box$
\end{theorem}

Moreover, the lower bound of terrestrial magnitude that is provided 
in \ref{thm:minDegTwoThirds} depends on a very recent theorem of 
P.~Ch{\^a}u, L.~DeBiasio and H.~A.~Kierstead (who say 
\cite[p.~17, Section~5, l.~5]{CHAUDEBIASIOKIERSTEAD} that by optimizing their proof 
one may not push the bound further down than to about $n_0 = 10^5$, but who 
nevertheless express optimism as to the possibility of getting rid of the 
$f_0$-condition altogether by some new graph-theoretical methods): 

\begin{theorem}[{Koml{\'o}s--S\'{a}rk{\"o}zy--Szemer{\'e}di \cite[Theorem~1]{MR1611764},  Jamshed \cite[Chapter 3]{JAMSHEDembeddingspanningsubgraphsintolargedensegraphs}; explicit lower bound on $f_0$ proved by Ch{\^a}u--DeBiasio--Kierstead \cite[Theorem~7]{CHAUDEBIASIOKIERSTEAD}}]\label{thm:KomlosSarkozySzemeredi1996}
For every graph $X$ with $f_0(X) \geq 2 \cdot 10^8$ and 
$\delta(X)\geq \frac23 f_0(X)$ there exists an embedding 
$\upC_{f_0(X)}^2\hookrightarrow X$.\hfill $\Box$
\end{theorem}

Whereas for \ref{thm:minDegTwoThirds} our use of 
Theorem~\ref{thm:KomlosSarkozySzemeredi1996} dictates employing $\upC_{f_0(\cdot)}^2$ 
as the auxiliary subgraph, there are choices to be made as to what subgraph to 
employ from the set of spanning subgraphs offered by the 
Theorems~\ref{thm:BoettcherSchachtTaraz2009} and \ref{thm:BoettcherHeinigTaraz2010}. 
We will choose to use the following graphs (in Definition~\ref{def:bipCyclicLadder} 
let $b_r:=b_0$):  

\begin{definition}[Bipartite cyclic ladder]\label{def:bipCyclicLadder}
For $r\in \Z_{\geq 3}$ let $\CL_r$ be the bipartite graph 
with $\upV(\CL_r)$ $:=$ $\{a_0,\dotsc,a_{r-1}\} \sqcup \{b_0,\dotsc,b_{r-1}\}$ 
and $\upE(\CL_r)$ $:=$ $\bigsqcup_{i = 0}^{r-1}$  $\{ a_i b_{i-1} \}$ $\sqcup$ 
$\bigsqcup_{i = 0}^{r-1}$ $\{ a_i b_i \}$ $\sqcup$ 
$\bigsqcup_{i = 0}^{r-1}$ $\{ a_ib_{i+1} \}$. 
\end{definition}

\begin{definition}[{prism, M{\"o}bius ladder}]\label{def:squareofcircuitprismmobiusladder} 
For every $n\geq 3$ and $r\geq 3$ let (where $v_n:=v_0$, $x_r:=x_0$ and $y_r:=y_0$)
the \emph{prism} $\Pr_r$ be defined by 
$\upV(\Pr_r)$ $:=$ $\{  x_0, \dotsc, x_{r-1}, y_0, \dotsc, y_{r-1}\}$ and
$\upE(\Pr_r)$ $:=$ $\bigsqcup_{i=0}^{r-1}$ $\{$ $x_i x_{i+1}$ $\}$ 
$\sqcup$ $\bigsqcup_{i=0}^{r-1}$ $\{$  $y_iy_{i+1}$ $\}$ $\sqcup$ 
$\bigsqcup_{i=0}^{r-1}$ $\{$  $x_iy_i$ $\}$, and the \emph{M{\"o}bius ladder} 
$\upM_r$ be defined by $\upV(\upM_r)$ $:=$ $\upV(\Pr_r)$ and 
$\upE(\upM_r)$ $:=$ $\bigl($ $\upE(\Pr_r)$ $\setminus$ 
$\{$ $x_0x_{r-1},$ $y_0y_{r-1}$ $\}$ $\bigr)$ $\sqcup$ 
$\{$ $x_0y_{r-1},$ $y_0x_{r-1}$ $\}$.
\end{definition}

\begin{definition}[{$\Pr_r^{\boxtimes}$ and $\upM_r^{\boxtimes}$}]
\label{8767856456576878787}
For every $r\geq 3$ let $\Pr_r^{\boxtimes}$ be defined by 
$\upV(\Pr_r^{\boxtimes}) := \upV(\Pr_r) \sqcup \{z\}$,  with $z$ some new element, 
and $\upE(\Pr_r^{\boxtimes})$ $:=$ $\upE(\Pr_r)$ $\sqcup$ 
$\{$ $zx_0,$ $zy_0,$ $zx_1,$ $zy_1$ $\}$. Let $\upM_r^{\boxtimes}$ be defined by 
$\upV(\upM_r^{\boxtimes}) := \upV(\Pr_r^{\boxtimes})$ and $\upE(\upM_r^{\boxtimes})$ $:=$ 
$($ $\upE(\Pr_r^{\boxtimes})$ $\setminus$ $\{$ $x_0x_{r-1},$ $y_0y_{r-1}$ $\}$ $)$ 
$\sqcup$ $\{$ $x_0y_{r-1},$ $y_0x_{r-1}$ $\}$. 
\end{definition}

\begin{definition}[{$\Pr_r^{\boxminus}$ and $\upM_r^{\boxminus}$}]
\label{463671341124867316783467813}
For every $r\geq 3$ let $\Pr_r^{\boxminus}$ be defined by 
$\upV(\Pr_r^{\boxminus}) := \upV(\Pr_r)\sqcup \{z',z''\}$ with $z'$ and $z''$ two new 
elements, $\upE(\Pr_r^{\boxminus})$ $:=$ $\upE(\Pr_r)$ $\sqcup$ 
$\{$ $x_0 z',$ $y_0 z',$ $x_0 z'',$ $x_1 z'',$ $y_1 z'',$ $z' z''$ $\}$. Let 
$\upM_r^{\boxminus}$ be defined by $\upV(\upM_r^{\boxminus}) := \upV(\Pr_r^{\boxminus})$ 
and $\upE(\upM_r^{\boxminus})$ $:=$ $\left(\upE(\Pr_r^{\boxminus})\setminus
\{x_0x_{r-1},y_0y_{r-1}\}\right)\sqcup\{x_0y_{r-1},y_0x_{r-1}\}$. 
\end{definition}

\begin{figure} 
\begin{center}
\input{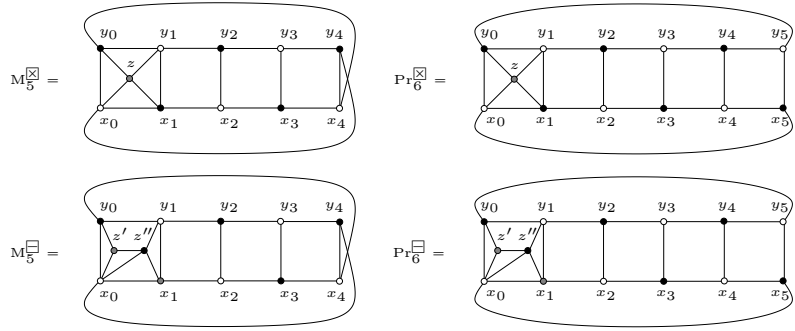}
\caption{The graphs $\upM_r^{\boxtimes}$ and $\upM_r^{\boxminus}$ for odd $r$, 
and $\Pr_r^{\boxtimes}$ and $\Pr_r^{\boxminus}$ for even $r$ play a key role in the 
proof. These are bounded-degree, bounded-bandwidth and $3$-chromatic graphs 
admitting a $3$-colouring with a constant-sized third colour class. 
The B{\"o}ttcher--Schacht--Taraz-theorem in its full form 
\cite[Theorem~2]{MR2448444} is sufficiently general to guarantee the existence of 
embeddings of these graphs as \emph{spanning} subgraphs into graphs $X$ with 
$\delta(X)\geq (\tfrac12+\gamma)f_0(X)$. If $\upM_r^{\boxtimes}$ or $\Pr_r^{\boxtimes}$ 
spanningly embed into $X$, this implies that $\upZ_1(X;\Z/2)$ is generated by 
Hamilton circuits. If $\upM_r^{\boxminus}$ or $\Pr_r^{\boxminus}$ spanningly embed 
into $X$ this implies that $\upZ_1(X;\Z/2)$ is generated by the circuits having 
lengths in $\{f_0(X)-1,f_0(X)\}$. If the edge $x_0 z''$ were omitted from 
$\upM_r^{\boxminus}$ or $\Pr_r^{\boxminus}$, the remaining graph could no longer serve 
the purpose these graphs have in the present paper.}
\label{u7yt6r54e4e4ttf5drdr4rdxrdx}
\end{center}
\end{figure}

Justifying that $\CL_r$ is indeed one of the subgraphs guaranteed by 
Theorem~\ref{thm:BoettcherHeinigTaraz2010} will pose no difficulty and can be done 
uniformly for every $r\in\Z_{\geq 3}$. Matters are being complicated by parity 
issues when it comes to step \ref{it:proofstrategy:structuralgraphtheorystep}. We 
will later make essential use of the following sets. 

\begin{definition}\label{90987564224567898787878777899831323}
For every even $r\geq 4$ we define the sets of edge sets 
{\small 
\begin{enumerate}[label={\rm(P.$\boxtimes$.ES.\arabic{*})},leftmargin=7em] 
\item\label{425486446546434454567565454548} 
$\mathcal{CB}_{\Pr_r^{\boxtimes}}^{(1)}$ $:=$ 
$\left\{ 
\begin{matrix} 
C_{\ev,r,1} & := 
& z y_1x_1x_2y_2y_3\dotsc x_{r-2}y_{r-2}y_{r-1}x_{r-1}x_0y_0 z\quad ,\\
C_{\ev,r,2} & := 
& z x_1x_2 y_2y_3\dotsc  x_{r-2}y_{r-2}y_{r-1}x_{r-1}x_0y_0y_1 z\quad ,\\
C_{\ev,r,3} & := 
& z x_1y_1y_2x_2x_3\dotsc x_{r-2}x_{r-1}y_{r-1}y_0x_0 z\quad ,\\
C_{\ev,r,4} & := 
& z x_0x_1y_1y_2\dotsc y_{r-3}y_{r-2}x_{r-2}x_{r-1}y_{r-1}y_0 z\quad ,\\
C_{\ev,r,5} & := 
& z y_1y_2x_2x_3\dotsc x_{r-2}x_{r-1}y_{r-1}y_0x_0x_1 z\quad 
\end{matrix}\right\}$\quad ,
\item\label{867554134354675546485657563365} 
$\mathcal{CB}_{\Pr_r^{\boxtimes}}^{(2)}$ $:=$ 
$\left\{ \begin{matrix} 
C_{\ev,r}^{x_1y_1} & := 
& z x_0x_{r-1}x_{r-2}\dotsc x_2y_2y_3\dotsc y_{r-1}y_0y_1x_1 z\quad ,\\
C_{\ev,r}^{x_2y_2} & := 
& z x_0x_{r-1}x_{r-2}\dotsc x_3y_3y_4\dotsc y_{r-1}y_0y_1y_2x_2x_1 z\quad ,\\
& \vdots &  \\
C_{\ev,r}^{x_{r-2}y_{r-2}} & := 
& z x_0x_{r-1}y_{r-1}y_0y_1\dotsc y_{r-2}x_{r-2}x_{r-3}\dotsc x_1 z\quad ,\\
C_{\ev,r}^{x_{r-1}y_{r-1}} & := 
& z x_0x_1\dotsc x_{r-1}y_{r-1}y_{r-2}\dotsc y_0 z\quad 
\end{matrix}\right\}$ \quad .
\end{enumerate}
}
\end{definition}

The set $C_{\ev,r}^{x_{r-1}y_{r-1}}$ does not follow the pattern to be found in 
$C_{\ev,r}^{x_1y_1}$, $\dotsc$, $C_{\ev,r}^{x_{r-2}y_{r-2}}$. 

\begin{definition}\label{897564676522229897678687876876}
For every odd $r\geq 5$ we define the sets of edge sets 
{\small 
\begin{enumerate}[label={\rm(M.$\boxtimes$.ES.\arabic{*})},leftmargin=7em] 
\item\label{89845242425699876554213343443434} 
$\mathcal{CB}_{\upM_r^{\boxtimes}}^{(1)}$ $:=$ 
$\left\{ 
\begin{matrix} 
C_{\od,r,1} & := 
& z y_1x_1x_2y_2y_3\dotsc y_{r-2}x_{r-2}x_{r-1}y_{r-1}x_0y_0 z\quad ,\\
C_{\od,r,2} & := 
& z x_1x_2 y_2y_3\dotsc y_{r-2}x_{r-2}x_{r-1}y_{r-1}x_0y_0y_1 z\quad ,\\
C_{\od,r,3} & := 
& z x_1y_1y_2x_2x_3\dotsc y_{r-2}y_{r-1}x_{r-1}y_0x_0 z\quad ,\\
C_{\od,r,4} & := 
& z x_0x_1y_1y_2\dotsc x_{r-3}x_{r-2}y_{r-2}y_{r-1}x_{r-1}y_0 z\quad ,\\
C_{\od,r,5} & := 
& z y_1y_2x_2x_3\dotsc y_{r-2}y_{r-1}x_{r-1}y_0x_0x_1 z\quad 
\end{matrix}\right\}$\quad ,
\item\label{65534545456766765656554008976756565} 
$\mathcal{CB}_{\upM_r^{\boxtimes}}^{(2)}$ $:=$ 
$\left\{ 
\begin{matrix} 
C_{\od,r}^{x_1y_1} & := 
& z x_0y_{r-1}y_{r-2}\dotsc y_2x_2x_3\dotsc x_{r-1}y_0y_1x_1 z\quad ,\\
C_{\od,r}^{x_2y_2} & := 
& z x_0y_{r-1}y_{r-2}\dotsc y_3x_3x_4\dotsc x_{r-1}y_0y_1y_2x_2x_1 z\quad ,\\
& \vdots &  \\
C_{\od,r}^{x_{r-2}y_{r-2}} & := 
& z x_0y_{r-1}x_{r-1}y_0y_1\dotsc y_{r-2}x_{r-2}x_{r-3}\dotsc x_1 z\quad ,\\
C_{\od,r}^{x_{r-1}y_{r-1}} & := 
& z x_0x_1\dotsc x_{r-1}y_{r-1}y_{r-2}\dotsc y_0 z\quad 
\end{matrix}\right\}$ \quad .
\end{enumerate}
}
\end{definition}

Again, the set $C_{\od,r}^{x_{r-1}y_{r-1}}$ does not conform to the pattern to be found in  
$C_{\od,r}^{x_1y_1}$, $\dotsc$, $C_{\od,r}^{x_{r-2}y_{r-2}}$. 

\begin{definition}\label{988537675343454677887987987987}
For every even $r\geq 4$ we define the sets of edge sets 
{\small
\begin{enumerate}[label={\rm(P.$\boxminus$.ES.\arabic{*})},leftmargin=7em] 
\item\label{9899765435767548946635545443} 
$\mathcal{CB}_{\Pr_r^{\boxminus}}^{(1)}$ $:=$
$\left\{\begin{matrix}
C_{\boxminus,\ev,r,1} & := & 
z'x_0z''x_1x_2\dotsc x_{r-1}y_{r-1}y_{r-2}\dotsc y_0z' \quad ,\\
C_{\boxminus,\ev,r,2} & := & 
z'z''x_0x_{r-1}x_{r-2}\dotsc x_1y_1y_2\dotsc y_{r-1}y_0z'\quad ,\\
C_{\boxminus,\ev,r,3} & := & 
z'x_0z''x_1y_1y_2x_2x_3\dotsc x_{r-2}x_{r-1}y_{r-1}y_0z'\quad ,\\
C_{\boxminus,\ev,r,4} & := & 
z'z''x_1x_2\dotsc x_{r-1}y_{r-1}y_{r-2}\dotsc y_0x_0z'\quad ,\\
C_{\boxminus,\ev,r,5} & := & 
z'x_0x_{r-1}y_{r-1}y_{r-2}x_{r-2}x_{r-3}\dotsc x_2x_1z''y_1y_0z'
\end{matrix}\right\}$\quad ,
\item\label{76645565667786786676655009877} 
$\mathcal{CB}_{\Pr_r^{\boxminus}}^{(2)}$ $:=$ 
$\left\{ \begin{matrix} 
C_{\boxminus,\ev,r}^{x_1y_1} & := & 
z'x_0x_{r-1}x_{r-2}\dotsc x_2y_2y_3\dotsc y_{r-1}y_0y_1x_1z'' z'\quad ,\\
C_{\boxminus,\ev,r}^{x_2y_2} & := & 
z'x_0x_{r-1}x_{r-2}\dotsc x_3y_3y_4\dotsc y_{r-1}y_0y_1y_2x_2x_1z''z'\quad ,\\
& \vdots &  \\
C_{\boxminus,\ev,r}^{x_{r-2}y_{r-2}} & := & 
z'x_0x_{r-1}y_{r-1}y_0y_1\dotsc y_{r-2}x_{r-2}x_{r-3}\dotsc x_1z''z'\quad ,\\
C_{\boxminus,\ev,r}^{x_{r-1}y_{r-1}} & := & 
z'z''x_0x_1\dotsc x_{r-1}y_{r-1}y_{r-2}\dotsc y_0z'\quad 
\end{matrix}\right\}$ \quad .
\end{enumerate}
}
\end{definition}

\begin{definition}\label{987876676678787877676509876543}
For every odd $r\geq 5$ we define the sets of edge sets 
{\small
\begin{enumerate}[label={\rm(M.$\boxminus$.ES.\arabic{*})},leftmargin=7em] 
\item\label{98187655434547887654435465778977665544334} 
$\mathcal{CB}_{\upM_r^{\boxminus}}^{(1)}$ $:=$
$\left\{\begin{matrix}
C_{\boxminus,\od,r,1} & := & z'x_0z''x_1x_2\dotsc x_{r-1}y_{r-1}y_{r-2}\dotsc y_0 z'\  = \ 
C_{\boxminus,\ev,r,1} \quad ,\\
C_{\boxminus,\od,r,2} & := & 
z'z''x_0y_{r-1}y_{r-2}\dotsc y_1x_1x_2\dotsc x_{r-1}x_0z' \quad ,\\
C_{\boxminus,\od,r,3} & := & 
z'x_0z''x_1y_1y_2x_2x_3y_3\dotsc y_{r-2}y_{r-1}x_{r-1}x_0z' \quad ,\\
C_{\boxminus,\od,r,4} & := & 
z'z''x_1x_2\dotsc x_{r-1}y_{r-1}y_{r-2}\dotsc y_0x_0z'\ =\ C_{\boxminus,\ev,r,4} \quad ,\\
C_{\boxminus,\od,r,5} & := & z'x_0y_{r-1}x_{r-1}x_{r-2}y_{r-2}y_{r-3}\dotsc x_2x_1z''y_1y_0z'
\end{matrix}\right\}$\quad ,
\item\label{9198786653454809958487332626263} 
$\mathcal{CB}_{\upM_r^{\boxminus}}^{(2)}$ $:=$ $\left\{ \begin{matrix} 
C_{\boxminus,\od,r}^{x_1y_1} & := & 
z'x_0y_{r-1}y_{r-2}\dotsc y_2x_2x_3\dotsc x_{r-1}y_0y_1x_1z'' z'\quad ,\\
C_{\boxminus,\od,r}^{x_2y_2} & := & 
z'x_0y_{r-1}y_{r-2}\dotsc y_3x_3x_4\dotsc x_{r-1}y_0y_1y_2x_2x_1z'' z'\quad ,\\
& \vdots &  \\ 
C_{\boxminus,\od,r}^{x_{r-2}y_{r-2}} & := & 
z'x_0y_{r-1}x_{r-1}y_0y_1\dotsc y_{r-2}x_{r-2}x_{r-3}\dotsc x_1z''z' \quad ,\\
C_{\boxminus,\od,r}^{x_{r-1}y_{r-1}} & := & 
z'z''x_0x_1\dotsc x_{r-1}y_{r-1}y_{r-2}\dotsc y_0z'\ =\ C_{\boxminus,\ev,r}^{x_{r-1}y_{r-1}}\quad 
\end{matrix}\right\}$ \quad  .
\end{enumerate}
}
\end{definition}

\subsubsection{Explanation of step \ref{it:proofstrategy:structuralgraphtheorystep}}\label{subsubsection:explanationofthestepdealingwiththeauxiliarystructures}

If $G$ is a finite abelian group in additive notation, and $0\notin S\subseteq G$ 
has the property that $-S:=\{ -s\colon s\in S\} = S$, then we write 
$\langle S \rangle := \sum_{s\in S} \Z s$ for the abelian group generated by $S$ and 
define a graph $X:=\Cay(\langle S \rangle ; S)$ by $\upV(X):=\langle S\rangle$ and 
$\{a,b\}\in \upE(X)$ $:\Leftrightarrow$ $a-b\in S$, called the \emph{Cayley graph} 
associated to $G$ and $S$. The following theorem of C.~C.~Chen and N.~F.~Quimpo has 
proved to be fertile for the theory of Cayley graphs on finite abelian groups: 

\begin{theorem}[{Chen--Quimpo; \cite[Theorem~4]{MR641233} gives the 
non-bipartite case.\footnote{The bipartite case appears to be susceptible to analogous 
arguments as in \cite{MR641233}. The author does not know of any published proof 
of the bipartite case. Nevertheless, it is mentioned 
in \cite[Theorem~1.4]{MR1848324}, \cite[Theorem~1.7]{MR2592513}, 
\cite[Introductory Remarks and Proposition~2.1]{MIKLAVICSPARLhamiltoncycleandhamiltonpathextendability2011} and \cite[Proposition~3]{MIKLAVICSPARLonextendabilityofcayleygraphs2009}. Moreover, what little we need of the general bipartite case, namely 
Lemma~\ref{lem:generatingpropertiesofauxiliarestructures}.\ref{lem:it:bipartitecyclicladder:hamiltongenerated}, can be easily shown directly.}}]\label{thm:ChenQuimpo} 
For every finite abelian group $G$ and every $S\subseteq G$ with $-S = S$ and 
$\lvert S \rvert \geq 3$ the graph $X = \Cay(\langle S \rangle;S)$ is 
Hamilton-connected in case $X$ is not bipartite, and Hamilton-laceable 
in case $X$ is bipartite.\hfill $\Box$
\end{theorem}

We will use the following theorem of B.~Alspach, S.~C.~Locke and D.~Witte which 
appears to be the first result in the literature dealing with linear algebraic 
properties of Hamilton circuits (as to terminology, a graph $X$ is called a 
\emph{prism over the graph $Y$} if and only if $X\cong Y\square \upP_1$): 

\begin{theorem}[{Alspach--Locke--Witte \cite[Theorem~2.1 and Corollary~2.3]{MR1057481}}]\label{thm:AlspachLockeWitte} 
For every finite abelian group $G$ and every $0\notin S \subseteq G$ with $-S=S$ 
the graph $X:=\Cay(\langle S \rangle ; S)$ has the following properties: 
\begin{enumerate}[label={\rm(\arabic{*})}]
\item\label{thm:alspachlockewitte:cayleygraphbipartite} 
if $X$ is bipartite, then $\mathcal{H}(X)$ generates $\upZ_1(X;\Z/2)$ \quad ,
\item\label{thm:alspachlockewitte:numberofverticesodd} 
if $\lvert X \rvert = \lvert \langle S \rangle \rvert$ is odd, 
then $\mathcal{H}(X)$ generates $\upZ_1(X;\Z/2)$ \quad ,
\item\label{thm:alspachlockewitte:whenthecodimensionisone} 
if $\lvert X \rvert = \lvert \langle S \rangle \rvert$ is even and $X$ is not 
bipartite and not a prism over any circuit of odd length, then 
$\dim_{\Z/2}\bigl ( \upZ_1(X;\Z/2) / \langle\mathcal{H}(X)\rangle_{\Z/2}\bigr ) = 1$ 
\quad .\hfill $\Box$
\end{enumerate}
\end{theorem}

To efficiently formulate properties of the auxiliary substructures, we have to 
agree upon some further terminology:

\begin{definition}\label{def:monotonegraphpropertiesGlzetaandBGlzeta}
Let $\mathfrak{L}$ be a map from graphs to subsets of $\Z_{\geq 1}$, let 
$\mathfrak{L}-1 := \{ l-1\colon l\in\mathfrak{L} \}$ and let $\xi\in\Z_{\geq 0}$. 
We define 
\begin{enumerate}[label={\rm(\arabic{*})}]
\item\label{it:def:Lconnected} 
a graph $X$ to be \emph{$\mathfrak{L}$-path-connected} (if 
$\mathfrak{L}=\{ f_0(\cdot)-1 \}$ we speak of being \emph{Hamilton-connected}) if 
and only if for every $\{v,w\}\in\binom{\upV(X)}{2}$ there exists in $X$ at least 
one $v$-$w$-path having its length in the set $\mathfrak{L}(X)$ (we denote the 
collection of all such graphs by $\mathcal{CO}_{\mathfrak{L}}$) \quad ,
\item\label{it:def:Llaceable}
a variant of $\mathcal{CO}_{\mathfrak{L}}$ for bipartite graphs: adopting a by 
now widespread usage dating back at least to work of G.~J.~Simmons \cite{MR527987}, 
a bipartite graph $X$ will be called \emph{$\mathfrak{L}$-laceable} 
(if $\mathfrak{L}=\{f_0(\cdot)-1\}$ also \emph{Hamilton-laceable}) if and only if 
for any two $v,w\in \upV(X)$ not in the same bipartition class there exists at 
least one $v$-$w$-path having its length in the set $\mathfrak{L}(X)$ (we denote 
the collection of all such graphs by $\mathcal{LA}_{\mathfrak{L}}$) \quad ,
\item\label{it:def:Lgenerated} 
for a graph $X$ the set $\mathcal{C}_{\mathfrak{L}}(X)$ as the set of all 
graph-theoretical circuits in $X$ whose length is an element 
of $\mathfrak{L}$. (In particular, $\mathcal{C}_{\{ f_0(X) \}}(X) = \mathcal{H}(X)$.)
\item\label{it:def:CDxi} 
$\mathrm{cd}_{\xi}\mathcal{C}_{\mathfrak{L}}$ as the collection of graphs $X$ with 
$\dim_{\Z/2}\bigl ( \langle \mathcal{C}_{\mathfrak{L}}(X) \rangle_{\Z/2} \bigr ) = 
\beta_1(X) - \xi$ \quad ,
\item\label{it:def:bCDxi} 
$\mathrm{b}\mathrm{cd}_{\xi}\mathcal{C}_{\mathfrak{L}}\subseteq 
\mathrm{cd}_{\xi}\mathcal{C}_{\mathfrak{L}}$ as the collection of all the 
\emph{bipartite} elements of $\mathrm{cd}_{\xi}\mathcal{C}_{\mathfrak{L}}$ \quad ,
\item\label{it:def:monotonegraphpropertiesGlzetaandBGlzeta:Glzeta} 
$\mathcal{M}_{\mathfrak{L},\xi} := \mathrm{cd}_{\xi}\mathcal{C}_{\mathfrak{L}} \cap 
\mathcal{CO}_{\mathfrak{L}-1}$ and $\mathrm{b}\mathcal{M}_{\mathfrak{L},\xi} := 
\mathrm{b}\mathrm{cd}_{\xi}\mathcal{C}_{\mathfrak{L}}\cap\mathcal{LA}_{\mathfrak{L}-1}$\quad .
\end{enumerate}
\end{definition}

The condition in \ref{it:def:CDxi} is equivalent to 
$\dim_{\Z/2}\bigl( \upZ_1(X;\Z/2)/\langle \mathcal{C}_{\mathfrak{L}(X)}(X) \rangle_{\Z/2} \bigr ) = \xi$, in other words, $\mathrm{cd}_{\xi}\mathcal{C}_{\mathfrak{L}}(X)$ is the 
set of all graphs for which $\langle \mathcal{C}_{\mathfrak{L}(X)}(X) \rangle_{\Z/2}$ 
has codimension $\xi$ in $\upZ_1(X;\Z/2)$. In particular 
$\mathrm{cd}_{0}\mathcal{C}_{\{f_0(\cdot)\}}(X)$ is the set of all graphs whose cycle 
space is generated by the set of their Hamilton circuits. We will now formulate all 
the properties of the auxiliary spanning substructures that we use in the proof:

\begin{lemma}[{properties of the auxiliary structures}]
\label{lem:generatingpropertiesofauxiliarestructures}
For every $n\geq 5$ and every $r\in\Z_{\geq 4}$, 

{\small
\begin{minipage}[b]{0.9\linewidth}\begin{enumerate}[label={\rm(a\arabic{*})},start=1]
\item\label{lem:it:squareofhamiltoncircuitascayleygraph} 
$\upC_n^2$ $\cong$ 
$\Cay(\Z/n;\{\overline{1},\overline{2},\overline{n-2},\overline{n-1}\})$\; ,
\end{enumerate}\end{minipage}

\begin{minipage}[b]{0.9\linewidth}\begin{enumerate}[label={\rm(a\arabic{*})},start=2]
\item\label{lem:it:squareofhamiltoncircuit:notaprism} $\upC_n^2$ is not a prism 
over a graph (i.e. there does not exist a graph $Y$ 
with $\upC_n^2 \cong Y\square \upP_1$)\; ,
\end{enumerate}\end{minipage}

\begin{minipage}[b]{0.45\linewidth}\begin{enumerate}[label={\rm(a\arabic{*})},start=3]
\item\label{876676545545655rr565667}
if $n$ is even, then $\upC_n^2\in\mathcal{M}_{\{f_0(\cdot)\},1}$\ ,
\end{enumerate}\end{minipage}
\begin{minipage}[b]{0.5\linewidth}\begin{enumerate}[label={\rm(a\arabic{*})},start=4]
\item\label{6t6t655r5r55r56656565rtrt}
if $n$ is odd, then $\upC_n^2\in\mathcal{M}_{\{f_0(\cdot)\},0}$\ ,
\end{enumerate}\end{minipage}

\begin{minipage}[b]{0.6\linewidth}\begin{enumerate}[label={\rm(a\arabic{*})},start=5]
\item\label{898u889y8778y6t5r5454e446446} 
if $n$ is even, then $\upC_n^2\in\mathcal{M}_{\{f_0(\cdot)-1,f_0(\cdot)\},0}$\quad ,  
\end{enumerate}\end{minipage}

\begin{minipage}[b]{0.53\linewidth}\begin{enumerate}[label={\rm(a\arabic{*})},start=6]
\item\label{lem:it:PrIsoCayC2Cr} 
$\Pr_r \cong \Cay(\Z/2\oplus\Z/r ; \{ (\overline{1},\overline{0}), 
(\overline{0},\overline{1}), (\overline{0},\overline{r-1})\})$ \ ,
\end{enumerate}\end{minipage}
\begin{minipage}[b]{0.42\linewidth}\begin{enumerate}[label={\rm(a\arabic{*})},start=7]
\item\label{lem:it:MrIsoCayC2r} 
$\upM_r \cong \Cay(\Z/(2r); \{\overline{1},\overline{r},\overline{2r-1}\})$ \ ,
\end{enumerate}\end{minipage}

\begin{minipage}[b]{0.45\linewidth}\begin{enumerate}[label={\rm(a\arabic{*})},start=8]
\item\label{lem:it:prism:hamiltonlaceable} 
if $r$ is even, then $\Pr_r\in\mathcal{LA}_{f_0(\cdot)-1}$ \; ,
\end{enumerate}\end{minipage}
\begin{minipage}[b]{0.5\linewidth}\begin{enumerate}[label={\rm(a\arabic{*})},start=9]
\item\label{lem:it:moebiusladder:hamiltonlaceable} 
if $r$ is odd, then $\upM_r\in\mathcal{LA}_{f_0(\cdot)-1}$ \; ,  
\end{enumerate}\end{minipage}

\begin{minipage}[b]{0.45\linewidth}\begin{enumerate}[label={\rm(a\arabic{*})},start=10]
\item\label{lem:it:prism:generator} 
if $r$ is even, then $\Pr_r\in\mathrm{b}\mathcal{M}_{\{f_0(\cdot)\},0}$ \; ,
\end{enumerate}\end{minipage}
\begin{minipage}[b]{0.5\linewidth}\begin{enumerate}[label={\rm(a\arabic{*})},start=11]
\item\label{lem:it:moebiusladder:generator} 
if $r$ is odd, then $\upM_r\in\mathrm{b}\mathcal{M}_{\{f_0(\cdot)\},0}$ \; ,  
\end{enumerate}\end{minipage}

\begin{minipage}[b]{0.45\linewidth}\begin{enumerate}[label={\rm(a\arabic{*})},start=12]
\item\label{lem:it:relationofBCrtoPr} 
if $r$ is even, then $\CL_r\cong \Pr_r$ \; , 
\end{enumerate}\end{minipage}
\begin{minipage}[b]{0.5\linewidth}\begin{enumerate}[label={\rm(a\arabic{*})},start=13]
\item\label{lem:it:relationofBCrtoMr}
if $r$ is odd, then $\CL_r\cong \upM_r$ \; , 
\end{enumerate}\end{minipage}

\begin{minipage}[b]{0.45\linewidth}\begin{enumerate}[label={\rm(a\arabic{*})},start=14]
\item\label{lem:it:bipartitecyclicladder:hamiltongenerated} 
$\CL_r\in\mathcal{LA}_{f_0(\cdot)-1}$ \; ,  
\end{enumerate}\end{minipage}
\begin{minipage}[b]{0.45\linewidth}\begin{enumerate}[label={\rm(a\arabic{*})},start=15]
\item\label{lem:it:bipartitecyclicladder:generator} 
$\CL_r\in\mathrm{b}\mathcal{M}_{\{f_0(\cdot)\},0}$ \; ,  
\end{enumerate}\end{minipage}

\begin{minipage}[b]{0.45\linewidth}\begin{enumerate}[label={\rm(a\arabic{*})},start=16]
\item\label{t6t67y8u899o8987g6g678jik90o}
if $r$ is even, then $\Pr_r^{\boxtimes}\in\mathcal{CO}_{\{f_0(\cdot)-1\}}$\; ,
\end{enumerate}\end{minipage}
\begin{minipage}[b]{0.5\linewidth}\begin{enumerate}[label={\rm(a\arabic{*})},start=17]
\item\label{89878675643545465r5556}
if $r$ is odd, then $\upM_r^{\boxtimes}\in\mathcal{CO}_{\{f_0(\cdot)-1\}}$\; ,
\end{enumerate}\end{minipage}

\begin{minipage}[b]{0.45\linewidth}\begin{enumerate}[label={\rm(a\arabic{*})},start=18]
\item\label{32671163432867327682}
if $r$ is even, then $\Pr_r^{\boxminus}\in\mathcal{CO}_{\{f_0(\cdot)-1\}}$\; ,
\end{enumerate}\end{minipage}
\begin{minipage}[b]{0.5\linewidth}\begin{enumerate}[label={\rm(a\arabic{*})},start=19]
\item\label{0o978986765565454454e}
if $r$ is odd, then $\upM_r^{\boxminus}\in\mathcal{CO}_{\{f_0(\cdot)-1\}}$\; ,
\end{enumerate}\end{minipage}

\begin{minipage}[b]{0.9\linewidth}\begin{enumerate}[label={\rm(a\arabic{*})},start=20]
\item\label{38794320986754435465687}
concerning $\Pr_r^{\boxtimes}$ and $\Pr_r^{\boxminus}$ for even $r$, and 
concerning $\upM_r^{\boxtimes}$ and $\upM_r^{\boxminus}$ for odd $r$, the 
set $\{\upc_{C}\colon C\in \mathcal{CB}_{X}^{(1)}\}$ is a linearly independent subset 
of $\upZ_1(X;\Z/2)$ for all 
$X\in \{ \Pr_r^{\boxtimes}, \Pr_r^{\boxminus}, \upM_r^{\boxtimes}, \upM_r^{\boxminus}\}$\; ,
\end{enumerate}\end{minipage}

\begin{minipage}[b]{0.9\linewidth}\begin{enumerate}[label={\rm(a\arabic{*})},start=21]
\item\label{98786543322154675878987}
Concerning $\Pr_r^{\boxtimes}$ and $\Pr_r^{\boxminus}$ for even $r$, and 
concerning $\upM_r^{\boxtimes}$ and $\upM_r^{\boxminus}$ for odd $r$, the 
set $\{\upc_{C}\colon C\in \mathcal{CB}_{X}^{(2)}\}$ is a linearly independent subset 
of $\upZ_1(X;\Z/2)$ for all 
$X\in \{ \Pr_r^{\boxtimes}, \Pr_r^{\boxminus}, \upM_r^{\boxtimes}, \upM_r^{\boxminus}\}$\ , 
\end{enumerate}\end{minipage}

\begin{minipage}[b]{0.9\linewidth}\begin{enumerate}[label={\rm(a\arabic{*})},start=22]
\item\label{98897665642334334544334879768}
Concerning $\Pr_r^{\boxtimes}$ and $\Pr_r^{\boxminus}$ for even $r\geq 4$, 
and concerning $\upM_r^{\boxtimes}$ and $\upM_r^{\boxminus}$ for odd $r\geq 5$, 
the sum $\left\langle\mathcal{CB}_{X}^{(1)}\right\rangle_{\Z/2}+ 
\left\langle \mathcal{CB}_{X}^{(2)}\right\rangle_{\Z/2} \subseteq \upC_1(X;\Z/2)$ 
is direct for all 
$X\in \{ \Pr_r^{\boxtimes}, \Pr_r^{\boxminus}, \upM_r^{\boxtimes}, \upM_r^{\boxminus}\}$ \ , 
\end{enumerate}\end{minipage}

\begin{minipage}[b]{0.95\linewidth}\begin{enumerate}[label={\rm(a\arabic{*})},start=23]
\item\label{9iu8u7yr5f4d3s44d5789u}
Concerning $\Pr_r^{\boxtimes}$ and $\Pr_r^{\boxminus}$ for even $r$, and 
concerning $\upM_r^{\boxtimes}$ and $\upM_r^{\boxminus}$ for odd $r$\ , 

{\scriptsize
\begin{minipage}[b]{0.5\linewidth}
\begin{enumerate}[label={\rm($\boxtimes$.(\arabic{*}))},labelsep=1.5em,start=0]
\item\label{3287923478248979847238947238} 
$\langle \mathcal{H}(\Pr_r^{\boxtimes}) \rangle_{\Z/2} = \upZ_1(\Pr_r^{\boxtimes};\Z/2)$
\ ,
\end{enumerate}
\end{minipage}
\begin{minipage}[b]{0.48\linewidth}
\begin{enumerate}[label={\rm($\boxtimes$.(\arabic{*}))},labelsep=1.5em,start=1]
\item\label{935872702348723487932893248934} 
$\langle \mathcal{H}(\upM_r^{\boxtimes}) \rangle_{\Z/2} = \upZ_1(\upM_r^{\boxtimes};\Z/2)$ 
\ ,
\end{enumerate}
\end{minipage}

\begin{minipage}[b]{0.5\linewidth}
\begin{enumerate}[label={\rm($\boxminus$.(\arabic{*}))},labelsep=1.5em,start=0]
\item\label{349078364454938747474747329} 
$\dim_{\Z/2}\bigl(\upZ_1(\Pr_r^{\boxminus};\Z/2) / \langle \mathcal{H}(\Pr_r^{\boxminus}) \rangle_{\Z/2}\bigr) = 1$ \ ,
\end{enumerate}
\end{minipage}
\begin{minipage}[b]{0.48\linewidth}
\begin{enumerate}[label={\rm($\boxminus$.(\arabic{*}))},labelsep=1.5em,start=1]
\item\label{1364867134086734108413408746} 
$\dim_{\Z/2}\bigl(\upZ_1(\upM_r^{\boxminus};\Z/2) / \langle \mathcal{H}(\upM_r^{\boxminus}) \rangle_{\Z/2}\bigr) = 1$ \ ,
\end{enumerate}
\end{minipage}

\begin{enumerate}[label={\rm($\boxminus$.$f_0(\cdot)-1$.(\arabic{*}))},
leftmargin=13em,labelsep=1.5em,start=0]
\item\label{83276268932879348793254328793} 
$\langle \mathcal{C}_{\{f_0(\cdot)-1,f_0(\cdot)\}}(\Pr_r^{\boxminus})\rangle_{\Z/2}$ 
$=$ $\upZ_1(\Pr_r^{\boxminus};\Z/2)$ \; ,
\item\label{328238346786839328743267237773} 
$\langle \mathcal{C}_{\{f_0(\cdot)-1,f_0(\cdot)\}}(\upM_r^{\boxminus})\rangle_{\Z/2}$ 
$=$ $\upZ_1(\upM_r^{\boxminus};\Z/2)$ \quad ,
\end{enumerate}
}
\vspace{0.5em}
\end{enumerate}\end{minipage}

\begin{minipage}[b]{0.46\linewidth}\begin{enumerate}[label={\rm(a\arabic{*})},start=24]
\item\label{3w44345465566565565r56} 
if $r$ is even, then {\scriptsize $\Pr_r^{\boxtimes}\in\mathcal{M}_{\{f_0(\cdot)\},0}$}\; ,
\end{enumerate}\end{minipage}
\begin{minipage}[b]{0.5\linewidth}\begin{enumerate}[label={\rm(a\arabic{*})},start=25]
\item\label{8767856454354e54565676767} 
if $r$ is odd, then {\scriptsize $\upM_r^{\boxtimes}\in\mathcal{M}_{\{f_0(\cdot)\},0}$}\; ,  
\end{enumerate}\end{minipage}

\begin{minipage}[b]{0.46\linewidth}\begin{enumerate}[label={\rm(a\arabic{*})},start=26]
\item\label{7667545344edr65656554r54d} 
if $r$ is even, then {\scriptsize $\Pr_r^{\boxminus}\in\mathcal{M}_{\{f_0(\cdot)\},1}$}\; ,
\end{enumerate}\end{minipage}
\begin{minipage}[b]{0.5\linewidth}\begin{enumerate}[label={\rm(a\arabic{*})},start=27]
\item\label{6756667898887y7yy7765544544ew} 
if $r$ is odd, then {\scriptsize $\upM_r^{\boxminus}\in\mathcal{M}_{\{f_0(\cdot)\},1}$}\; ,  
\end{enumerate}\end{minipage}

\begin{minipage}[b]{0.46\linewidth}\begin{enumerate}[label={\rm(a\arabic{*})},start=28]
\item\label{87876665r5r5r5665656565} if $r$ is even, then 
{\scriptsize $\Pr_r^{\boxminus}\in\mathcal{M}_{\{f_0(\cdot)-1,f_0(\cdot)\},0}$}\; ,
\end{enumerate}\end{minipage}
\begin{minipage}[b]{0.5\linewidth}\begin{enumerate}[label={\rm(a\arabic{*})},start=29]
\item\label{676776hy7h77h7767765656657} if $r$ is odd, then 
{\scriptsize $\upM_r^{\boxminus}\in\mathcal{M}_{\{f_0(\cdot)-1,f_0(\cdot)\},0}$}\; ,  
\end{enumerate}\end{minipage}

\begin{minipage}[b]{0.95\linewidth}\begin{enumerate}[label={\rm(a\arabic{*})},start=30]\item\label{it:roughbandwidthstatements}
for every $\beta>0$ there exists $n_0 = n_0(\beta)\in \Z$ such that---in case of 
$\Pr_r^{\boxtimes}$ and $\Pr_r^{\boxminus}$ for even $r$ while in case 
of $\upM_r^{\boxtimes}$ and $\upM_r^{\boxminus}$ for odd $r$---if 
$Y \in\{C_n^2,\CL_r,\Pr_r^{\boxtimes},\Pr_r^{\boxminus},\upM_r^{\boxtimes},\upM_r^{\boxminus}\}$ 
and $f_0(Y)\geq n_0$, the following is true: the bandwidth satisfies 
$\bw(Y) \leq \beta \cdot f_0(Y)$, and moreover for each 
$Y \in\{\Pr_r^{\boxtimes},\Pr_r^{\boxminus},\upM_r^{\boxtimes},\upM_r^{\boxminus}\}$ there 
exists a bijection $\upb_Y\colon\upV(Y)\to\{1,\dotsc,f_0(Y)\}$ and a map 
$\uph_Y\colon \upV(Y) \to \{0,1,2\}$ such that $\upb_Y$ is a 
bandwidth-$\beta f_0(Y)$-labelling and $\uph_Y$ a $3$-colouring of $Y$, and $\uph_Y$ 
has $\lvert \uph_Y^{-1}(0)\rvert \leq \beta f_0(Y)$ and is 
$\bigl(8 \cdot 2\cdot \beta\cdot f_0(Y), 4 \cdot 2\cdot \beta\cdot f_0(Y) \bigr)$-zero-free w.r.t. $\upb_Y$ \; .
\end{enumerate}\end{minipage}
}
\end{lemma}

There are arbitrary choices to be made when proving 
Lemma~\ref{lem:generatingpropertiesofauxiliarestructures}. Let us especially 
mention that there are three different feasible strategies for proving 
\ref{lem:it:bipartitecyclicladder:generator}: 

\begin{enumerate}[label={\rm(A\arabic{*})}]
\item\label{it:way1} Realize $\CL_r$ as a Cayley graph on a finite abelian group. 
Then cite a theorem of B.~Alspach, S.~C.~Locke and D.~Witte which implies 
that $\upZ_1(\CL_r;\Z/2)$ is generated by Hamilton circuits. 
\item\label{it:way2} Determine the full set of non-separating induced 
circuits of $\CL_r$, then realize every single such circuit as a $\Z/2$-sum 
of Hamilton circuits of $\CL_r$ and then appeal to a theorem 
of W.~T.~Tutte (\cite[Statement (2.5)]{MR0158387} \cite[Theorem~3.2.3]{MR2159259}) 
which states that in a $3$-connected graph $X$ the cycle space $\upZ_1(X;\Z/2)$ is 
generated by the set of all non-separating induced circuits.  
\item\label{it:way3} Exhibit sufficiently many explicit Hamilton circuits 
of $\CL_r$ so that after choosing some basis the matrix of these circuits 
has $\Z/2$-rank equal to $\dim_{\Z/2} \upZ_1(\CL_r;\Z/2)$. It then follows 
that $\upZ_1(\CL_r;\Z/2) = \langle \mathcal{H}(\CL_r)\rangle_{\Z/2}$, since 
in a vector space, a maximal linearly independent subset is a generating system. 
\end{enumerate}

Each of \ref{it:way1}--\ref{it:way3} demands attention to the parity of $r$, for 
despite a superficial similarity, the sets of circuits in $\CL_r$ for odd and 
even $r$ turn out to be quite different. A positive way to look at this is as 
helping to decide which of \ref{it:way1}--\ref{it:way3} to choose. While 
each argument can be used for each parity of $r$, there are some reasons to choose 
\ref{it:way1} for odd $r$ and \ref{it:way2} for even $r$. The reason is a 
\emph{trade-off between being a circulant graph (i.e. a Cayley graph on a finite 
cyclic group) and being a planar graph:} 
if $r$ is even, then it can be shown that $\CL_r$ is not isomorphic to any 
Cayley graph on a cyclic group, whereas when $r$ is odd, $\CL_r$ \emph{is} a 
circulant graph. In return, $\CL_r$ is \emph{planar} if and only if $r$ is even, 
and this facilitates \ref{it:way2}: when it comes to proving that no 
non-separating induced circuits of $\CL_r$ have been overlooked, the planarity 
of $\CL_r$ for even $r$ opens up a shortcut via a theorem of A.~Kelmans. 
For odd $r$, however, the non-planarity of $\CL_r$ (easy to prove 
via Kuratowski's theorem \cite[p.~494]{MR0224499}), makes this shortcut 
disappear.\footnote{And then the non-separating induced circuits 
of $\CL_r \cong \upM_r$ are more numerous to boot. While an argument by 
realizing each non-separating induced circuit as a $\Z/2$-sum of 
Hamilton circuits is of course still possible due to $3$-connectedness 
of $\upM_r$, the point is that 
carrying out this argument suddenly becomes more laborious for the double 
reason that the convenient criterion for completeness of the list of all such 
circuits loses its validity and at the same time the number of such circuits 
is even larger.} For these reasons, \ref{it:way2} 
takes considerably more work when $r$ is odd than when $r$ is even. 

In the proofs in Section~\ref{t6t66t7y8778655454445e45} we will opt for 
the shortest route \ref{it:way1}. However, since an argument via non-separating 
induced circuits appears to have some value for auxiliary structures not realizable 
as Cayley-graphs, we will give an example for the constructive 
argumentation \ref{it:way2} in the special 
Section~\ref{sec:alternativeargumentation}---but only for even $r$. 
Strategy~\ref{it:way3}, the most arbitrary of all three (usually there is no 
overriding justification for choosing a particular set of linearly-independent 
Hamilton circuits except that it works) will be used for 
proving \ref{9iu8u7yr5f4d3s44d5789u}, i.e. for 
dealing with the rather ad-hoc auxiliary structures 
$\Pr_r^{\boxtimes}$, $\Pr_r^{\boxminus}$, $\upM_r^{\boxtimes}$ and $\upM_r^{\boxminus}$. 

\subsubsection{Explanation of \ref{it:proofstrategy:liftingstep}}  

A set of graphs is called a \emph{graph property} if and only if it is fixed 
(as a set) by graph isomorphisms. A graph property 
$\mathfrak{X}$ is called \emph{monotone increasing} if and only if for every 
$X\in\mathfrak{X}$, adding to $X$ an arbitrary edge again results in an 
element of $\mathfrak{X}$. A graph property $\mathfrak{X}$ consisting of 
bipartite graphs only is called a 
\emph{monotone increasing property of bipartite graphs} if and only if for every 
$X\in\mathfrak{X}$, adding to $X$ an arbitrary edge which does not create 
an odd circuit again results in an element of $\mathfrak{X}$. 

\begin{lemma}\label{monotonicityofsetoflminus1pathconnectedlgeneratedfinitegraphs}
For any function $\mathfrak{L}$ mapping graphs to subsets of $\Z_{\geq 1}$ and 
any $\xi\in\Z_{\geq 0}$, 
\begin{enumerate}[label={\rm(\arabic{*})}]
\item\label{lockelemmaforgeneralgraphs} the set $\mathcal{M}_{\mathfrak{L},\xi}$ is a 
monotone increasing graph property \quad , 
\item\label{lockelemmaforbipartitegraphs} the set 
$\mathrm{b}\mathcal{M}_{\mathfrak{L},\xi}$ is a monotone increasing property of 
bipartite graphs  \quad . 
\end{enumerate}
\end{lemma}

Lemma~\ref{monotonicityofsetoflminus1pathconnectedlgeneratedfinitegraphs} can serve 
to elevate theorems guaranteeing the existence of spanning subgraphs with a certain 
property to theorems guaranteeing this property for the entire ambient graph: 

\begin{corollary}[lifting properties from spanning subgraphs to host graphs]
\label{cor:conversion}
Let $\mathfrak{L}$ be a function mapping graphs to subsets of $\Z_{\geq 1}$, 
let $\xi\in\Z_{\geq 0}$, let $\mathfrak{X}$ be a set of graphs and 
let $\mathrm{b}\mathfrak{X}$ be a set of bipartite graphs. Then: 
{\small
\begin{enumerate}[label={\rm(\arabic{*})}]
\item\label{cor:lifting:generalgraphs} $\Biggl ($ 
\begin{minipage}[b]{0.37\linewidth} 
if $X\in\mathfrak{X}$, then $\exists Y\in \mathcal{M}_{\mathfrak{L},\xi}$ 
with $f_0(Y)=f_0(X)$ and $Y\hookrightarrow X$ 
\end{minipage} 
$\Biggr)$ $\Longrightarrow$ $\bigl ($ 
if $X\in\mathfrak{X}$, then $X\in\mathcal{M}_{\mathfrak{L},\xi}$ $\bigr)$, 
\item\label{cor:lifting:bipartitegraphs}
$\Biggl ($ 
\begin{minipage}[b]{0.37\linewidth} 
if $X\in\mathrm{b}\mathfrak{X}$, then 
$\exists Y\in\mathrm{b}\mathcal{M}_{\mathfrak{L},\xi}$ 
with $f_0(Y)=f_0(X)$ and $Y\hookrightarrow X$ 
\end{minipage} 
$\Biggr)$ $\Longrightarrow$ $\bigl ($ if $X\in\mathrm{b}\mathfrak{X}$, then 
$X\in\mathrm{b}\mathcal{M}_{\mathfrak{L},\xi}$ $\bigr)$. 
\hfill $\Box$
\end{enumerate}
}
\end{corollary}

Lemma~\ref{monotonicityofsetoflminus1pathconnectedlgeneratedfinitegraphs} is what 
makes \ref{it:proofstrategy:liftingstep} tick. It is very similar to a lemma 
of S.~C.~ Locke \cite[Lemma~1]{MR818599}, but we will re-prove 
Lemma~\ref{monotonicityofsetoflminus1pathconnectedlgeneratedfinitegraphs} 
in Section~\ref{t6t66t7y8778655454445e45}, for three reasons: first, Locke's 
assumption of $2$-connectedness and the attendant appeal to Menger's 
theorem \cite[p.~253, last line]{MR818599} were 
appropriate while 
being concerned with a (possibly small) subgraph of special nature within 
a larger $2$-connected graph but seem out of place when dealing 
with \emph{spanning} subgraphs. It feels more to the point to explicitly 
name a rank-one direct summand which is acquired as a result of the added edge. 

Second, we will need a version of Locke's lemma especially 
phrased for bipartite graphs, and this is not to be found in (but easily 
obtained by a small modification of) \cite{MR818599}. 

Third, there is a simple algebraic lemma underlying 
Lemma~\ref{monotonicityofsetoflminus1pathconnectedlgeneratedfinitegraphs}, and for 
this lemma it appears that free modules over principal 
ideal domains provide the natural generality. With a view towards possible future 
research on the role of $\mathcal{H}(X)$ vis-{\`a}-vis the $\Z$-module 
$\upZ_1(X;\Z)$ (for $X$ with high $\delta(X)$), let us opt for this generality 
right-away, at negligible additional cost, but with more insight into the 
underlying mechanism. If $R$ is a commutative ring, $M$ a free $R$-module and 
$\mathcal{B}\subseteq M$ an $R$-basis of $M$, then for every $v\in M$ we 
write $(\lambda_{\mathcal{B},v,b})_{b\in\mathcal{B}}\in R^B$ for the unique element 
of $R^B$ (cofinitely-many components zero) with 
$v = \sum_{b\in\mathcal{B}} \lambda_{\mathcal{B},v,b}\ b$. Then for every $b\in\mathcal{B}$ 
the map $\lambda_{\mathcal{B},\cdot,b} \colon v\mapsto \lambda_{\mathcal{B},v,b}$ is an 
element of $\Hom_R(M,R)$ (which elsewhere is often denoted by $b^\ast$). Moreover, we 
define $\Supp_{\mathcal{B}}(v) := 
\{ b\in \mathcal{B}\colon \lambda_{\mathcal{B},v,b}\neq 0 \} \subseteq \mathcal{B}$. 
We can now formulate a slight generalization 
of \cite[Lemma~1]{MR818599} and \cite[Corollary~3.2]{MR1057481}, 
which is the algebraic mechanism underlying 
Lemma~\ref{monotonicityofsetoflminus1pathconnectedlgeneratedfinitegraphs}:

\begin{lemma}\label{lem:decomposition}
If $R$ is a principal ideal domain, $R^\times$ its group of units, 
$M$ a finitely-generated free $R$-module, 
$\mathcal{B}\subseteq M$ an $R$-basis of $M$, $b_0\in\mathcal{B}$ an 
arbitrary element, $U\subseteq M$ an arbitrary sub-$R$-module, and 
$u_0\in U$ an arbitrary element with $\lambda_{\mathcal{B},u_0,b_0}\in R^\times$, then 
\begin{equation}\label{eq:decompositioningeneral}
U = \langle \{ u\in U \colon b_0\notin\Supp_{\mathcal{B}}(u)\} \rangle_R 
\oplus \langle u_0 \rangle_R \quad . 
\end{equation}
\end{lemma}

\section{Proofs}\label{677656565658899822222222222222221}

\subsection{Proofs of the main results}\label{876556675765rr5r56565656565}

\subsubsection{Proofs of the implications in Theorem~\ref{thm:mainresults}}\label{787y5tr54r54565656556565}

As to \ref{thm:minDegOneHalfPlusGammaImpliesHamiltonGeneratedForOddOrder}, let 
$\gamma>0$ be given and invoke Theorem~\ref{thm:BoettcherSchachtTaraz2009} with 
this $\gamma$, $\rho:=2$ and $\Delta:=4$ to get a $\beta>0$ and an $n_0$, here 
denoted by $n_0'$, with the property stated there. Give this $\beta$ 
to Lemma~\ref{lem:generatingpropertiesofauxiliarestructures}.\ref{it:roughbandwidthstatements} to get an $n_0 = n_0(\beta)$, here denoted by $n_0''$, with the properties 
stated there. We now argue that with $n_0:=\max(n_0',n_0'')$ the claim 
in \ref{thm:minDegOneHalfPlusGammaImpliesHamiltonGeneratedForOddOrder} is true. 
Let $\mathfrak{X}$ be the set of all graphs $X$ with odd $f_0(X)\geq n_0$ and 
$\delta(X)\geq (\tfrac12+\gamma)f_0(X)$. Let $X\in\mathfrak{X}$ be arbitrary, 
$r:= \tfrac12(f_0(X)-1)$ and $Y:=\Pr_r^{\boxtimes}$ 
in case $f_0(X) \equiv 1\ (\mathrm{mod}\ 4)$, 
resp. $Y:=\upM_r^{\boxtimes}$ in case $f_0(X) \equiv 3\ (\mathrm{mod}\ 4)$. 
Then $Y\in\mathcal{M}_{\{f_0(\cdot)\},0}$ in 
view of Lemmas~\ref{lem:generatingpropertiesofauxiliarestructures}.\ref{3w44345465566565565r56} and \ref{lem:generatingpropertiesofauxiliarestructures}.\ref{8767856454354e54565676767}, moreover $f_0(Y)=f_0(X)$ and also $Y\hookrightarrow X$ since $\Delta(Y)=4\leq \Delta$ and Lemma~\ref{lem:generatingpropertiesofauxiliarestructures}.\ref{it:roughbandwidthstatements} in the case `$Y=\Pr_r^{\boxtimes}$' (resp. `$Y=\upM_r^{\boxtimes}$') 
allows us to apply Theorem~\ref{thm:BoettcherSchachtTaraz2009}---with the 
$\gamma$, $\rho$, $\Delta$, $\beta$, $n_0$ we already fixed---to the 
graphs $X$ and $Y$. Therefore, by 
Corollary~\ref{cor:conversion}.\ref{cor:lifting:generalgraphs} it follows 
that $X\in\mathcal{M}_{\{f_0(\cdot)\},0}$, 
in particular $X\in\mathrm{cd}_0\mathcal{C}_{\{f_0(\cdot)\}}$, which is what is claimed 
in \ref{thm:minDegOneHalfPlusGammaImpliesHamiltonGeneratedForOddOrder}. 

As to \ref{thm:minDegOneHalfPlusGammaImpliesAllOneCanAskForWhenOrderIsEven}, if 
throughout the preceding paragraph we replace 
`\ref{thm:minDegOneHalfPlusGammaImpliesHamiltonGeneratedForOddOrder}' 
by `\ref{thm:minDegOneHalfPlusGammaImpliesAllOneCanAskForWhenOrderIsEven}', 
`odd' by `even', 
`$r:=\tfrac12(f_0(X)-1)$' by `$r:=\tfrac12f_0(X)$', 
`$\Pr_r^{\boxtimes}$' by `$\Pr_r^{\boxminus}$', 
`$\upM_r^{\boxtimes}$' by `$\upM_r^{\boxminus}$', 
`$\mathcal{M}_{\{f_0(\cdot)\},0}$' by `$\mathcal{M}_{\{f_0(\cdot)\},1}$', 
`Lemma~\ref{lem:generatingpropertiesofauxiliarestructures}.\ref{3w44345465566565565r56}' by `Lemma~\ref{lem:generatingpropertiesofauxiliarestructures}.\ref{7667545344edr65656554r54d}', 
`Lemma~\ref{lem:generatingpropertiesofauxiliarestructures}.\ref{8767856454354e54565676767}' by `Lemma~\ref{lem:generatingpropertiesofauxiliarestructures}.\ref{6756667898887y7yy7765544544ew}', 
`$\Delta(Y)=4$' by `$\Delta(Y)=5$', and 
`$\mathrm{cd}_0\mathcal{C}_{f_0(\cdot)}$' by `$\mathrm{cd}_1\mathcal{C}_{f_0(\cdot)}$', 
then we obtain a proof of the codimension-one-statement in \ref{thm:minDegOneHalfPlusGammaImpliesAllOneCanAskForWhenOrderIsEven}. Moreover, if in these replacement 
instructions we replace 
`$\mathcal{M}_{\{f_0(\cdot)\},1}$' by `$\mathcal{M}_{\{f_0(\cdot)-1,f_0(\cdot)\},0}$', 
`Lemma~\ref{lem:generatingpropertiesofauxiliarestructures}.\ref{7667545344edr65656554r54d}' by `Lemma~\ref{lem:generatingpropertiesofauxiliarestructures}.\ref{87876665r5r5r5665656565}', and 
`Lemma~\ref{lem:generatingpropertiesofauxiliarestructures}.\ref{6756667898887y7yy7765544544ew}' by `Lemma~\ref{lem:generatingpropertiesofauxiliarestructures}.\ref{676776hy7h77h7767765656657}', and then apply the new instructions once more to the first 
paragraph, we obtain a proof of the second claim 
in \ref{thm:minDegOneHalfPlusGammaImpliesAllOneCanAskForWhenOrderIsEven}. 

As to \ref{thm:minDegOneFourthPlusGammaImpliesHamiltonGeneratednessInBipartiteGraphs}, 
let $\gamma>0$ be given and invoke Theorem~\ref{thm:BoettcherHeinigTaraz2010} with 
this $\gamma$ and $\Delta:=3$ to get a $\beta>0$ and an $n_0$, here 
denoted by $n_0'$, with the property stated there. Give this $\beta$ 
to Lemma~\ref{lem:generatingpropertiesofauxiliarestructures}.\ref{it:roughbandwidthstatements} to get an $n_0 = n_0(\beta)$, here denoted by $n_0''$, with the 
properties stated there. We now argue that with $n_0:=\max(n_0',n_0'')$ the claim 
in \ref{thm:minDegOneFourthPlusGammaImpliesHamiltonGeneratednessInBipartiteGraphs} 
is true. Let $\mathrm{b}\mathfrak{X}$ be the set of all square bipartite 
graphs $X$ with $f_0(X)\geq n_0$ and $\delta(X)\geq (\tfrac14+\gamma)f_0(X)$. 
Let $X\in\mathfrak{X}$ be arbitrary and set $r:=\tfrac12 f_0(X)$ and $Y:=\CL_r$. 
Then $Y\in\mathrm{b}\mathcal{M}_{\{f_0(\cdot)\},0}$ in 
view of Lemma~\ref{lem:generatingpropertiesofauxiliarestructures}.\ref{lem:it:bipartitecyclicladder:generator}, moreover $f_0(Y)=f_0(X)$ and also $Y\hookrightarrow X$ 
since $\Delta(Y)=3\leq\Delta$ and Lemma~\ref{lem:generatingpropertiesofauxiliarestructures}.\ref{it:roughbandwidthstatements} in the case $Y=\CL_r$ allows us to apply 
Theorem~\ref{thm:BoettcherHeinigTaraz2010}---with the 
$\gamma$, $\rho$, $\Delta$, $\beta$, $n_0$ we already fixed---to the 
graphs $X$ and $Y$. Therefore, by 
Corollary~\ref{cor:conversion}.\ref{cor:lifting:bipartitegraphs} it follows 
that $X\in\mathrm{b}\mathcal{M}_{\{f_0(\cdot)\},0}$, 
in particular $X\in\mathrm{bcd}_0\mathcal{C}_{\{f_0(\cdot)\}}$, which is what is 
claimed in \ref{thm:minDegOneFourthPlusGammaImpliesHamiltonGeneratednessInBipartiteGraphs}. 

As to \ref{thm:minDegTwoThirds}, let $\mathfrak{X}$ be the set of all graphs $X$ 
with $f_0(X)\geq 2\cdot 10^8$ and $\delta(X)\geq\tfrac23 f_0(X)$. 
Let $X\in\mathfrak{X}$ be arbitrary. Then 
Theorem~\ref{thm:KomlosSarkozySzemeredi1996} 
guarantees that $\upC_{f_0(X)}^2\hookrightarrow X$. If $f_0(X)$ is odd, then by 
combining Corollary~\ref{cor:conversion}.\ref{cor:lifting:generalgraphs} and 
Lemma~\ref{lem:generatingpropertiesofauxiliarestructures}.\ref{6t6t655r5r55r56656565rtrt}, it follows that $X\in\mathcal{M}_{\{f_0(\cdot)\},0}$, 
in particular $X\in\mathrm{cd}_0\mathcal{C}_{\{f_0(\cdot)\}}$, which proves 
\ref{thm:minDegTwoThirds} in the case of odd $f_0$ . If $f_0(X)$ is even, 
then \ref{thm:minDegTwoThirds} follows by combining 
Corollary~\ref{cor:conversion}.\ref{cor:lifting:generalgraphs} with 
Lemma~\ref{lem:generatingpropertiesofauxiliarestructures}.\ref{876676545545655rr565667}, resp. Lemma~\ref{lem:generatingpropertiesofauxiliarestructures}.\ref{898u889y8778y6t5r5454e446446}. 
All the implications in Theorem~\ref{thm:mainresults} have now been proved. 

\subsubsection{Proof of the claim about weakening the hypothesis of \ref{thm:minDegOneHalfPlusGammaImpliesHamiltonGeneratedForOddOrder} in Theorem~\ref{thm:mainresults}}\label{78676r55r5655566567878888u}

\begin{figure} 
\begin{center}
\input{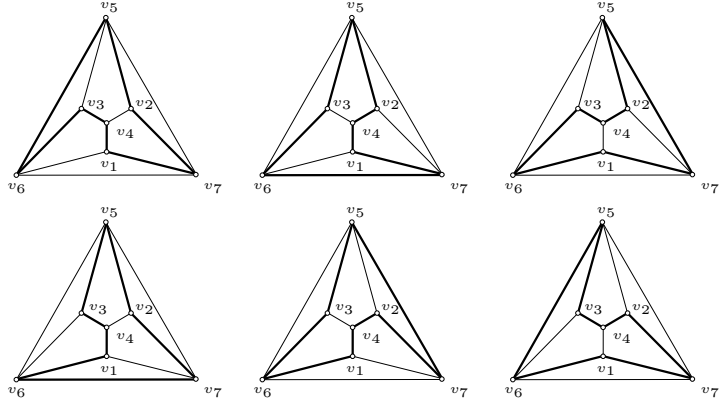}
\caption{A counterexample which proves that a graph having several properties 
which intuitively may seem conducive to the property of being 
Hamilton-generated, can nevertheless fail to have it: the 
graph $\mathrm{CE}_{{\tiny \ref{thm:minDegOneHalfPlusGammaImpliesHamiltonGeneratedForOddOrder}}}$ 
underlying Figure~\ref{3212bh21b1231nk123kn213nmk21} has odd $f_0$, 
is $3$-vertex-connected, only barely fails to satisfy the Dirac condition, 
is pancyclic (despite with $f_1 = 12 \ngeq 12.25$ narrowly missing Bondy's sufficient 
size-condition for the pancyclicity of a hamiltonian graph \cite[p.~81]{MR0285424}), 
is Hamilton-connected and has each of its edges contained in a Hamilton circuit. 
And yet it has its cycle space \emph{not} generated by its Hamilton circuits (all of 
which are shown in the figure).}\label{3212bh21b1231nk123kn213nmk21}
\end{center}
\end{figure}

Let $\mathrm{CE}_{{\tiny \ref{thm:minDegOneHalfPlusGammaImpliesHamiltonGeneratedForOddOrder}}}$ 
denote the seven-vertex graph with 
$\upV(\mathrm{CE}_{{\tiny \ref{thm:minDegOneHalfPlusGammaImpliesHamiltonGeneratedForOddOrder}}})$ $:=$ 
$\{v_1,v_2,v_3,v_4,v_5,v_6,v_7\}$ and 
$\upE(\mathrm{CE}_{{\tiny \ref{thm:minDegOneHalfPlusGammaImpliesHamiltonGeneratedForOddOrder}}})$ 
$:=$ $\{$ $\{v_1,v_4\},$ $\{v_1,v_6\},$ $\{v_1,v_7\},$ 
$\{v_2,v_4\},$ $\{v_2,v_5\},$  $\{v_2,v_7\},$ 
$\{v_3,v_4\},$ $\{v_3,v_5\},$ $\{v_3,v_6\},$ 
$\{v_5,v_6\},$ $\{v_5,v_7\},$ $\{v_6,v_7\}$ $\}$. 
(This is the graph underlying Figure~\ref{3212bh21b1231nk123kn213nmk21}.)
Then $\tfrac12 f_0(\mathrm{CE}_{{\tiny \ref{thm:minDegOneHalfPlusGammaImpliesHamiltonGeneratedForOddOrder}}})$ $=$ $3.5$ $\nleq$ $3$ $=$ $\delta(\mathrm{CE}_{{\tiny \ref{thm:minDegOneHalfPlusGammaImpliesHamiltonGeneratedForOddOrder}}})$, i.e. $\mathrm{CE}_{{\tiny \ref{thm:minDegOneHalfPlusGammaImpliesHamiltonGeneratedForOddOrder}}}$ barely misses the Dirac threshold. The graph $\mathrm{CE}_{{\tiny \ref{thm:minDegOneHalfPlusGammaImpliesHamiltonGeneratedForOddOrder}}}$ has odd $f_0$, 
is $3$-vertex-connected, pancyclic (i.e. contains at least one circuit of each of 
all possible lengths $3,\dotsc, f_0(X)$), Hamilton-connected and has each of its 
edges contained in a Hamilton circuit. Therefore the following fact (which proves 
the claim made in Theorem \ref{thm:mainresults} about weakening 
\ref{thm:minDegOneHalfPlusGammaImpliesHamiltonGeneratedForOddOrder}) also 
shows that the open question 
\ref{question:doesdiracthresholdalreadyimplyhamiltongenerated} in 
Section~\ref{sec:concludingremarks} can easily acquire a negative answer 
if its hypotheses are slightly weakened:

\begin{proposition}\label{3541567456216753254}
$\dim_{\Z/2}\bigl(\upZ_1(\mathrm{CE}_{{\tiny \ref{thm:minDegOneHalfPlusGammaImpliesHamiltonGeneratedForOddOrder}}};\Z/2)/\bigl\langle \mathcal{H}(\mathrm{CE}_{{\tiny \ref{thm:minDegOneHalfPlusGammaImpliesHamiltonGeneratedForOddOrder}}})\bigr\rangle_{\Z/2}\bigr)=1$
\end{proposition}
\begin{proof}
The smallness of 
$\mathrm{CE}_{{\tiny \ref{thm:minDegOneHalfPlusGammaImpliesHamiltonGeneratedForOddOrder}}}$ makes it 
easy to check that $\mathcal{H}(\mathrm{CE}_{{\tiny \ref{thm:minDegOneHalfPlusGammaImpliesHamiltonGeneratedForOddOrder}}})$  consists precisely of the six circuits (shown in 
Figure~\ref{3212bh21b1231nk123kn213nmk21}) 
$C_1 := v_1v_7v_2v_5v_6v_3v_4v_1$, 
$C_2 := v_1v_7v_6v_3v_5v_2v_4v_1$, 
$C_3 := v_1v_7v_5v_2v_4v_3v_6v_1$, 
$C_4 := v_1v_6v_7v_2v_5v_3v_4v_1$, 
$C_5 := v_1v_6v_3v_5v_7v_2v_4v_1$, 
$C_6 := v_1v_6v_5v_3v_4v_2v_7v_1$. 
If the standard basis of $\upC_1(\mathrm{CE}_{{\tiny \ref{thm:minDegOneHalfPlusGammaImpliesHamiltonGeneratedForOddOrder}}};\Z/2)$ is labelled 
$\upe_1$ $:=$ $\upc_{v_1v_4}$, 
$\upe_2$ $:=$ $\upc_{v_1v_6}$, 
$\upe_3$ $:=$ $\upc_{v_1v_7}$, 
$\upe_4$ $:=$ $\upc_{v_2v_4}$, 
$\upe_5$ $:=$ $\upc_{v_2v_5}$, 
$\upe_6$ $:=$ $\upc_{v_2v_7}$, 
$\upe_7$ $:=$ $\upc_{v_3v_4}$, 
$\upe_8$ $:=$ $\upc_{v_3v_5}$, 
$\upe_9$ $:=$ $\upc_{v_3v_6}$, 
$\upe_{10}$ $:=$ $\upc_{v_5v_6}$, 
$\upe_{11}$ $:=$ $\upc_{v_5v_7}$, 
$\upe_{12}$ $:=$ $\upc_{v_6v_7}$, then w.r.t. to this basis the Hamilton circuits 
$C_1$, $\dotsc$, $C_6$ give rise to the matrix 
shown in \eqref{23786324874872117204}, which has 
$\Z/2$-rank $5$. 
\begin{equation}\label{23786324874872117204}
\begin{smallmatrix}
& & C_1 & C_2 & C_3 & C_4 & C_5 & C_6  \\
& & & & & & & \\
\upe_1\  &     & 1 & 1 & 0 & 1 & 1 & 0 \\
\upe_2\  &     & 0 & 0 & 1 & 1 & 1 & 1 \\
\upe_3\  &     & 1 & 1 & 1 & 0 & 0 & 1 \\
\upe_4\  &     & 0 & 1 & 1 & 0 & 1 & 1 \\
\upe_5\  &     & 1 & 1 & 1 & 1 & 0 & 0 \\
\upe_6\  &     & 1 & 0 & 0 & 1 & 1 & 1 \\
\upe_7\  &     & 1 & 0 & 1 & 1 & 0 & 1 \\
\upe_8\  &     & 0 & 1 & 0 & 1 & 1 & 1 \\
\upe_9\  &     & 1 & 1 & 1 & 0 & 1 & 0 \\
\upe_{10}\  &   & 1 & 0 & 0 & 0 & 0 & 1 \\
\upe_{11}\  &   & 0 & 0 & 1 & 0 & 1 & 0 \\
\upe_{12}\  &   & 0 & 1 & 0 & 1 & 0 & 0 
\end{smallmatrix}
\end{equation}
Therefore $\bigl\langle \mathcal{H}(\mathrm{CE}_{{\tiny \ref{thm:minDegOneHalfPlusGammaImpliesHamiltonGeneratedForOddOrder}}})\bigr\rangle_{\Z/2}$ is a 
$5$-dimensional subspace of $\upZ_1(\mathrm{CE}_{{\tiny \ref{thm:minDegOneHalfPlusGammaImpliesHamiltonGeneratedForOddOrder}}};\Z/2)$, which has dimension $\beta_1(\mathrm{CE}_{{\tiny \ref{thm:minDegOneHalfPlusGammaImpliesHamiltonGeneratedForOddOrder}}})$ $=$ $f_1(\mathrm{CE}_{{\tiny \ref{thm:minDegOneHalfPlusGammaImpliesHamiltonGeneratedForOddOrder}}})$ $-$ $f_0(\mathrm{CE}_{{\tiny \ref{thm:minDegOneHalfPlusGammaImpliesHamiltonGeneratedForOddOrder}}})$ $+$ $1$ $=$ $12$ $-$ $7$ $+$ $1$ $=$ $6$. 
This proves Proposition~\ref{3541567456216753254}. 
\end{proof}

\subsubsection{Proof of the claim about weakening the hypothesis of \ref{thm:minDegOneFourthPlusGammaImpliesHamiltonGeneratednessInBipartiteGraphs} in Theorem~\ref{thm:mainresults}}\label{7y6t65r44r5676756565}

Let 
$\mathrm{CE}_{{\tiny \ref{thm:minDegOneFourthPlusGammaImpliesHamiltonGeneratednessInBipartiteGraphs}}}$ 
denote the six by six square bipartite graph with 
$\upV(\mathrm{CE}_{{\tiny \ref{thm:minDegOneFourthPlusGammaImpliesHamiltonGeneratednessInBipartiteGraphs}}})$ $:=$ $\{ v_1,\dotsc, v_6\}$ $\sqcup$ $\{v_7,\dotsc, v_{12}\}$ (bipartition classes 
indicated) and $\upE(\mathrm{CE}_{{\tiny \ref{thm:minDegOneFourthPlusGammaImpliesHamiltonGeneratednessInBipartiteGraphs}}})$ $:=$ $\{$ $v_1v_7,$ $v_1v_8,$ $v_1v_9,$ $v_1v_{12},$ 
$v_2v_7,$ $v_2v_8,$ $v_2v_9,$ $v_3v_7,$ $v_3v_8,$ $v_3v_9,$ 
$v_4v_9,$ $v_4v_{10},$ $v_4v_{11},$ $v_5v_{10},$ $v_5v_{11},$ $v_5v_{12},$ 
$v_6v_{10},$ $v_6v_{11},$ $v_6v_{12}$ $\}$. 
(This is the graph in Figure~\ref{5612512543225632657328}.)
Then $\tfrac14 f_0(\mathrm{CE}_{{\tiny \ref{thm:minDegOneFourthPlusGammaImpliesHamiltonGeneratednessInBipartiteGraphs}}}) = \delta(\mathrm{CE}_{{\tiny \ref{thm:minDegOneFourthPlusGammaImpliesHamiltonGeneratednessInBipartiteGraphs}}}) = 3$ and $\mathrm{CE}_{{\tiny \ref{thm:minDegOneFourthPlusGammaImpliesHamiltonGeneratednessInBipartiteGraphs}}}$ is hamiltonian. We will now prove by a short argument 
that $\langle \mathcal{H}(\mathrm{CE}_{{\tiny \ref{thm:minDegOneFourthPlusGammaImpliesHamiltonGeneratednessInBipartiteGraphs}}})\rangle_{\Z/2}$ has \emph{at least} codimension one in $\upZ_1(\mathrm{CE}_{{\tiny \ref{thm:minDegOneFourthPlusGammaImpliesHamiltonGeneratednessInBipartiteGraphs}}};\Z/2)$, which is enough to establish $\mathrm{CE}_{{\tiny \ref{thm:minDegOneFourthPlusGammaImpliesHamiltonGeneratednessInBipartiteGraphs}}}$ as a counterexample of the claimed kind. (By 
determining all $16$ Hamilton circuits of $\mathrm{CE}_{{\tiny \ref{thm:minDegOneFourthPlusGammaImpliesHamiltonGeneratednessInBipartiteGraphs}}}$ and subsequently computing the $\Z/2$-rank 
of a $12$ by $16$ matrix with zero-one entries it is possible to show that 
$\dim_{\Z/2} \langle\mathcal{H}(\mathrm{CE}_{{\tiny \ref{thm:minDegOneFourthPlusGammaImpliesHamiltonGeneratednessInBipartiteGraphs}}}) \rangle_{\Z/2} = 7 = \dim_{\Z/2} \upZ_1(\mathrm{CE}_{{\tiny \ref{thm:minDegOneFourthPlusGammaImpliesHamiltonGeneratednessInBipartiteGraphs}}};\Z/2) - 1$, i.e. the codimension actually is equal to $1$.)

\begin{proposition}\label{3476532561255674253476547356}
$\dim_{\Z/2} \bigl(\upZ_1(\mathrm{CE}_{{\tiny \ref{thm:minDegOneFourthPlusGammaImpliesHamiltonGeneratednessInBipartiteGraphs}}};\Z/2) / \langle\mathcal{H}(\mathrm{CE}_{{\tiny \ref{thm:minDegOneFourthPlusGammaImpliesHamiltonGeneratednessInBipartiteGraphs}}}) \rangle_{\Z/2}\bigr) \geq 1$
\end{proposition}
\begin{proof}
It is enough to make the following simple observation: since $\{v_1,v_9\}$ is a 
separator of $\mathrm{CE}_{{\tiny \ref{thm:minDegOneFourthPlusGammaImpliesHamiltonGeneratednessInBipartiteGraphs}}}$, the edge $\{v_1,v_9\}$ cannot be an edge of any Hamilton circuit of 
$\mathrm{CE}_{{\tiny \ref{thm:minDegOneFourthPlusGammaImpliesHamiltonGeneratednessInBipartiteGraphs}}}$. 
Therefore the set of all Hamilton circuits of 
$\mathrm{CE}_{{\tiny \ref{thm:minDegOneFourthPlusGammaImpliesHamiltonGeneratednessInBipartiteGraphs}}}$ 
\emph{equals} the set of all Hamilton circuits of the graph 
$\mathrm{CE}_{{\tiny \ref{thm:minDegOneFourthPlusGammaImpliesHamiltonGeneratednessInBipartiteGraphs}}}-\{v_1,v_9\}$ obtained after deleting $\{v_1,v_9\}$ from 
$\mathrm{CE}_{{\tiny \ref{thm:minDegOneFourthPlusGammaImpliesHamiltonGeneratednessInBipartiteGraphs}}}$. 
This in particular implies the first equality in the calculation 
$\dim_{\Z/2} \langle\mathcal{H}(\mathrm{CE}_{{\tiny \ref{thm:minDegOneFourthPlusGammaImpliesHamiltonGeneratednessInBipartiteGraphs}}}) \rangle_{\Z/2}$ $=$ $\dim_{\Z/2} \langle\mathcal{H}(\mathrm{CE}_{{\tiny \ref{thm:minDegOneFourthPlusGammaImpliesHamiltonGeneratednessInBipartiteGraphs}}} - \{v_1,v_9\}) \rangle_{\Z/2}$ $\leq$ (since the dimension of a subspace of a vector space is bounded 
by the dimension of the latter's dimension) $\leq$ $\dim_{\Z/2} \upZ_1(\mathrm{CE}_{{\tiny \ref{thm:minDegOneFourthPlusGammaImpliesHamiltonGeneratednessInBipartiteGraphs}}} - \{v_1,v_9\};\Z/2)$ $=$ (by the Euler--Poincar{\'e} relation) $=$ $\dim_{\Z/2} \upZ_1(\mathrm{CE}_{{\tiny \ref{thm:minDegOneFourthPlusGammaImpliesHamiltonGeneratednessInBipartiteGraphs}}};\Z/2)$~$-$~$1$, which is 
just what is claimed in Proposition~\ref{3476532561255674253476547356}.
\end{proof}

\begin{figure} 
\begin{center}
\input{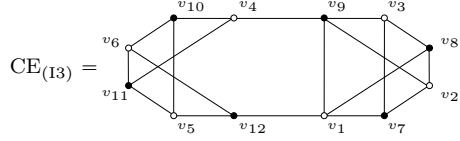}
\caption{A counterexample which proves that if in 
\ref{thm:minDegOneFourthPlusGammaImpliesHamiltonGeneratednessInBipartiteGraphs} the 
hypothesis `$\delta(X)\geq(\tfrac14 + \gamma)f_0(X)$' is weakened to 
`$\delta(X)\geq \tfrac14 f_0(X)$ and $X$~hamiltonian' the implication becomes false: 
the graph 
$\mathrm{CE}_{{\tiny \ref{thm:minDegOneFourthPlusGammaImpliesHamiltonGeneratednessInBipartiteGraphs}}}$ has 
$\delta = 3 = \tfrac14 f_0$ and is hamiltonian, 
yet $\langle \mathcal{H}(\cdot) \rangle_{\Z/2}$ has codimension one in 
$\upZ_1(\cdot;\Z/2)$. If the edge $\{v_1,v_9\}$ were omitted, we \emph{would} have 
$\langle \mathcal{H}(\cdot) \rangle_{\Z/2} = \upZ_1(\cdot;\Z/2)$, hence the resulting 
graph $\mathrm{CE}_{{\tiny \ref{thm:minDegOneFourthPlusGammaImpliesHamiltonGeneratednessInBipartiteGraphs}}}-\{v_1,v_9\}$ would---while still satisfying the weakenend hypotheses with 
respect to which $\mathrm{CE}_{{\tiny \ref{thm:minDegOneFourthPlusGammaImpliesHamiltonGeneratednessInBipartiteGraphs}}}$ \emph{is} a counterexample---cease to be a counterexample. (This 
does not contradict the fact that `Hamilton-laceable and Hamilton-generated' is a 
monotone property of bipartite graphs: 
$\mathrm{CE}_{{\tiny \ref{thm:minDegOneFourthPlusGammaImpliesHamiltonGeneratednessInBipartiteGraphs}}}-\{v_1,v_9\}$ is not Hamilton-laceable.) The author could not find a counterexample 
showing that 
\ref{thm:minDegOneFourthPlusGammaImpliesHamiltonGeneratednessInBipartiteGraphs} 
would become false were `$\delta(X)\geq(\tfrac14 + \gamma)f_0(X)$' weakened only 
to `$\delta(X)\geq \tfrac14 f_0(X)$ and $X$~Hamilton-laceable.'}
\label{5612512543225632657328}
\end{center}
\end{figure}

\subsection{Proofs of the auxiliary results}\label{t6t66t7y8778655454445e45}

\begin{proof}[Proof of 
Lemma~\ref{monotonicityofsetoflminus1pathconnectedlgeneratedfinitegraphs}]
First note that for both $\mathcal{M}_{\mathfrak{L},\xi}$ and 
$\mathrm{b}\mathcal{M}_{\mathfrak{L},\xi}$, it is obvious that the sets are fixed 
(as sets) under any graph isomorphism, i.e. both are graph properties.

As to the monotonicity claim in \ref{lockelemmaforgeneralgraphs}, 
if $\mathcal{M}_{\mathfrak{L},\xi}=\emptyset$, the claim is vacuously true. Otherwise, 
let $X\in \mathcal{M}_{\mathfrak{L},\xi}$ be an arbitrary 
element  and let $e\in \binom{\upV(X)}{2}\setminus \upE(X)$ be arbitrary. We will 
use the abbreviation $X+e := (\upV(X),\upE(X)\sqcup\{e\})$. We have to prove 
$X+e \in \mathcal{M}_{\mathfrak{L},\xi}$. 
Trivially, $X + e\in\mathcal{CO}_{\mathfrak{L}-1}$. What has to be justified is 
that $X+e$ $\in$ $\mathrm{cd}_{\xi}\mathcal{C}_{\mathfrak{L}}$.  
Since $X\in\mathcal{CO}_{\mathfrak{L}-1}$, there exists in $X$ a path $P$ with length 
in $\{ l-1\colon l\in \mathfrak{L} \}$ linking the endvertices of $e$ and we have 
$e\notin \upE(P)$ since $e\notin \upE(X)$. Choose any such $P$. We now use 
Lemma~\ref{lem:decomposition} twice: let $R := \Z/2$, $M := \upC_1(X+e;\Z/2)$, 
$\mathcal{B} := \{\upc_{\tilde{e}}\colon \tilde{e}\in\upE(X+e) \}$ 
(the standard basis of $\upC_1(X+e;\Z/2)$) and $b_0 := e$. Since (with $\{u,v\}:=e$) 
the circuit $C := uPvu$ satisfies both $C \in \mathcal{C}_{\mathfrak{L}}(X+e)$ and 
$C\in \upZ_1(X+e;\Z/2)$, it follows that whether we define 
$U := \langle \mathcal{C}_{\mathfrak{L}}(X + e) \rangle_{\Z/2}$ or 
$U := \upZ_1(X+e;\Z/2)$, in both cases we have $u_0 := \upc_C \in U$, and therefore 
Lemma~\ref{lem:decomposition} gives us 

{\small
\begin{minipage}[b]{0.5\linewidth}
\begin{enumerate}[label={\rm(ds\arabic{*})},leftmargin=3em,start=1]
\item\label{directsum:Circl} $\langle \mathcal{C}_{\mathfrak{L}}(X + e) \rangle_{\Z/2}$ 
$=$ $\langle \mathcal{C}_{\mathfrak{L}}(X) \rangle_{\Z/2} \oplus 
\langle \upc_C \rangle_{\Z/2}$ \quad ,
\end{enumerate}
\end{minipage}
\begin{minipage}[b]{0.5\linewidth}
\begin{enumerate}[label={\rm(ds\arabic{*})},leftmargin=3em,start=2]
\item\label{directsum:Z1} 
$\upZ_1(X + e;\Z/2) = \upZ_1(X;\Z/2) \oplus \langle \upc_C \rangle_{\Z/2}$ \quad . 
\end{enumerate}
\end{minipage}
}
The direct sum decompositions \ref{directsum:Circl} and \ref{directsum:Z1} imply 
$\dim_{\Z/2} \bigl ( \upZ_1(X + e;\Z/2) / 
\langle \mathcal{C}_{\mathfrak{L}}(X + e) \rangle_{\Z/2} \bigr )$ $=$ 
$\dim_{\Z/2} \bigl ( \upZ_1(X;\Z/2) / 
\langle \mathcal{C}_{\mathfrak{L}}(X) \rangle_{\Z/2} \bigr )$ $=$ $\xi$ and therefore 
$X+e$ $\in$ $\mathrm{cd}_{\xi}\mathcal{C}_{\mathfrak{L}}$, completing the proof of 
statement \ref{lockelemmaforgeneralgraphs}. 
As to \ref{lockelemmaforbipartitegraphs}, it suffices to note that the proof 
of \ref{lockelemmaforgeneralgraphs} may be repeated to yield a 
proof of \ref{lockelemmaforbipartitegraphs}, the only change required being to 
restrict $e$ to be an edge whose addition keeps the graph bipartite and 
to replace `$\mathcal{CO}_{\mathfrak{L}-1}$' by `$\mathcal{LA}_{\mathfrak{L}-1}$'. 
\end{proof} 
\begin{proof}[{Proof of Lemma~\ref{lem:decomposition}}]
The sum is obviously direct: $b_0\in \Supp_{\mathcal{B}}(u_0)$ 
while $b_0\notin\Supp_{\mathcal{B}}(v)$ for every 
$v\in \langle \{ u\in U \colon b_0\notin \Supp_{\mathcal{B}}(u) \} \rangle_R$, hence 
the intersection of the summands is $\{0\}$. What is to be justified is that 
$U \subseteq \langle \{ u\in U \colon b_0\notin \Supp_{\mathcal{B}}(u) \} \rangle_R$ 
$+$ $\langle u_0 \rangle_R$. So let $v\in U$ be arbitrary. 
By a well-known theorem (e.g.~\cite[Theorem~6.1]{MR2330890}), since $M$ is a free 
module over a principal ideal domain, so is $U$, and there exists a 
finite $R$-basis $\mathcal{E} \in \binom{U}{\rk_R(U)}$ of $U$. 
Let $\mathcal{E}_0 := \{ e\in \mathcal{E}\colon b_0\in \Supp_{\mathcal{B}}(e)\}$. 
Since $\lambda_{\mathcal{B},\cdot,b_0}\in\Hom_R(M,R)$, we have 
$\lambda_{\mathcal{B},\cdot,b_0}$ $\bigl($ $($ $\sum_{e\in\mathcal{E}\setminus\mathcal{E}_0}$ 
$\lambda_{\mathcal{E},v,e}$ $e$ $)$ $+$ $($ $\sum_{e\in\mathcal{E}_0}$ 
$\lambda_{\mathcal{E},v,e}$ $($ $e$ $-$ $\lambda_{\mathcal{B},e,b_0}$ 
$(\lambda_{\mathcal{B},u_0,b_0})^{-1}$ $u_0$ $)$ $)$  $\bigr)$ $=$ $0$, and 
therefore $b_0$ is not an element of $\Supp_{\mathcal{B}}(\cdot)$ of 
{\small
\begin{equation}\label{8797y7767878878787}
v - \biggl( (\lambda_{\mathcal{B},u_0,b_0})^{-1}\sum_{e\in\mathcal{E}_0}\lambda_{\mathcal{E},v,e}\ \lambda_{\mathcal{B},e,b_0}\  \biggr ) \ u_0 
= \biggl( \sum_{e\in \mathcal{E}\setminus \mathcal{E}_0} 
\lambda_{\mathcal{E},v,e}\ e \biggr) + \biggl(\sum_{e\in\mathcal{E}_0} \lambda_{\mathcal{E},v,e}\ (e -\lambda_{\mathcal{B},e,b_0} (\lambda_{\mathcal{B},u_0,b_0})^{-1} u_0) \biggr) \ .
\end{equation}
}
Thus, writing 
$v  = \bigl( v - \bigl( (\lambda_{\mathcal{B},u_0,b_0})^{-1}\sum_{e\in\mathcal{E}_0}\lambda_{\mathcal{E},v,e}\ \lambda_{\mathcal{B},e,b_0}\ \bigr )\ u_0 \bigr) 
+ \bigl( (\lambda_{\mathcal{B},u_0,b_0})^{-1}\sum_{e\in\mathcal{E}_0}\lambda_{\mathcal{E},v,e}\ \lambda_{\mathcal{B},e,b_0}\  \bigr) \ u_0$ shows that 
$v\in \langle \{ u\in U \colon b_0\notin \Supp_{\mathcal{B}}(u) \} \rangle_R + 
\langle u_0 \rangle_R$, completing the proof of 
$U \subseteq \langle \{ u\in U \colon 
b_0\notin \Supp_{\mathcal{B}}(u) \} \rangle_R \oplus \langle u_0 \rangle_R$. 
\end{proof}

The above proof of Lemma~\ref{lem:decomposition} does not work if 
the assumption of $M$ being finitely generated is dropped: while 
\cite[Theorem~6.1]{MR2330890} remains applicable, i.e. $U$ then still admits a 
basis, there is no general reason why $\mathcal{E}_0$ should then still 
be a finite set, hence the sums in \eqref{8797y7767878878787} may not 
be defined. This obstacle to adapting the monotonicity argument to an 
infinite setting may be a point of interest (possibly one to start from) 
in the unknown territory of linear-algebraic properties Hamilton of circuits 
in infinite graphs. There is also the issue of how to adapt the 
monotonicity argument in order to allow one to add infinitely-many edges.

\begin{proof}[Proof of Lemma~\ref{lem:generatingpropertiesofauxiliarestructures}]
As to \ref{lem:it:squareofhamiltoncircuitascayleygraph}, an easy verification 
shows that the map $\{ v_0, \dotsc, v_{n-1}\} \to \Z/n$, 
$v_i\mapsto \overline{i}$ is a graph isomorphism 
$\upC_n^2 \to\Cay(\Z/n;\{\overline{1},\overline{2},\overline{n-2},\overline{n-1}\})$. 
(Both for this verification and for the ones required in \ref{lem:it:PrIsoCayC2Cr}, 
\ref{lem:it:MrIsoCayC2r}, \ref{lem:it:relationofBCrtoPr} and 
\ref{lem:it:relationofBCrtoMr}, it is recommendable to use an obvious 
and known \cite[Section~1.5, first paragraph]{MR2089014} characterization of graph 
isomorphisms: \emph{every injective graph homomorphism 
between two graphs with equal $f$-vectors is a graph isomorphism.} This relieves 
one of the responsibility to explicitly show that non-edges are mapped to non-edges.)

As to \ref{lem:it:squareofhamiltoncircuit:notaprism}, the definition of $\square$ 
implies that for every graph $X$, every vertex of the graph $X\square \upP_1$ has 
odd degree. But for every $n\geq 5$ the graph $\upC_n^2$ is regular with 
vertex degree four.

As to \ref{876676545545655rr565667} and \ref{6t6t655r5r55r56656565rtrt},
first note that $\upC_n^2$ is non-bipartite, for both parities of $n$, and therefore 
\ref{lem:it:squareofhamiltoncircuitascayleygraph} and Theorem~\ref{thm:ChenQuimpo} 
combined imply that $\upC_n^2\in \mathcal{CO}_{\{f_0(\cdot)\}}$, for every $n$. 
It remains to justify that $\upC_n^2\in \mathrm{cd}_1\mathcal{C}_{\{f_0(\cdot)\}}$ for 
even $n$, resp. $\upC_n^2\in \mathrm{cd}_0\mathcal{C}_{\{f_0(\cdot)\}}$ for odd $n$. 
Both these statements follows from combining \ref{lem:it:squareofhamiltoncircuitascayleygraph} and \ref{lem:it:squareofhamiltoncircuit:notaprism} with 
Theorem~\ref{thm:AlspachLockeWitte}.\ref{thm:alspachlockewitte:numberofverticesodd} 
and Theorem~\ref{thm:AlspachLockeWitte}.\ref{thm:alspachlockewitte:whenthecodimensionisone}. 

As to \ref{898u889y8778y6t5r5454e446446}, first note that 
$\upC_n^2$ does indeed contain circuits of length $f_0(\upC_n^2)-1$ (in fact, 
$f_0(\upC_n^2)$ different ones), and then arbitrarily choose one such circuit $C$. 
Since $n$ is even, $C$ has odd length, and therefore 
$\upc_C\notin\langle\mathcal{H}(\upC_n^2)\rangle_{\Z/2}$. Moreover,  
$\dim_{\Z/2} \langle\mathcal{H}(\upC_n^2)\rangle_{\Z/2}
=\dim_{\Z/2}\upZ_1(\upC_n^2;\Z/2) - 1$ by \ref{876676545545655rr565667}, hence 
$\dim_{\Z/2} \langle\{\upc_C\}\sqcup\mathcal{H}(\upC_n^2)\rangle_{\Z/2} 
\geq \dim_{\Z/2}\upZ_1(\upC_n^2;\Z/2)$ and due to 
$\langle\{\upc_C\}\sqcup\mathcal{H}(\upC_n^2)\rangle_{\Z/2}$ being a $\Z/2$-linear 
subspace of $\upZ_1(\upC_n^2;\Z/2)$, this must hold with equality, 
proving \ref{898u889y8778y6t5r5454e446446}. 

As to \ref{lem:it:PrIsoCayC2Cr}, an easy verification shows that the map 
$\{ x_0,\dotsc,x_{r-1},$ $y_0,\dotsc,y_{r-1} \}$ $\to$ 
$\Z/2\oplus\Z/r$, $x_i\mapsto (\overline{0},\overline{i})$, 
$y_i\mapsto (\overline{1},\overline{i})$ is a graph isomorphism 
$\Pr_r\to\Cay(\Z/2\oplus\Z/r ; \{ (\overline{1},\overline{0}), 
(\overline{0},\overline{1}), (\overline{0},\overline{r-1})\})$. 

As to \ref{lem:it:MrIsoCayC2r}, an easy verification shows that the map 
$\upV(\upM_r)$ $=$ $\{ x_0,\dotsc,x_{r-1},$ $y_0,\dotsc,y_{r-1} \}$ $\to$ $\Z/(2r)$, 
$x_i\mapsto \overline{i}$, $y_i\mapsto \overline{i+r}$ is a graph isomorphism 
$\upM_r \to \Cay\left(\Z/(2r) ; \{ \overline{1},\overline{r},\overline{2r-1}\}\right)$.

As to \ref{lem:it:prism:hamiltonlaceable}, it is easy to check that $r$ being even 
implies that $\Pr_r$ is bipartite. Therefore \ref{lem:it:prism:hamiltonlaceable} 
follows from \ref{lem:it:PrIsoCayC2Cr} combined with Theorem~\ref{thm:ChenQuimpo}. 
Moreover, \ref{lem:it:prism:hamiltonlaceable} is straightforward 
to prove directly. 

As to \ref{lem:it:moebiusladder:hamiltonlaceable}, it is easy to check that $r$ being 
odd implies that $\upM_r$ is bipartite. Therefore \ref{lem:it:moebiusladder:hamiltonlaceable} follows from \ref{lem:it:MrIsoCayC2r} combined with 
Theorem~\ref{thm:ChenQuimpo}. Moreover, \ref{lem:it:moebiusladder:hamiltonlaceable} 
is straightforward to prove directly. 

As to \ref{lem:it:prism:generator}, it is easy to check that $r$ being even implies 
that $\Pr_r$ is bipartite. Therefore, combining \ref{lem:it:PrIsoCayC2Cr} with 
Theorem~\ref{thm:ChenQuimpo} yields that $\Pr_r\in\mathcal{LA}_{\{f_0(\cdot)-1\}}$, and 
combining \ref{lem:it:PrIsoCayC2Cr} with 
Theorem~\ref{thm:AlspachLockeWitte}.\ref{thm:alspachlockewitte:cayleygraphbipartite} 
yields $\Pr_r \in \mathrm{cd}_0\mathcal{C}_{\{f_0(\cdot)\}}$, completing the proof 
of \ref{lem:it:prism:generator}. 

As to \ref{lem:it:moebiusladder:generator}, it is easy to check that $r$ being 
odd implies that $\upM_r$ is bipartite. Therefore, combining \ref{lem:it:MrIsoCayC2r} 
with Theorem~\ref{thm:ChenQuimpo} yields that $\upM_r\in\mathcal{LA}_{\{f_0(\cdot)-1\}}$, 
and combining \ref{lem:it:MrIsoCayC2r} with 
Theorem~\ref{thm:AlspachLockeWitte}.\ref{thm:alspachlockewitte:cayleygraphbipartite} 
yields $\upM_r \in \mathrm{cd}_0\mathcal{C}_{\{f_0(\cdot)\}}$, completing the proof of 
\ref{lem:it:moebiusladder:generator}. 

As to \ref{lem:it:relationofBCrtoPr} and \ref{lem:it:relationofBCrtoMr}, an easy 
verification shows that the map $\upV(\CL_r)\to\upV(\Pr_r)=\upV(\upM_r)$ defined by  
$a_i\mapsto x_i$ for every even $0\leq i \leq r-1$, 
$a_i\mapsto y_i$ for every odd $0\leq i \leq r-1$, 
$b_i\mapsto y_i$ for every even $0\leq i \leq r-1$, 
$b_i\mapsto x_i$ for every odd $0\leq i \leq r-1$, is a graph 
isomorphism $\CL_r\to\Pr_r$ for every even $r\geq 4$ and a graph 
isomorphism  $\CL_r\to\upM_r$ for every odd $r\geq 4$. 

As to \ref{lem:it:bipartitecyclicladder:hamiltongenerated}, this follows by 
combining \ref{lem:it:prism:hamiltonlaceable} and 
\ref{lem:it:moebiusladder:hamiltonlaceable} with 
\ref{lem:it:relationofBCrtoPr} and \ref{lem:it:relationofBCrtoMr}.

As to \ref{lem:it:bipartitecyclicladder:generator}, this follows by combining 
\ref{lem:it:prism:generator} and \ref{lem:it:moebiusladder:generator} with 
\ref{lem:it:relationofBCrtoPr} and \ref{lem:it:relationofBCrtoMr}.

As to \ref{t6t67y8u899o8987g6g678jik90o} and \ref{89878675643545465r5556}, the 
literature apparently does not contain a sufficient criterion for 
Hamilton-connectedness which would apply to either 
$\Pr_r^{\boxtimes}$ or $\upM_r^{\boxtimes}$. Therefore a direct proof by distinguishing 
cases and providing explicit Hamilton paths appears to be 
unavoidable.\footnote{It might be possible to economize somewhat by putting more 
emphasis on the known Hamilton-laceability of cartesian products of the form 
$\Pr_1\square \Pr_\ell$ (which opens up the possibilty to argue by dividing the 
graph into appropriate pieces and subsequently glue Hamilton paths together). 
But even then one has to pay attention to parities, making the gain in brevity 
over explicitly exhibiting Hamilton paths seem small. To give a short example of 
this, Case~2.1.1 (where there is not much gluing to do) has been treated in that 
manner.} Let $\{v,w\}\subseteq\upV(\upM_r^{\boxtimes}) = \upV(\Pr_r^{\boxtimes})$ be 
arbitrary distinct vertices. 

We will repeatedly reduce the work to be done by making use of symmetries. The 
automorphism group of both $\Pr_r^{\boxtimes}$ and $\upM_r^{\boxtimes}$ is the group 
generated by the two unique homomorphic extensions of the maps 
$\binom{\{z,x_0,y_0,x_1,y_1\}\to\{z,x_0,y_0,x_1,y_1\}}
{z\mapsto z,\ x_0\leftrightarrow y_0,\ x_1\leftrightarrow y_1}$ and 
$\binom{\{z,x_0,y_0,x_1,y_1\}\to\{z,x_0,y_0,x_1,y_1\}}
{z\mapsto z,\ x_0\leftrightarrow x_1,\ y_0\leftrightarrow y_1}$ 
to all of $\upV(\Pr_r^{\boxtimes})$ $=$ $\upV(\upM_r^{\boxtimes})$ (thus both 
$\Aut(\Pr_r^{\boxtimes})$ and $\Aut(\upM_r^{\boxtimes})$ are isomorphic to the 
Klein four-group $\Z/2\oplus \Z/2$). These extensions are involutions on 
$\upV(\Pr_r^{\boxtimes})$ $=$ $\upV(\upM_r^{\boxtimes})$ and will be denoted by 
$\Psi_{xy}$ (the map $z\mapsto z$ and $x_i \leftrightarrow y_i$ for every 
$0\leq i \leq r-1$) and $\Psi_{xx}$ (the map $z\mapsto z$ and, for 
$u\in \{x,y\}$, by $u_1\leftrightarrow u_0$, $u_2\leftrightarrow u_{r-1}$, 
$u_3\leftrightarrow u_{r-2}$, $\dotsc$, 
$u_{\lfloor\frac{r+1}{2}\rfloor} \leftrightarrow u_{\lceil\frac{r+1}{2}\rceil}$). 
Both $\Psi_{xy}$ and $\Psi_{xx}$ are automorphisms of both $\upM_r^{\boxtimes}$ 
(for every $r\geq 5$) and $\Pr_r^{\boxtimes}$ (for every $r\geq 4$).

\textsl{Case~1.} $z\in \{v,w\}$. In the absence of information distinguishing 
$v$ from $w$ we may assume $z=v$. 

\textsl{Case~1.1.} $w\in \{x_0,y_0,x_1,y_1\}$. Since $\Aut(\Pr_r^{\boxtimes})$ acts 
transitively on the set $\{x_0,y_0,x_1,y_1\}$  while keeping $z$ fixed, we may 
assume that $w=x_0$. Then $x_0x_1\dotsc x_{r-1}y_{r-1}y_{r-2}\dotsc y_1y_0z$ 
in both $\Pr_r^{\boxtimes}$ and $\upM_r^{\boxtimes}$ is Hamilton path linking 
$v$ and $w$. 
This proves both \ref{t6t67y8u899o8987g6g678jik90o} and \ref{89878675643545465r5556} in the Case~1.1.

\textsl{Case~1.2.} $w\notin \{x_0,y_0,x_1,y_1\}$. Due to $\Psi_{xy}$ we may assume 
that $w=x_i$ with $2\leq i \leq r-1$. Now consider the expressions: 
{\scriptsize
\begin{enumerate}[label={\rm(Pr.1.2.(\arabic{*}))},leftmargin=9em,start=0]
\item\label{31241874612348761248} $x_iy_iy_{i+1}x_{i+1}x_{i+2}y_{i+2}\dotsc y_{r-2}y_{r-1}x_{r-1}x_0x_1x_2\dotsc x_{i-1}y_{i-1}y_{i-2}y_{i-3}\dotsc y_0z$ \quad ,
\item\label{56645546565545454745} $x_iy_iy_{i+1}x_{i+1}x_{i+2}y_{i+2}\dotsc x_{r-2}x_{r-1}y_{r-1}y_0y_1y_2\dotsc y_{i-1}x_{i-1}x_{i-2}x_{i-3}\dotsc x_0z$ \quad ,
\end{enumerate}
}
{\scriptsize
\begin{enumerate}[label={\rm(M.1.2.(\arabic{*}))},leftmargin=9em,start=0]
\item\label{67123867913467346712} $x_iy_iy_{i+1}x_{i+1}x_{i+2}y_{i+2}\dotsc x_{r-2}x_{r-1}y_{r-1}x_0x_1x_2\dotsc x_{i-1}y_{i-1}y_{i-2}y_{i-3}\dotsc y_0z$ \quad ,
\item\label{87553658454352435454} $x_iy_iy_{i+1}x_{i+1}x_{i+2}y_{i+2}\dotsc y_{r-2}y_{r-1}x_{r-1}y_0y_1y_2 \dotsc y_{i-1}x_{i-1}x_{i-2}x_{i-3}\dotsc x_0z$ \quad .
\end{enumerate}
}
If $i$ is even, then \ref{31241874612348761248}, and if $i$ is odd 
then \ref{56645546565545454745} is a Hamilton path of $\Pr_r$ 
linking $v$ and $w$, for every even $r\geq 4$. 
If $i$ is even, then \ref{67123867913467346712}, and if $i$ is odd 
then \ref{87553658454352435454} is a Hamilton path of $\upM_r$ 
linking $v$ and $w$, for every odd $r\geq 5$. 
This proves both \ref{t6t67y8u899o8987g6g678jik90o} and \ref{89878675643545465r5556} 
in the Case~1.2.

\textsl{Case~2.} $z\notin\{v,w\}$. 

\textsl{Case~2.1.} $\{v,w\} \subseteq \{x_0,\dotsc,x_{r-1}\}$ or 
$\{v,w\} \subseteq \{y_0,\dotsc,y_{r-1}\}$. In view of $\Phi_{xy}$  we may 
assume that $\{v,w\} \subseteq \{x_0,\dotsc,x_{r-1}\}$. 

\textsl{Case~2.1.1.} $\{v,w\}\cap\{x_0,x_1\}\neq\emptyset$. In the absence of 
information distinguishing $v$ from $w$ we may assume that $v\in \{x_0,x_1\}$. In 
view of the transitivity of both $\Aut(\Pr_r^{\boxtimes})$ and $\Aut(\upM_r^{\boxtimes})$ 
on $\{x_0,x_1,y_0,y_1\}$ we may further assume that $v=x_0$. Then $w=x_i$ for some 
$i\in [1,r-1]$. We can now reduce the claim we are currently proving to claims about 
a cartesian product of the form $\upP_1\square \upP_l$ (for some $l$) which is 
obtained after deleting certain vertices. The reduction is made possible by 
making---depending on the parity of the $i$ in $x_i$---the right choice of a 
$3$-path or a $4$-path within the graph induced by $\{z,x_0,x_1,y_0,y_1\}$. 

If $i$ is even (hence in particular $i\geq 2$), then starting out with the 
$4$-path $x_0 y_0 z x_1 y_1$ leaves us facing the task of connecting $y_2$ with 
$x_i$ (which lies in the opposite colour class compared to $y_2$) via a Hamilton 
path of the graph remaining after deletion of $\{x_0,y_0,x_1,y_1,z\}$. This 
remaining graph is---regardless of whether we are currently speaking about 
$\upM_r^{\boxtimes}$ or $\Pr_r^{\boxtimes}$---isomorphic to the cartesian product 
$\upP_2\square \upP_{r-3}$, of which the vertex $y_2$ is a `corner vertex' in 
the sense of \cite[Section~2]{MR641233}. Therefore this task \emph{can} be 
accomplished according to \cite[Lemma~1]{MR641233}.

If on the contrary $i$ is odd, then starting out with the $3$-path $x_0 z y_0 y_1$ 
leaves us facing the task of connecting $y_1$ with $x_i$ (which lies in the opposite 
colour class compared to $y_1$) by a Hamilton path of the graph remaining after 
deletion of $\{x_0,y_0,z\}$. This remaining graph is---regardless of whether we are 
currently speaking about $\upM_r^{\boxtimes}$ or $\Pr_r^{\boxtimes}$---isomorphic to the 
cartesian product $\upP_2\square \upP_{r-2}$, of which the vertex `$y_1$' is a corner 
vertex. Therefore this task, too, \emph{can} be accomplished according to 
\cite[Lemma~1]{MR641233}. 
This proves both \ref{t6t67y8u899o8987g6g678jik90o} and \ref{89878675643545465r5556} 
in the Case~2.1.1.

\textsl{Case~2.1.2.} $\{v,w\}\cap\{x_0,x_1\} = \emptyset$. Then $v=x_i$ and $w=x_j$ 
for some $\{i,j\}\in\binom{\{2,3,\dotsc,r-1\}}{2}$. In the absence of information 
distinguishing $v$ from $w$ we may assume that $2\leq i < j \leq r-1$. 

Now consider the expressions 
{\scriptsize
\begin{enumerate}[label={\rm(Pr.2.1.2.(\arabic{*}))},leftmargin=7em]
\item\label{547655676532984023032} 
$x_i x_{i+1}\dotsc x_{j-1} y_{j-1} y_{j-2}\dotsc y_{i-1} x_{i-1} x_{i-2} y_{i-2} y_{i-3} \dotsc x_2 y_2 y_1 x_1 z y_0 x_0 x_{r-1} y_{r-1} y_{r-2} \dotsc x_{j+1} y_{j+1} y_j x_j$ \quad ,
\item\label{787024879423237032872} 
$x_i x_{i+1}\dotsc x_{j-1} y_{j-1} y_{j-2}\dotsc y_{i-1} x_{i-1} x_{i-2} y_{i-2} y_{i-3} \dotsc x_2 y_2 y_1 x_1 z x_0 y_0 y_{r-1} x_{r-1} x_{r-2} \dotsc x_{j+1} y_{j+1} y_j x_j$ \quad ,
\item\label{987352879235897238792} 
$x_i x_{i+1}\dotsc x_{j-1} y_{j-1} y_{j-2}\dotsc y_{i-1} x_{i-1} x_{i-2} y_{i-2} y_{i-3} \dotsc y_2 x_2 x_1 y_1 z y_0 x_0 x_{r-1} y_{r-1} y_{r-2} \dotsc x_{j+1} y_{j+1} y_j x_j$ \quad ,
\item\label{782458725628791111111} 
$x_i x_{i+1}\dotsc x_{j-1} y_{j-1} y_{j-2}\dotsc y_{i-1} x_{i-1} x_{i-2} y_{i-2} y_{i-3} \dotsc y_2 x_2 x_1 y_1 z x_0 y_0 y_{r-1} x_{r-1} x_{r-2} \dotsc x_{j+1} y_{j+1} y_j x_j$ \quad .
\end{enumerate}
} 
and
{\scriptsize
\begin{enumerate}[label={\rm(M.2.1.2.(\arabic{*}))},leftmargin=7em]
\item\label{232356467656724625654} 
$x_i x_{i+1}\dotsc x_{j-1} y_{j-1} y_{j-2}\dotsc y_{i-1} x_{i-1} x_{i-2} y_{i-2} y_{i-3} \dotsc x_2 y_2 y_1 x_1 z y_0 x_0 y_{r-1} x_{r-1} x_{r-2} \dotsc x_{j+1} y_{j+1} y_j x_j$ \quad ,
\item\label{653125543246555443634} 
$x_i x_{i+1}\dotsc x_{j-1} y_{j-1} y_{j-2}\dotsc y_{i-1} x_{i-1} x_{i-2} y_{i-2} y_{i-3} \dotsc x_2 y_2 y_1 x_1 z x_0 y_0 x_{r-1} y_{r-1} y_{r-2} \dotsc x_{j+1} y_{j+1} y_j x_j$ \quad ,
\item\label{432347675412323434344} 
$x_i x_{i+1}\dotsc x_{j-1} y_{j-1} y_{j-2}\dotsc y_{i-1} x_{i-1} x_{i-2} y_{i-2} y_{i-3} \dotsc y_2 x_2 x_1 y_1 z y_0 x_0 y_{r-1} x_{r-1} x_{r-2} \dotsc x_{j+1} y_{j+1} y_j x_j$ \quad ,
\item\label{866432324546665234323} 
$x_i x_{i+1}\dotsc x_{j-1} y_{j-1} y_{j-2}\dotsc y_{i-1} x_{i-1} x_{i-2} y_{i-2} y_{i-3} \dotsc y_2 x_2 x_1 y_1 z x_0 y_0 x_{r-1} y_{r-1} y_{r-2} \dotsc x_{j+1} y_{j+1} y_j x_j$ \quad .
\end{enumerate}
}

If $i$ is even and $j$ is even, 
then \ref{547655676532984023032} for even $r$ is a Hamilton path of 
$\Pr_r^{\boxtimes}$ linking $v$ and $w$ 
and \ref{232356467656724625654} for odd $r$ is one of $\upM_r^{\boxtimes}$, 
while if $i$ is even and $j$ is odd, 
then \ref{787024879423237032872} for even $r$ is a Hamilton path of 
$\Pr_r^{\boxtimes}$ linking $v$ and $w$ 
and \ref{653125543246555443634} for odd $r$ is one of $\upM_r^{\boxtimes}$, 
while if $i$ is odd and $j$ is even, 
then \ref{987352879235897238792} for even $r$ is a Hamilton path of 
$\Pr_r^{\boxtimes}$ linking $v$ and $w$ 
and \ref{432347675412323434344} for odd $r$ is one of $\upM_r^{\boxtimes}$, 
while if $i$ is odd and $j$ is odd, 
then \ref{782458725628791111111} for even $r$ is a Hamilton path of 
$\Pr_r^{\boxtimes}$ linking $v$ and $w$ 
and \ref{866432324546665234323} for odd $r$ is one of $\upM_r^{\boxtimes}$.
This proves both \ref{t6t67y8u899o8987g6g678jik90o} and \ref{89878675643545465r5556} 
in the Case~2.1.2. 

\textsl{Case~2.2.} $\{v,w\} \cap \{x_0,\dotsc,x_{r-1}\} \neq \emptyset$ and $\{v,w\} \cap \{y_0,\dotsc,y_{r-1}\} \neq \emptyset$. Since we are within Case~2 we know that 
$\{v,w\}\subseteq \{x_0,\dotsc,x_{r-1}\}\sqcup\{y_0,\dotsc,y_{r-1}\}$. Therefore the 
statement defining Case~2.2 is the negation of the one defining Case~2.1.
Due to $\Phi_{xy}$ we may assume 
$v = x_i$ with $0\leq i \leq r-1$ and $w = y_j$ with $0\leq j \leq r-1$. 
Due to $\Phi_{xx}$ we may further assume that $i\leq j$. 

\textsl{Case~2.2.1.} $i\in \{0,1\}$. Not only do both $\Aut(\Pr_r^{\boxtimes})$ and 
$\Aut(\upM_r^{\boxtimes})$ act transitively on $\{x_0,x_1,y_0,y_1\}$, but it is possible 
to use this symmetry while still preserving the assumption $i\leq j$ that we already 
made: namely, if $i=1$, hence $v=x_1$ and $w=y_j$ with 
$1=i\leq j$, then $\Psi_{xx}(v)=x_0$ and $\Psi_{xx}(w) = y_{r+1-i}$ (with $y_r:=y_0$) 
and still $0=i\leq j=r+1-i$. Therefore we may further assume that $i=0$, 
i.e. $v=x_0$.
Now consider the expressions 
{\scriptsize
\begin{enumerate}[label={\rm(Pr.2.2.1.(\arabic{*}))},leftmargin=10em,start=0]
\item\label{67814286726326713263261232} 
$x_0 z x_1 x_2 \dotsc x_{j+1} y_{j+1} y_{j+2} x_{j+2} \dotsc  x_{r-2} x_{r-1} y_{r-1} y_0 \dotsc y_{j-1} y_j$ \quad ,
\item\label{98734567837822362368723628} 
$x_0 x_{r-1} x_{r-2}\dotsc x_{j+1} y_{j+1} y_{j+2} \dotsc y_{r-1} y_0 z x_1 y_1 y_2 x_2 \dotsc x_{j-1} x_j y_j$ \quad .
\end{enumerate}
}
{\scriptsize
\begin{enumerate}[label={\rm(M.2.2.1.(\arabic{*}))},leftmargin=10em,start=0]
\item\label{878676755945454656587879897} 
$x_0 z x_1 x_2 \dotsc x_{j+1} y_{j+1} y_{j+2} x_{j+2} \dotsc y_{r-2} y_{r-1} x_{r-1} y_0 y_1 \dotsc y_j$ \quad ,
\item\label{8797567544354454r54d5665} 
$x_0 y_{r-1} x_{r-1} x_{r-2} y_{r-2} \dotsc x_j x_{j+1} \dotsc x_1 z y_0 y_1 \dotsc y_j$ \quad .
\end{enumerate}
}

If $j$ is even, then \ref{67814286726326713263261232}, and if $j$ is odd 
then \ref{98734567837822362368723628} is a Hamilton path of $\Pr_r^{\boxtimes}$ 
linking $v$ and $w$, for every even $r\geq 4$. 
If $j$ is even, then \ref{878676755945454656587879897}, and if $j$ is odd 
then \ref{8797567544354454r54d5665} is a Hamilton path of $\upM_r^{\boxtimes}$ 
linking $v$ and $w$, for every odd $r\geq 4$. This proves 
\ref{t6t67y8u899o8987g6g678jik90o} in the Case~2.2.1. 

\textsl{Case~2.2.2.} $i\notin \{0,1\}$. Now consider the expressions 
{\scriptsize
\begin{enumerate}[label={\rm(Pr.2.2.2.(\arabic{*}))},leftmargin=8em,start=0]
\item\label{326752367523675326753} 
$x_i x_{i+1} \dotsc x_{j+1} y_{j+1} y_{j+2} x_{j+2} \dotsc x_{r-2} x_{r-1} y_{r-1} y_0 x_0 z x_1 y_1 y_2 x_2 x_3 y_3 \dotsc x_{i-2} x_{i-1} y_{i-1} y_i y_{i+1}\dotsc y_j$ \quad ,
\item\label{438977129421747384732} 
$x_i x_{i+1} \dotsc x_{j+1} y_{j+1} y_{j+2} x_{j+2} \dotsc y_{r-2} y_{r-1} x_{r-1} x_0 y_0 z x_1 y_1 y_2 x_2 x_3 y_3 \dotsc x_{i-2} x_{i-1} y_{i-1} y_i y_{i+1}\dotsc y_j$ \quad .
\end{enumerate}
}
{\scriptsize
\begin{enumerate}[label={\rm(M.2.2.2.(\arabic{*}))},leftmargin=8em,start=0]
\item\label{323256345657232433452} 
$x_i x_{i+1} \dotsc x_{j+1} y_{j+1} y_{j+2} x_{j+2} \dotsc y_{r-2} y_{r-1} x_{r-1} y_0 x_0 z x_1 y_1 y_2 x_2 x_3 y_3 \dotsc x_{i-2} x_{i-1} y_{i-1} y_i y_{i+1}\dotsc y_j$ \quad , 
\item\label{867912678032867032867} 
$x_i x_{i+1} \dotsc x_{j+1} y_{j+1} y_{j+2} x_{j+2} \dotsc x_{r-2} x_{r-1} y_{r-1} x_0 y_0 z x_1 y_1 y_2 x_2 x_3 y_3 \dotsc x_{i-2} x_{i-1} y_{i-1} y_i y_{i+1}\dotsc y_j$ \quad .
\end{enumerate}
} 
Since the automorphism $\Psi_{xx}$ changes the parity of the index of an $x_i$, and 
since (as explained in Case~2.2.1) the relation $i\leq j$ is preserved when applying 
$\Psi_{xx}$, we may assume that $i$ is even. 

If $j$ is even, then \ref{326752367523675326753}, and if $j$ is odd then 
\ref{438977129421747384732} is a Hamilton path of $\Pr_r^{\boxtimes}$ 
linking $v$ and $w$, for every even $r\geq 4$. 
If $j$ is even, then \ref{323256345657232433452}, and if $j$ is odd then 
\ref{867912678032867032867} is a Hamilton path of $\upM_r^{\boxtimes}$ 
linking $v$ and $w$, for every odd $r\geq 5$, completing the Case~2.2.2.  

Since at each level of the case distinction the property defining the preceding 
level was partitioned into mutually exclusive properties, both 
\ref{t6t67y8u899o8987g6g678jik90o} and \ref{89878675643545465r5556} have now 
been proved.

As to \ref{32671163432867327682} and \ref{0o978986765565454454e}, 
let $\{v,w\}\subseteq\upV(\Pr_r^{\boxminus})$ be arbitrary distinct vertices. 
For a large part (i.e. for a large majority of instances of the property of being 
Hamilton connected) it is possible to deduce the Hamilton-connectedness 
of $\Pr_r^{\boxminus}$ and $\upM_r^{\boxminus}$ from (the proof of) 
\ref{t6t67y8u899o8987g6g678jik90o} in Lemma~\ref{lem:generatingpropertiesofauxiliarestructures}: if $\{v,w\}\cap\{z',z''\} = \emptyset$, 
then we have $\{v,w\}\subseteq \upV(\Pr_r)\setminus\{z\}$ and therefore each 
Hamilton path $P$ in $\Pr_r$ or $\upM_r$ linking $v$ and $w$ contains $z$ as a 
vertex of degree two. This implies that $P$ can be extended to a Hamilton path 
in $\Pr_r^{\boxminus}$ linking $v$ and $w$. 

If on the contrary $\{v,w\}\cap\{z',z''\}\neq\emptyset$, then there are subcases: 
if $\{v,w\}=\{z',z''\}$, then $z' x_0y_0y_1\dotsc y_{r-1}x_{r-1}x_{r-2}\dotsc x_1 z''$ 
is---in $\Pr_r$ and in $\upM_r$ as well---a Hamilton path linking $v$ and $w$. 

We are left with the case $\lvert\{v,w\}\cap\{z',z''\}\rvert = 1$. In the absence of 
information distinguishing $v$ from $w$ we may assume that 
$v\in \{z',z''\}$ and $w\notin\{z',z''\}$. One may treat this case, too, 
by re-using Hamilton paths in $\Pr_r$ or $\upM_r$, but now it can make a 
difference (for the extendability) \emph{how} such Hamilton path looks like 
around the `special' subgraph induced on the vertices $\{z,x_0,y_0,x_1,y_1\}$ 
and it therefore seems quicker to treat this case directly. Since the property 
`$v\in \{z',z''\}$ and $w\notin\{z',z''\}$', at face value, still comprises several  
cases, we should reduce their number via automorphisms. However---essentially due 
to $x_0 z''$ and the unique degree-$5$-vertex $x_0$ caused by it---both 
$\Aut(\Pr_r^{\boxminus})$ and $\Aut(\upM_r^{\boxminus})$ are trivial. But since 
Hamilton-connectedness is a monotone graph property, it suffices to prove that 
$\Pr_r^{\boxminus,-} := \Pr_r^{\boxminus} - x_0 z'' $ and 
$\upM_r^{\boxminus,-} := \upM_r^{\boxminus} -  x_0 z''$ are 
Hamilton-connected, and these graphs \emph{do} have symmetries again, essentially 
the same as $\Pr_r^{\boxtimes}$ and $\upM_r^{\boxtimes}$. 

The automorphism group of both $\Pr_r^{\boxminus,-}$ and $\upM_r^{\boxminus,-}$ is the 
group generated by the two unique homomorphic extensions of 
$\binom{\{z',z'',x_0,y_0,x_1,y_1\} \to \{z',z'',x_0,y_0,x_1,y_1\}}{z'\mapsto z',\ z''\mapsto z'',\ x_0\leftrightarrow y_0,\ x_1\leftrightarrow y_1}$ and $\binom{\{z',z'',x_0,y_0,x_1,y_1\}\to\{z',z'',x_0,y_0,x_1,y_1\}}{z'\leftrightarrow z'',\ x_0\leftrightarrow x_1,\ y_0\leftrightarrow y_1}$ to all of 
$\upV(\Pr_r^{\boxminus,-})$ $=$ $\upV(\upM_r^{\boxminus,-})$ (thus both 
$\Aut(\Pr_r^{\boxminus,-})$ and $\Aut(\upM_r^{\boxminus,-})$ are isomorphic to the 
Klein four-group $\Z/2\oplus \Z/2$). These extensions are involutions on 
$\upV(\Pr_r^{\boxminus,-})$ $=$ $\upV(\upM_r^{\boxminus,-})$ and will be denoted by 
$\Xi_{xy}$ (the map with $z'\mapsto z'$, $z''\mapsto z''$ and 
$x_i \leftrightarrow y_i$ for every $0\leq i \leq r-1$) and $\Xi_{xx}$ 
(the map with $z'\leftrightarrow z''$ and, for $u\in \{x,y\}$, 
$u_1\leftrightarrow u_0$, $u_2\leftrightarrow u_{r-1}$, 
$u_3\leftrightarrow u_{r-2}$, $\dotsc$, 
$u_{\lfloor\frac{r+1}{2}\rfloor} \leftrightarrow u_{\lceil\frac{r+1}{2}\rceil}$. 
Both $\Xi_{xy}$ and $\Xi_{xx}$ are automorphisms of both $\upM_r^{\boxminus,-}$ 
(for every $r\geq 5$) and $\Pr_r^{\boxminus,-}$ (for every $r\geq 4$). 

Since $\Xi_{xx}$ interchanges $z'$ and $z''$, we may assume that $v=z'$. Then 
there are two cases left: $w\in \{x_0,y_0,x_1,y_1\}$ and its negation  
$w\in \{x_2,y_2,x_3,y_3\dotsc,x_{r-1},y_{r-1}\}$ (keep in mind that we already 
assumed $w\notin\{z',z''\}$ and therefore this indeed is the negation). 

\textsl{Case~1.} $w\in \{x_0,y_0,x_1,y_1\}$. Then since $\Xi_{xy}$ maps 
$x_0\leftrightarrow y_0$ and $x_1\leftrightarrow y_1$ while keeping $z'$ fixed, 
we may assume that $w\in \{ x_0,x_1\}$ and are left with two cases. 

\textsl{Case~1.1.} If $w=x_0$, then 
$z' y_0 y_1 z'' x_1 x_2 y_2 y_3 x_3 \dotsc y_{r-2} y_{r-1} x_{r-1} x_0$ 
is a Hamilton path linking $v$ and $w$ in $\Pr_r^{\boxminus,-}$ 
for every even $r\geq 4$, and 
$z' y_0 y_1 z'' x_1 x_2 y_2 y_3 x_3 \dotsc x_{r-2} x_{r-1} y_{r-1} x_0$ 
is one in $\upM_r^{\boxminus,-}$ for every odd $r\geq 5$. 

\textsl{Case~1.2.} If $w=x_1$, then 
$z' x_0 y_0 y_{r-1} x_{r-1} x_{r-2} y_{r-2} y_{r-3} x_{r-3} \dotsc y_2 y_1 z'' x_1$ 
is a Hamilton path linking $v$ and $w$ in $\Pr_r^{\boxminus,-}$ 
for every even $r\geq 4$, and 
$z' x_0 y_0 x_{r-1} y_{r-1} y_{r-2} x_{r-2} x_{r-3} y_{r-3} \dotsc y_2 y_1 z'' x_1$ 
is one in $\upM_r^{\boxminus,-}$ 
for every odd $r\geq 5$. 

\textsl{Case~2.} $w\in \{x_2,y_2,x_3,y_3\dotsc,x_{r-1},y_{r-1}\}$.
 Then since $\Xi_{xy}$ interchanges the sets $\{x_0,\dotsc,x_{r-1}\}$ 
and $\{y_0,\dotsc,y_{r-1}\}$ while fixing $z'$, we may assume that 
$w = x_i$ with $2\leq i \leq r-1$. If $i\geq 3$, then 
$z' x_0 y_0 y_1 z'' x_1 x_2 y_2 y_3 \dotsc y_{r-1} x_{r-1} x_{r-2} \dotsc x_i$ 
is---regardless of whether $i$ is odd or even---a Hamilton path linking $v$ and $w$ 
in both $\Pr_r^{\boxminus,-}$ and $\upM_r^{\boxminus,-}$. In the case that $i=2$, the path 
$z' y_0 x_0 x_{r-1} y_{r-1} y_{r-2} x_{r-2} x_{r-3} \dotsc x_3 y_3 y_2 y_1 z'' x_1 x_2$
is a Hamilton path linking $v$ and $w$ in $\Pr_r^{\boxminus,-}$, and 
$z' y_0 x_0 y_{r-1} x_{r-1} x_{r-2} y_{r-2} y_{r-3} \dotsc x_3 y_3 y_2 y_1 z'' x_1 x_2$
is one in $\upM_r^{\boxminus,-}$, completing Case~2, and also the proof of 
both \ref{32671163432867327682} and \ref{0o978986765565454454e}. 

As to \ref{38794320986754435465687} in the case $X=\Pr_r^{\boxtimes}$, for every 
even $r\geq 4$, the ($5\times 5$)-minor indexed by 
$x_0y_0$, $x_1y_1$, $zx_1$, $zy_1$, $y_0y_{r-1}$ 
of the $\bigl(f_1(\Pr_r^{\boxtimes})\times 5\bigr)$-matrix which represents the 
elements of $\{\upc_{C}\colon C\in \mathcal{CB}_{\Pr_r^{\boxtimes}}^{(1)}\}$ as elements of 
$\upC_1(\Pr_r^{\boxtimes};\Z/2)\supseteq \upZ_1(\Pr_r^{\boxtimes};\Z/2)$ w.r.t. the 
standard basis of $\upC_1(\Pr_r^{\boxtimes};\Z/2)$, is the one shown 
in \eqref{875633254768786664321}. 
\begin{equation}\label{875633254768786664321}
\begin{smallmatrix}
              & C_{\ev,r,1} & C_{\ev,r,2} & C_{\ev,r,3} & C_{\ev,r,4} & C_{\ev,r,5} \\ 
x_0\wedge    y_0 & 1 & 1 & 1 & 0 & 1 \\
x_1\wedge    y_1 & 1 & 0 & 1 & 1 & 0 \\
  z\wedge    x_1 & 0 & 1 & 1 & 0 & 1 \\
  z\wedge    y_1 & 1 & 1 & 0 & 0 & 1 \\
y_0\wedge y_{r-1} & 0 & 0 & 1 & 1 & 1 
\end{smallmatrix}
\end{equation}
The matrix in \eqref{875633254768786664321} is a nonsingular element of 
$(\Z/2)^{[5]^2}$, its inverse being 
{\tiny$\left(\begin{smallmatrix}
 1 & 0 & 1 & 0 & 0 \\
 0 & 1 & 0 & 1 & 1 \\
 1 & 0 & 0 & 1 & 0 \\
 0 & 1 & 1 & 1 & 0 \\
 1 & 1 & 1 & 0 & 1
\end{smallmatrix}\right)$} $\in(\Z/2)^{[5]^2}$. The existence of one such minor 
by itself proves \ref{38794320986754435465687} in the case $X=\Pr_r^{\boxtimes}$.
As to \ref{38794320986754435465687} in the case $X=\upM_r^{\boxtimes}$, for every 
odd $r\geq 5$, 
the ($5\times 5$)-minor indexed by $x_0y_0$, $x_1y_1$, $zx_1$, $zy_1$, $x_0y_{r-1}$ 
of the $\bigl(f_1(\upM_r^{\boxtimes})\times 5\bigr)$-matrix which represents the 
elements of $\{\upc_{C}\colon C\in \mathcal{CB}_{\upM_r^{\boxtimes}}^{(1)}\}$ as elements of 
$\upC_1(\upM_r^{\boxtimes};\Z/2)\supseteq \upZ_1(\upM_r^{\boxtimes};\Z/2)$ w.r.t. the 
standard basis of $\upC_1(\upM_r^{\boxtimes};\Z/2)$, is the one shown 
in \eqref{8767654323444334434343}. 
\begin{equation}\label{8767654323444334434343}
\begin{smallmatrix}
              & C_{\od,r,1} & C_{\od,r,2} & C_{\od,r,3} & C_{\od,r,4} & C_{\od,r,5} \\ 
x_0\wedge    y_0 & 1 & 1 & 1 & 0 & 1 \\
x_1\wedge    y_1 & 1 & 0 & 1 & 1 & 0 \\
  z\wedge    x_1 & 0 & 1 & 1 & 0 & 1 \\
  z\wedge    y_1 & 1 & 1 & 0 & 0 & 1 \\
x_0\wedge y_{r-1} & 1 & 1 & 0 & 0 & 0 
\end{smallmatrix}
\end{equation}
The matrix in \eqref{8767654323444334434343} is a nonsingular element 
of $(\Z/2)^{[5]^2}$, 
its inverse being 
{\tiny $\left(\begin{smallmatrix}
 1 & 0 & 1 & 0 & 0 \\
 1 & 0 & 1 & 0 & 1 \\
 1 & 0 & 0 & 1 & 0 \\
 0 & 1 & 1 & 1 & 0 \\
 0 & 0 & 0 & 1 & 1
\end{smallmatrix}\right)$} $\in(\Z/2)^{[5]^2}$. The existence of one such minor 
by itself proves \ref{38794320986754435465687} in the case $X=\upM_r^{\boxtimes}$.
As to \ref{38794320986754435465687} in the case $X=\Pr_r^{\boxminus}$, for every 
even $r\geq 4$ the $(5\times 5)$-minor indexed 
by $x_0y_0$, $x_1y_1$, $z'x_0$, $z''y_1$ and $x_0x_{r-1}$ of the 
$\bigl(f_1(\Pr_r^{\boxminus})\times 5\bigr)$-matrix which represents the 
elements of $\{\upc_C\colon C\in \mathcal{CB}_{\Pr_r^{\boxminus}}^{(1)}\}$ 
as elements of $\upC_1(\Pr_r^{\boxminus};\Z/2)\subseteq\upZ_1(\Pr_r^{\boxminus};\Z/2)$ 
w.r.t. the standard basis of  $\upC_1(\Pr_r^{\boxminus};\Z/2)$, is the one shown 
in \eqref{89887659878768767}.
\begin{equation}\label{89887659878768767}
\begin{smallmatrix} & C_{\boxminus,\ev,r,1} & C_{\boxminus,\ev,r,2} & C_{\boxminus,\ev,r,3} & 
C_{\boxminus,\ev,r,4} & C_{\boxminus,\ev,r,5} \\ 
x_0\wedge y_0     & 0 & 0 & 0 & 1 & 0 \\
x_1\wedge y_1     & 0 & 1 & 1 & 0 & 0 \\
z'\wedge x_0      & 1 & 0 & 1 & 1 & 1 \\
z''\wedge y_1     & 0 & 0 & 0 & 0 & 1 \\ 
x_0\wedge x_{r-1}  & 0 & 1 & 0 & 0 & 1 
\end{smallmatrix}
\end{equation}
The matrix in \eqref{89887659878768767} is a nonsingular element 
of $(\Z/2)^{[5]^2}$ with inverse {\tiny $\left(\begin{smallmatrix}
1 & 1 & 1 & 0 & 1 \\
0 & 0 & 0 & 1 & 1 \\
0 & 1 & 0 & 1 & 1 \\
1 & 0 & 0 & 0 & 0 \\
0 & 0 & 0 & 1 & 0 
\end{smallmatrix}\right)$}. The existence of one such minor by itself 
proves \ref{38794320986754435465687} in the case $X=\Pr_r^{\boxminus}$. 
As to \ref{38794320986754435465687} in the case $X=\upM_r^{\boxminus}$, due to the 
similar definitions in 
\ref{76645565667786786676655009877} and \ref{9198786653454809958487332626263}, 
it suffices to note that if in the preceding paragraph 
`$\Pr_r^{\boxminus}$' is replaced by `$\upM_r^{\boxminus}$', 
`even $r\geq 4$' by `odd $r\geq 5$' and `$x_0\wedge x_{r-1}$' by `$x_0\wedge y_{r-1}$', 
then the matrix obtained is exactly the one in \eqref{89887659878768767}. This 
completes the proof of \ref{38794320986754435465687} in its entirety.

As to \ref{98786543322154675878987} in the case $X=\Pr_r^{\boxtimes}$, for every 
even $r\geq 4$, the $\bigl((r-1)\times(r-1)\bigr)$-minor indexed by 
$x_1y_1$, $x_2y_2$, $\dotsc$, $x_{r-1}y_{r-1}$ of the 
$\bigl(f_1(\Pr_r^{\boxtimes})\times (r-1)\bigr)$-matrix which represents the 
elements of $\{\upc_{C}\colon C\in \mathcal{CB}_{\Pr_r^{\boxtimes}}^{(2)}\}$ as elements of 
$\upC_1(\Pr_r^{\boxtimes};\Z/2)\supseteq \upZ_1(\Pr_r^{\boxtimes};\Z/2)$ w.r.t. the 
standard basis of $\upC_1(\Pr_r^{\boxtimes};\Z/2)$, is the element $A$ of  
$(\Z/2)^{[r-1]^2}$ which is defined by $A\bigl [ x_1y_1 , C_{\ev,r}^{x_1y_1} \bigr ]:=1$, 
$A\bigl [ x_iy_i , C_{\ev,r}^{x_jy_j} \bigr ]:=1$ 
for every $(i,j)\in\bigsqcup_{2\leq \iota \leq r-1}\{(\iota,\iota-1),(\iota,\iota)\}$ 
and $A\bigl [ x_iy_i , C_{\ev,r}^{x_jy_j} \bigr ]:=0$  for every other 
$(i,j)\in \{1,\dotsc,r-1\}^2$. This is a band matrix which in particular is `lower' 
triangular with its main diagonal filled entirely with ones, hence nonsingular. The 
existence of one such minor alone implies the claim in the case $X=\Pr_r^{\boxtimes}$. 
As to the case $X=\Pr_r^{\boxminus}$, due to the similar definition of 
$\mathcal{CB}_{\Pr_r^{\boxminus}}^{(2)}$ compared to $\mathcal{CB}_{\Pr_r^{\boxtimes}}^{(2)}$, 
a proof in this case is obtained if in the first paragraph 
`$\Pr_r^{\boxtimes}$' is replaced by `$\Pr_r^{\boxminus}$', 
`$C_{\ev,r}^{x_1y_1}$' by `$C_{\boxminus,\ev,r}^{x_1y_1}$' and 
`$C_{\ev,r}^{x_iy_i}$' by `$C_{\boxminus,\ev,r}^{x_iy_i}$'.
As to \ref{98786543322154675878987} in the cases $X=\upM_r^{\boxtimes}$ 
(respectively,  $X=\upM_r^{\boxminus}$), due to the similar definition of 
$\mathcal{CB}_{\upM_r^{\boxtimes}}^{(2)}$ compared to $\mathcal{CB}_{\Pr_r^{\boxtimes}}^{(2)}$, 
a proof of these two cases is obtained if in the first paragraph 
`even $r\geq 4$' is replaced by `odd $r\geq 5$', 
`$\Pr_r^{\boxtimes}$' by `$\upM_r^{\boxtimes}$' (respectively, `$\upM_r^{\boxminus}$'), 
`$C_{\ev,r}^{x_1y_1}$' by `$C_{\od,r}^{x_1y_1}$' (respectively, `$C_{\boxminus,\od,r}^{x_1y_1}$'), 
and `$C_{\ev,r}^{x_iy_i}$' by `$C_{\od,r}^{x_iy_i}$' (respectively, 
`$C_{\boxminus,\od,r}^{x_iy_i}$'). This completes the proof 
of \ref{98786543322154675878987} in its entirety. 

As to \ref{98897665642334334544334879768} in the case $X=\Pr_r^{\boxtimes}$, for an 
arbitrary even $r\geq 4$ let $c$ $\in$ 
$\left\langle\mathcal{CB}_{\Pr_r^{\boxtimes}}^{(1)}\right\rangle_{\Z/2}$ $\cap$ 
$\left\langle\mathcal{CB}_{\Pr_r^{\boxtimes}}^{(2)}\right\rangle_{\Z/2}$ be arbitrary. 
Then there exist $(\lambda^{(1)})\in (\Z/2)^{[5]}$ and 
$(\lambda^{(2)})\in (\Z/2)^{[r-1]}$ such that 

{\small
\begin{minipage}[b]{0.5\linewidth}
\begin{enumerate}[label={\rm($\boxtimes$.Su \arabic{*})},leftmargin=5em,start=1]
\item\label{7089878676755443323454554} 
$c = \sum_{1\leq i \leq 5} \lambda_i^{(1)} \ \upc_{C_{\ev,r,i}}$ \quad ,
\end{enumerate}
\end{minipage}
\begin{minipage}[b]{0.5\linewidth}
\begin{enumerate}[label={\rm($\boxtimes$.Su \arabic{*})},leftmargin=5em,start=2]
\item\label{7096433434455676743331} 
$c = \sum_{1\leq i \leq r-1} \lambda_i^{(2)} \ \upc_{C_{\ev,r}^{x_iy_i}}$ \quad .
\end{enumerate}
\end{minipage}
}
where $\upc_{M}$ for some set of edges $M$ denotes the element 
$c\in\upC_1(\Pr_r^{\boxtimes};\Z/2)$ with $\Supp(c)=M$. 
We now show by contradiction that 
$\lambda_1^{(2)} = \dotsm = \lambda_{r-1}^{(2)} = 0$, hence 
$\left\langle\mathcal{CB}_{\Pr_r^{\boxtimes}}^{(1)}\right\rangle_{\Z/2}$ $\cap$ 
$\left\langle\mathcal{CB}_{\Pr_r^{\boxtimes}}^{(2)}\right\rangle_{\Z/2}$ $=$ $\{0\}$. 
To this end, we make the assumption that, on the contrary, 
{\small
\begin{equation}\label{98776434546587978787656554543}
\lambda_i^{(2)} = 1 \hspace{5em}\text{(for at least one $1\leq i \leq r-1$)}\quad .
\end{equation}
}
Drawing on the facts (straightforward to check using the definitions \ref{425486446546434454567565454548} 
and \ref{867554134354675546485657563365}), 
{\small
\begin{enumerate}[label={\rm(F\arabic{*})},leftmargin=6em] 
\item\label{8747854872042828703487934732} 
$\{x_2y_2,x_3y_3,\dotsc,x_{r-1}y_{r-1}\} =  
C_{\ev,r,1}\cap C_{\ev,r,2}\cap 
C_{\ev,r,3}\cap C_{\ev,r,4}\cap C_{\ev,r,5}$ \quad ,
\item\label{87987958798787026702560250506509} 
$x_0x_{r-1}\in C_{\ev,r,1} \cap C_{\ev,r,2}$ \quad , \qquad 
$x_0x_{r-1}\notin C_{\ev,r,3}\cup C_{\ev,r,4}\cup C_{\ev,r,5}$
\quad ,
\item\label{489238928728787934879487974030} 
$y_0y_{r-1}\notin C_{\ev,r,1} \cup C_{\ev,r,2}$ \quad , \qquad 
$y_0y_{r-1}\in C_{\ev,r,3} \cap C_{\ev,r,4} \cap C_{\ev,r,5}$ 
\quad ,
\item\label{9348704187013287987932871248714}
$\{x_2y_2,x_3y_3,\dotsc,x_{r-1}y_{r-1}\}\cap C_{\ev,r}^{x_iy_i} \neq \emptyset$ 
for every $1\leq i \leq r-1$ \quad ,
\item\label{87138791243783287211872732896765} 
$ \{ i\in\{ 1,2,\dotsc,r-1\}\colon x_1y_1\in C_{\ev,r}^{x_iy_i}\} = \{1\}$ \quad ,
\item\label{878713783126782678267216712678326789} 
$\{ i\in \{1,\dotsc,r-1\}\colon zx_1\in C_{\ev,r}^{x_iy_i} \} = \{ 1,\dotsc, r-2\}$ 
\quad ,
\item\label{89870918068789643545684866898} 
$\{ \iota\in\{ 1,2,\dotsc,r-1\}\colon x_iy_i\in C_{\ev,r}^{x_\iota y_\iota}\} = \{ i-1,i \}$ 
for every $2\leq i \leq r-1$ \quad ,
\item\label{987665443565676867675431313313343443}
$\{zy_1,x_0y_0\}\cap C_{\ev,r}^{x_iy_i} = \emptyset$ for every $1\leq i \leq r-1$ 
\quad , 
\item\label{89867675650988978676554546576786} 
$\{x_0x_{r-1},y_0y_{r-1}\}\subseteq C_{\ev,r}^{x_iy_i}$ for every $1\leq i \leq r-2$ 
\quad , 
\item\label{8986775654433443556657688}
$\{x_0x_{r-1},y_0y_{r-1}\}\cap C_{\ev,r}^{x_{r-1}y_{r-1}} = \emptyset$\quad , 
\end{enumerate}
}
we can now reason as follows, distinguishing whether $x_2y_2\in\Supp(c)$ or not: 

\textsl{Case~1.} $x_2y_2\in\Supp(c)$. Then \ref{7089878676755443323454554} 
together with \ref{8747854872042828703487934732} implies that 
$\lvert \{ i\in\{1,\dotsc,5\}\colon \lambda_i^{(1)} = 1 \}\rvert$ is odd, 
and this implies that exactly one of the two numbers 
$\lvert \{ i\in\{1,2\}\colon \lambda_i^{(1)}= 1\} \rvert$ and 
$\lvert \{ i\in\{3,4,5\}\colon \lambda_i^{(1)}= 1\} \rvert$ 
is odd, which combined with \ref{7089878676755443323454554}, 
\ref{87987958798787026702560250506509} and \ref{489238928728787934879487974030} 
implies that $\lvert \{ x_0x_{r-1}, y_0y_{r-1} \} \cap \Supp(c) \rvert = 1$. But this 
contradicts \ref{7096433434455676743331}, \ref{89867675650988978676554546576786} 
and \ref{8986775654433443556657688}, which when taken together imply 
that $\lvert \{ x_0x_{r-1}, y_0y_{r-1} \} \cap \Supp(c) \rvert \in \{0,2\}\not\ni 1$. 
This contradiction proves that Case~1 cannot occur (and we have not used our 
assumption \eqref{98776434546587978787656554543} to arrive at this conclusion). 

\textsl{Case~2.} $x_2y_2\notin\Supp(c)$. From this we deduce 

{\small
\begin{minipage}[b]{0.5\linewidth}
\begin{enumerate}[label={\rm(Co \arabic{*})}]
\item\label{897865654657787889898876}
$zy_1\notin\Supp(c)$\; ,
\item\label{987865334568789089879877869} 
$\lvert \{ i\in\{1,\dotsc,5\}\colon \lambda_i^{(1)} = 1 \}\rvert$ is even\; , 
\item\label{989009887665543434090909} 
$\{x_2y_2,x_3y_3,\dotsc,x_{r-1}y_{r-1}\}\cap\Supp(c)=\emptyset$ \; ,
\item\label{9887675443456565665657675} 
$\lambda_1^{(2)} = \dotsm = \lambda_{r-1}^{(2)} = 1$ \; , 
\end{enumerate}
\end{minipage}
\begin{minipage}[b]{0.4\linewidth}
\begin{enumerate}[label={\rm(Co \arabic{*})},start=5] 
\item\label{98767867566577887989}
$\{x_0x_{r-1},y_0y_{r-1}\}\cap\Supp(c)=\emptyset$\; ,
\item\label{128787927832867932673269189}
$zx_1\notin\Supp(c)$\; ,
\item\label{988676756554546678878989}
$x_1y_1\in\Supp(c)$\; ,
\item\label{8987656544545445545455665} 
$x_0y_0\notin\Supp(c)$\; .
\end{enumerate}
\end{minipage}
}

These claims can be justified thus: \ref{897865654657787889898876} follows 
from \ref{7096433434455676743331} and \ref{987665443565676867675431313313343443}. 
\ref{987865334568789089879877869} follows from combining $x_2y_2\notin\Supp(c)$ with 
\ref{7089878676755443323454554} and \ref{8747854872042828703487934732}. 
\ref{989009887665543434090909} follows from \ref{987865334568789089879877869}, 
\ref{7089878676755443323454554} and \ref{8747854872042828703487934732}. 
\ref{9887675443456565665657675} follows from \ref{989009887665543434090909}, 
\ref{7096433434455676743331}, \ref{9348704187013287987932871248714} and 
\ref{89870918068789643545684866898}, together with our assumption 
\eqref{98776434546587978787656554543}. (At this instance we have learned 
that in \eqref{98776434546587978787656554543}---if it is true---the existential 
quantifier must necessarily hold as a universal quantifier.) 
\ref{98767867566577887989} follows from \ref{9887675443456565665657675}, 
\ref{7096433434455676743331}, \ref{89867675650988978676554546576786}, 
\ref{8986775654433443556657688} and the evenness of $r-2$. 
\ref{128787927832867932673269189} follows from \ref{9887675443456565665657675}, 
\ref{878713783126782678267216712678326789}, and the evenness of
$r-2 = \lvert \{1,\dotsc, r-2\} \rvert$. \ref{988676756554546678878989} follows from 
\ref{7096433434455676743331} and \ref{87138791243783287211872732896765}.  
\ref{8987656544545445545455665} follows from \ref{7096433434455676743331} and 
\ref{987665443565676867675431313313343443}. 

Now from \ref{98767867566577887989} combined with 
\ref{87987958798787026702560250506509} and \ref{489238928728787934879487974030}, 
it follows that \ref{987865334568789089879877869} cannot be true with both 
$n_{1,2}:=\lvert \{ i\in\{1,2\}\colon \lambda_i^{(1)}= 1\} \rvert$ and 
$n_{3,4,5}:=\lvert \{ i\in\{3,4,5\}\colon \lambda_i^{(1)}= 1\} \rvert$ being odd. 
Therefore both $n_{1,2}$ and $n_{3,4,5}$ must be even. 
To finish the proof, we use the abbreviations 
$S_{1,2} := \Supp ( \lambda_1^{(1)} \cdot\upc_{C_{\ev,r,1}} + 
\lambda_2^{(1)} \cdot\upc_{C_{\ev,r,2}})$
and 
$S_{3,4,5} := \Supp( \lambda_3^{(1)} \cdot\upc_{C_{\ev,r,3}} + 
\lambda_4^{(1)} \cdot\upc_{C_{\ev,r,4}} + 
\lambda_5^{(1)} \cdot\upc_{C_{\ev,r,5}} )$, with which we have 
{\small
\begin{equation}\label{776554434243547898989888}
\Supp(c) = S_{1,2}\vartriangle S_{3,4,5}\hspace{5em} 
\text{(symmetric difference)}\quad ,
\end{equation}
}
and distinguish cases according to the value of $n_{1,2}\in\{0,2\}$. 

\textsl{Case~2.1.} $n_{1,2} = 0$. Then in particular 
$x_1y_1\notin S_{1,2}$, $zx_1\notin S_{1,2}$ and $zy_1\notin S_{1,2}$. 

\textsl{Case~2.1.1.} $n_{3,4,5}=0$. This implies that $S_{3,4,5}=\emptyset$, and this 
together with $x_1y_1\notin S_{1,2}$ and \eqref{776554434243547898989888} in 
particular implies $x_1y_1\notin\Supp(c)$, 
contradicting \ref{988676756554546678878989} and proving Case~2.1.1 to be impossible. 

\textsl{Case~2.1.2.} $n_{3,4,5} = 2$. Let us distinguish whether 
$\lambda_5^{(1)}\in\Z/2$ is $0$ or $1$ (the motivation for this being that 
$zy_1\notin S_{1,2}$ and among 
$C_{\ev,r,3}$, $C_{\ev,r,4}$, $C_{\ev,r,5}$ only $C_{\ev,r,5}$ contains $zy_1$, making it 
possible to draw a conclusion from the value of $\lambda_5^{(1)}$). 
If $\lambda_5^{(1)} = 1$, then 
$zy_1\in\Supp(\lambda_5^{(1)}\cdot\upc_{C_{\ev,r,5}})$ and moreover exactly one of 
$\lambda_3^{(1)}$ and $\lambda_4^{(1)}$ is $= 1$. 
Whichever it is, due to $zy_1\notin\Supp(\lambda_3^{(1)}\cdot\upc_{C_{\ev,r,3}})$ and 
$zy_1\notin\Supp(\lambda_4^{(1)}\cdot\upc_{C_{\ev,r,4}})$ it follows that
$zy_1\in S_{3,4,5}$, which combined with $zy_1\notin S_{1,2}$ and 
\eqref{776554434243547898989888} implies $zy_1\in\Supp(c)$, contradicting 
\ref{897865654657787889898876} and proving $\lambda_5^{(1)}= 1$ to be impossible. 
If on the contrary $\lambda_5^{(1)}=0$, then $\lambda_3^{(1)} = \lambda_4^{(1)} = 1$ and 
it follows that $zx_1\in S_{3,4,5}$. 
Being within Case~2.1 we know that $zx_1\notin S_{1,2}$, hence 
in view of \eqref{776554434243547898989888} we may conclude that $zx_1\in\Supp(c)$, 
contradicting \ref{128787927832867932673269189}, proving Case~2.1.2, and therefore 
Case~2.1 as a whole, to be impossible. 

\textsl{Case~2.2.} $n_{1,2} = 2$. 
This implies $x_0y_0\notin S_{1,2}$, $x_1y_1\in S_{1,2}$ and $zx_1\in S_{1,2}$. 
Again it remains to consider the possibilities for 
$n_{3,4,5} \in \{0,1,2,3\}$ to be even. 

\textsl{Case~2.2.1.} $n_{3,4,5} = 0$. Then $S_{3,4,5}=\emptyset$, and this together 
with $zx_1\in S_{1,2}$ and \eqref{776554434243547898989888} in particular implies 
$zx_1\in \Supp(c)$, contradicting \ref{128787927832867932673269189} and proving 
Case~2.2.1 to be impossible. 

\textsl{Case~2.2.2.} $n_{3,4,5} = 2$. Again we analyse this case by distinguishing 
whether $\lambda_5^{(1)}\in\Z/2$ is $0$ or $1$. If $\lambda_5^{(1)}=1$, then 
exactly one of $\lambda_3^{(1)}$ and $\lambda_4^{(1)}$ is $=1$ and, whichever it is, 
it follows that $x_1y_1\in S_{3,4,5}$. Being within Case~2.2. we 
know $x_1y_1\in S_{1,2}$, hence in view of \eqref{776554434243547898989888} it 
follows that $x_1y_1\notin\Supp(c)$, contradicting \ref{988676756554546678878989} 
and proving $\lambda_5^{(1)}=1$ to be impossible. If on 
the contrary $\lambda_5^{(1)} = 0$, then $\lambda_3^{(1) } = \lambda_4^{(1)} = 1$ and  
it follows that $x_0y_0\in S_{3,4,5}$. Being within Case~2.2 we know that 
$x_0y_0\in S_{1,2}$ which in view of \eqref{776554434243547898989888} implies 
$x_0y_0\in\Supp(c)$, contradicting \ref{8987656544545445545455665} and proving 
$\lambda_5^{(1)} = 0$ to be impossible. This proves Case~2.2.2, and therefore also 
Case~2.2 and the entire Case~2, to be impossible. Since the mutually exclusive 
Cases~1 and 2 both lead to contradictions, the assumption 
\eqref{98776434546587978787656554543} is false, completing the proof of 
\ref{98897665642334334544334879768} for $X=\Pr_r^{\boxtimes}$.  

As to \ref{98897665642334334544334879768} in the case $X=\upM_r^{\boxtimes}$, the 
proof given 
for the case $X=\Pr_r^{\boxtimes}$ can be repeated with the appropriate minor 
changes to obtain a proof in the case $X=\upM_r^{\boxtimes}$, these changes being 
the following: first of all, the statements 
\ref{8747854872042828703487934732}--\ref{8986775654433443556657688} have been 
chosen in such a way 
that each of \ref{8747854872042828703487934732}--\ref{8986775654433443556657688} 
becomes a true statement about the set $\mathcal{CB}_{\upM_{r}^{\boxtimes}}^{(2)}$ if exactly 
the following changes are made in 
\ref{8747854872042828703487934732}--\ref{8986775654433443556657688}: `$\ev$' is to 
be replaced by `$\od$', `$x_0x_{r-1}$' is to be replaced by `$x_0y_{r-1}$' (all 
occurrences, i.e. in \ref{87987958798787026702560250506509}, 
in \ref{89867675650988978676554546576786} and in \ref{8986775654433443556657688}), 
`$y_0y_{r-1}$' is to be replaced by `$y_0x_{r-1}$' (all occurrences, 
i.e. in \ref{489238928728787934879487974030}, 
in \ref{89867675650988978676554546576786} and in \ref{8986775654433443556657688}). 
With the references to 
\ref{8747854872042828703487934732}--\ref{8986775654433443556657688} now referring 
to the statements thus modified, the only thing to be done in the entire 
remaining proof of the case $X=\Pr_r^{\boxtimes}$ (in order to arrive at a proof of 
the case $X=\upM_r^{\boxtimes}$) is to replace `$x_0x_{r-1}$' by `$x_0y_{r-1}$' and 
`$y_0y_{r-1}$' by `$y_0x_{r-1}$' at all three occurences of these edges (twice in 
Case~1, once in \ref{98767867566577887989}), and moreover to replace 
`$\ev$' by `$\od$'. This completes the proof of \ref{98897665642334334544334879768} 
for $X=\upM_r^{\boxtimes}$. 

As to \ref{98897665642334334544334879768} in the case $X = \Pr_r^{\boxminus}$, for an 
arbitrary even $r\geq 4$ let $c$ $\in$ 
$\left\langle\mathcal{CB}_{\Pr_r^{\boxminus}}^{(1)}\right\rangle_{\Z/2}$ $\cap$ 
$\left\langle\mathcal{CB}_{\Pr_r^{\boxminus}}^{(2)}\right\rangle_{\Z/2}$ be arbitrary. 
Then there are $(\lambda^{(1)})\in(\Z/2)^{[5]}$ and $(\lambda^{(2)})\in(\Z/2)^{[r-1]}$ 
such that 

{\small
\begin{minipage}[b]{0.5\linewidth}
\begin{enumerate}[label={\rm($\boxminus$.Su \arabic{*})},leftmargin=5em,start=1]
\item\label{88766655465787654554656565} 
$c = \sum_{1\leq i \leq 5} \lambda_i^{(1)} \cdot\upc_{C_{\boxminus,\ev,r,i}}$ \quad ,
\end{enumerate}
\end{minipage}
\begin{minipage}[b]{0.5\linewidth}
\begin{enumerate}[label={\rm($\boxminus$.Su \arabic{*})},leftmargin=5em,start=2]
\item\label{9889755468998765466667877777777777} 
$c = \sum_{1\leq i \leq r-1} \lambda_i^{(2)}\cdot\upc_{C_{\boxminus,\ev,r}^{x_iy_i}}$\quad .
\end{enumerate}
\end{minipage}
}
where $C_M$ for some set of edges $M$ denotes the unique element 
$c\in\upC_1(\Pr_r^{\boxminus};\Z/2)$ with $\Supp(c)=M$. We will show directly (this 
time we will not have any use for making the 
assumption \eqref{98776434546587978787656554543}) that $c=0$, hence 
$\left\langle\mathcal{CB}_{\Pr_r^{\boxminus}}^{(1)}\right\rangle_{\Z/2}$ $\cap$ 
$\left\langle\mathcal{CB}_{\Pr_r^{\boxminus}}^{(2)}\right\rangle_{\Z/2}$ $=$ $\{0\}$. 
We can now use the evident facts 
{\scriptsize
\begin{enumerate}[label={\rm($\boxminus$.F\arabic{*})},leftmargin=6em] 
\item\label{8976745357688998876545} 
$z'z''\in \bigcap_{1\leq i \leq r-1} C_{\boxminus,\ev,r}^{x_iy_i}$ \quad ,
\item\label{897665434345467887987687879} 
$\{ x_0x_{r-1}, y_0y_{r-1}\} \subseteq C_{\boxminus,\ev,r}^{x_iy_i}$ for every 
$1\leq i \leq r-2$ \quad , 
\item\label{907665554476878997654} 
$z'z''\notin C_{\boxminus,\ev,r,1}$, 
$z'z''\in C_{\boxminus,\ev,r,2}$, 
$z'z''\notin C_{\boxminus,\ev,r,3}$, 
$z'z''\in C_{\boxminus,\ev,r,4}$, 
$z'z''\notin C_{\boxminus,\ev,r,5}$ \quad , 
\item\label{87867665555656565567}
for even $r\geq 4$, the only circuit among the circuits 
in $\mathcal{CB}_{\Pr_r^{\boxminus}}^{(2)}$ to 
contain $x_0z''$ is $C_{\boxminus,\ev,r}^{x_{r-1}y_{r-1}}$ \quad ,
\item\label{87866565454434454}
for even $r\geq 4$, the only circuit among the circuits 
in $\mathcal{CB}_{\Pr_r^{\boxminus}}^{(1)}\sqcup\mathcal{CB}_{\Pr_r^{\boxminus}}^{(2)}$ to 
contain $y_1z''$ is $C_{\boxminus,\ev,r,5}$ \quad ,
\item\label{986754354798765434657788787} 
for even $r\geq 4$, the only circuit among the circuits in 
$\mathcal{CB}_{\Pr_r^{\boxminus}}^{(1)}\sqcup\mathcal{CB}_{\Pr_r^{\boxminus}}^{(2)}$ to 
contain $x_0y_0$ is $C_{\boxminus,\ev,r,4}$ \quad , 
\item\label{879564534456787644554567876652354}
for even $r\geq 4$, the only circuits among the circuits 
in $\mathcal{CB}_{\Pr_r^{\boxminus}}^{(1)}\sqcup\mathcal{CB}_{\Pr_r^{\boxminus}}^{(2)}$
to contain an odd number of the two edges $x_0x_{r-1}$ and $y_0y_{r-1}$ are 
the two circuits $C_{\boxminus,\ev,r,3}$ and $C_{\boxminus,\ev,r,5}$ \quad ,
\end{enumerate}
}
to argue as follows. First of all, we immediately conclude that  
\begin{enumerate}[label={\rm($\boxminus$.Co \arabic{*})},leftmargin=6em] 
\item\label{9876656546578799889889} $\lambda_4^{(1)} = 0$ because 
of \ref{88766655465787654554656565} and \ref{9889755468998765466667877777777777} 
combined with \ref{986754354798765434657788787}\quad , 
\item\label{88767565454767897978} $\lambda_5^{(1)}=0$ because 
of \ref{88766655465787654554656565} and \ref{9889755468998765466667877777777777} 
combined with \ref{87866565454434454} \quad . 
\end{enumerate}

\textsl{Case~1.} $\lvert \{i\in \{1,\dotsc,r-1\}\colon \lambda_i^{(2)} = 1 \}\rvert$ 
is odd. Then \ref{9889755468998765466667877777777777} together 
with \ref{8976745357688998876545} implies $z'z''\in\Supp(c)$. Therefore, and 
because of \ref{907665554476878997654}, it follows that exactly one of 
$\lambda_2^{(1)}$ and $\lambda_4^{(1)}$ is equal to $1$, hence $\lambda_2^{(1)}=1$ 
because of \ref{9876656546578799889889}. Now let us consider $\lambda_3^{(1)}$. It 
cannot be true that $\lambda_3^{(1)}=1$, since then 
\ref{879564534456787644554567876652354} implies $\lambda_5^{(1)}=1$, contradicting 
\ref{88767565454767897978}. Thus we may assume that $\lambda_3^{(1)}=0$. This 
implies $x_1y_1\in\Supp(c)$ due to \ref{88766655465787654554656565}, 
$\lambda_2^{(1)}=1$, \ref{9876656546578799889889} and the fact that .
for every even $r\geq 4$, the only circuits among the circuits 
in $\mathcal{CB}_{\Pr_r^{\boxminus}}^{(1)}$ to contain $x_1y_1$ are 
$C_{\boxminus,\ev,r,2}$ and $C_{\boxminus,\ev,r,3}$. 
Among the coefficients 
$\lambda_i^{(1)}$, $1\leq i \leq 5$, only the value of $\lambda_1^{(1)}$ is not 
yet known to us. 

\textsl{Case~1.1.} $\lambda_1^{(1)} = 0$. Then 
$z'y_0\in C_{\boxminus,\ev,r,2}$, $\lambda_2^{(1)}=1$ and 
$\lambda_1^{(1)}=\lambda_3^{(1)}=\lambda_4^{(1)}=\lambda_5^{(1)}=0$ together 
with \ref{88766655465787654554656565} imply that $z'y_0\in\Supp(c)$. 
Since for every even $r\geq 4$, the only circuit among the circuits in 
$\mathcal{CB}_{\Pr_r^{\boxminus}}^{(2)}$ to contain $y_0z'$ is $C_{\boxminus,\ev,r}^{x_{r-1}y_{r-1}}$, 
from $z'y_0\in\Supp(c)$ it follows that $\lambda_{r-1}^{(2)}=1$. Being within Case~1, 
this implies that 
$\lvert \{i\in \{1,\dotsc,r-2\}\colon \lambda_i^{(2)} = 1 \}\rvert$ is even, 
which by \ref{897665434345467887987687879} implies that 
$\{ x_0x_{r-1}, y_0y_{r-1} \}\cap \Supp(c) = \emptyset$; 
but $\{ x_0x_{r-1}, y_0y_{r-1} \}\subseteq C_{\boxminus,\ev,r,2}$ together with 
\ref{88766655465787654554656565}, 
$\lambda_1^{(1)}=\lambda_3^{(1)}=\lambda_4^{(1)}=\lambda_5^{(1)}=0$ and 
$\lambda_2^{(1)}=1$ implies that, on the contrary, 
$\{ x_0x_{r-1}, y_0y_{r-1} \}\subseteq \Supp(c)$. This contradiction proves Case~1.1 
to be impossible. 

\textsl{Case~1.2.} $\lambda_1^{(1)} = 1$. Then 
$\lambda_3^{(1)}=\lambda_4^{(1)}=\lambda_5^{(1)}=0$, $\lambda_1^{(1)}=\lambda_2^{(1)}=1$ 
and \ref{88766655465787654554656565} together imply $x_0z''\notin\Supp(c)$. 
Because of \ref{87867665555656565567}, this implies $\lambda_{r-1}^{(2)}=0$. Being 
within Case~1, it follows that  
$\lvert \{i\in \{1,\dotsc,r-2\}\colon \lambda_i^{(2)} = 1 \}\rvert$ is even, 
hence \ref{897665434345467887987687879} together 
with \ref{9889755468998765466667877777777777} implies that 
$\{ x_0x_{r-1}, y_0y_{r-1} \}\cap\Supp(c) = \emptyset$; but 
$\lambda_3^{(1)}=\lambda_4^{(1)}=\lambda_5^{(1)}=0$, $\lambda_1^{(1)}=\lambda_2^{(1)}=1$, 
and \ref{9889755468998765466667877777777777}, together with the facts that 
$\{x_0x_{r-1},y_{r-1}\}\cap C_{\boxminus,r,1} = \emptyset$ 
and $\{x_0x_{r-1},y_{r-1}\} \subseteq C_{\boxminus,r,2}$  imply 
$\{ x_0x_{r-1}, y_0y_{r-1} \} \subseteq \Supp(c)$, contradiction. Therefore 
Case~1.2 is impossible, too. 

This proves the entire Case~1 to be impossible. 

\textsl{Case~2.} $\lvert \{i\in\{1,\dotsc,r-1\}\colon\lambda_i^{(2)}=1 \}\rvert$ 
is even. Then \ref{9889755468998765466667877777777777} together 
with \ref{8976745357688998876545} imply $z'z''\notin\Supp(c)$, hence in view 
of \ref{907665554476878997654} it follows that 
either $\lambda_2^{(1)}=\lambda_4^{(1)}=0$ or $\lambda_2^{(1)}=\lambda_4^{(1)}=1$, 
the latter being impossible because of \ref{9876656546578799889889}. 
Therefore, $\lambda_2^{(1)}=\lambda_4^{(1)}=0$. 

\textsl{Case~2.1.} $\lambda_3^{(1)} =1 $. This, together 
with \ref{88766655465787654554656565}, \ref{9889755468998765466667877777777777}, 
\ref{879564534456787644554567876652354} and the fact that every 
$C\in \{C_{\boxminus,\ev,r}^{x_iy_i}\colon 1\leq i \leq r-1 \}$ contains an even number 
of the edges $x_0x_{r-1}$ and $y_0y_{r-1}$, implies that we must have 
$\lambda_5^{(1)}=1$, contradicting \ref{88767565454767897978}.  

\textsl{Case~2.2.} $\lambda_3^{(1)}=0$. Then \ref{88766655465787654554656565}, 
$\lambda_2^{(1)}=0$ and the fact that $C_{\boxminus,r,2}$ and $C_{\boxminus,r,3}$ are 
the only circuits among $C_{\boxminus,r,1},\dotsc,C_{\boxminus,r,5}$ to contain 
$x_1y_1$ imply that $x_1y_1\notin\Supp(c)$. Hence from
\ref{9889755468998765466667877777777777}, together with the fact that 
for every even $r\geq 4$, the only circuit among the circuits in 
$\mathcal{CB}_{\Pr_r^{\boxminus}}^{(2)}$ to contain $x_1y_1$ is $C_{\boxminus,\ev,r}^{x_1y_1}$, 
it follows that $\lambda_1^{(2)}=0$. Now let us consider $\lambda_{r-1}^{(2)}$. If we 
would have $\lambda_{r-1}^{(2)}=1$, then---being within Case~2---the number 
$\lvert \{i\in\{2,\dotsc,r-2\}\colon\lambda_i^{(2)}=1 \}\rvert$ is odd, hence 
$\{ x_0x_{r-1}, y_0y_{r-1}\}\subseteq \Supp(c)$ 
by \ref{9889755468998765466667877777777777} and \ref{897665434345467887987687879}; 
but this contradicts \ref{88766655465787654554656565}, 
$\lambda_2^{(1)}=\lambda_3^{(1)}=0$, 
$\{ x_0x_{r-1}, y_0y_{r-1}\}\cap C_{\boxminus,r,1}=\emptyset$, 
$\{ x_0x_{r-1}, y_0y_{r-1}\}\cap C_{\boxminus,r,4}=\emptyset$ and 
$\{ x_0x_{r-1}, y_0y_{r-1}\}\cap C_{\boxminus,r,5}=\{x_0x_{r-1}\}$, which when taken 
together imply 
$\{ x_0x_{r-1}, y_0y_{r-1}\}\cap\Supp(c) \in \bigl\{\emptyset, \{x_0,x_{r-1}\}\bigr\}$. 
Therefore we may assume $\lambda_{r-1}^{(2)}=0$. Then---being within Case~2---the 
number $\lvert \{i\in\{2,\dotsc,r-2\}\colon\lambda_i^{(2)}=1 \}\rvert$ is even, 
hence \ref{9889755468998765466667877777777777} and 
\ref{897665434345467887987687879} imply that 
$\{x_0x_{r-1}, y_0y_{r-1}\}\cap \Supp(c) = \emptyset$. Since among 
$C_{\boxminus,r,1},\dotsc,C_{\boxminus,r,5}$ only $C_{\boxminus,r,5}$ contains $x_0x_{r-1}$, 
this implies $\lambda_5^{(1)}=0$. We now know that 
$\lambda_2^{(2)}=\lambda_3^{(2)}=\lambda_4^{(2)}=\lambda_5^{(2)}=0$. Therefore, if we 
would have $\lambda_1^{(1)}=1$, then $x_1z''\in \Supp(c)$, contradicting the fact 
that \ref{9889755468998765466667877777777777}, $\lambda_{r-1}^{(2)}=0$, the evenness of 
$\lvert \{i\in\{2,\dotsc,r-2\}\colon\lambda_i^{(2)}=1 \}\rvert$ and the property 
$x_1z''\in C_{\boxminus,\ev,r}^{x_iy_i}$ for every $1\leq i \leq r-2$
together imply $x_1z''\notin\Supp(c)$. Thus, 
$\lambda_1^{(1)}=\lambda_2^{(2)}=\lambda_3^{(2)}=\lambda_4^{(2)}=\lambda_5^{(2)}=0$, hence 
$c=0$ by \ref{88766655465787654554656565}, completing the proof of 
$\left\langle\mathcal{CB}_{\Pr_r^{\boxminus}}^{(1)}\right\rangle_{\Z/2}$ $\cap$ 
$\left\langle\mathcal{CB}_{\Pr_r^{\boxminus}}^{(2)}\right\rangle_{\Z/2}$ $=$ $\{0\}$ in Case~2. 
This completes the proof of \ref{98897665642334334544334879768} in the 
case $X=\Pr_r^{\boxminus}$. 
As to \ref{98897665642334334544334879768} in the case $X=\upM_r^{\boxminus}$, again 
the proof of the case $X=\Pr_r^{\boxminus}$ can be repeated with the necessary small 
changes, namely: throughout, `$\Pr_r$' is to be replaced by `$\upM_r$', 
`$\ev$' by `$\od$', `$x_0x_{r-1}$' by `$x_0y_{r-1}$', `$y_0y_{r-1}$' by `$y_0x_{r-1}$'. 
Afterwards, \ref{8976745357688998876545}---\ref{897665434345467887987687879} are 
still true and the proof given for the case $X=\Pr_r^{\boxminus}$ has become a proof 
for the case $X=\upM_r^{\boxminus}$. The proof of 
Lemma~\ref{98897665642334334544334879768} is now complete.

As to \ref{9iu8u7yr5f4d3s44d5789u}.\ref{3287923478248979847238947238}, note that 
$\dim_{\Z/2} \upZ_1(\Pr_r^{\boxtimes};\Z/2) = (3r+4) - (2r+1) + 1 = r+4$, and that 
\ref{38794320986754435465687}, \ref{98786543322154675878987} and 
\ref{98897665642334334544334879768} in the case $X=\Pr_r^{\boxtimes}$ together imply 
that for even $r\geq 4$ we have 
$\dim_{\Z/2} \left ( \left\langle \mathcal{CB}_{\Pr_r^{\boxtimes}}^{(1)}\right\rangle_{\Z/2} + 
\left\langle \mathcal{CB}_{\Pr_r^{\boxtimes}}^{(2)}\right\rangle_{\Z/2} \right ) = r+4$. 
Therefore the set $\left\langle \mathcal{CB}_{\Pr_r^{\boxtimes}}^{(1)}\right\rangle_{\Z/2} + 
\left\langle \mathcal{CB}_{\Pr_r^{\boxtimes}}^{(2)}\right\rangle_{\Z/2}$ is a $\Z/2$-linear 
subspace of $\upZ_1(\Pr_r^{\boxtimes};\Z/2)$ having the same dimension as the ambient 
space. In a vector space this implies equality as a set. This 
proves \ref{3287923478248979847238947238}. An entirely analogous argument 
proves \ref{9iu8u7yr5f4d3s44d5789u}.\ref{935872702348723487932893248934}. 

As to \ref{9iu8u7yr5f4d3s44d5789u}.\ref{349078364454938747474747329}, note that 
$\dim_{\Z/2} \upZ_1(\Pr_r^{\boxminus};\Z/2) = (3r+6) - (2r+2) + 1 = r+5$ and that 
\ref{38794320986754435465687}, \ref{98786543322154675878987} and 
\ref{98897665642334334544334879768} in the case $X=\Pr_r^{\boxtimes}$ together imply 
that for even $r\geq 4$ we have 
$\dim_{\Z/2} \left ( \left\langle \mathcal{CB}_{\Pr_r^{\boxtimes}}^{(1)}\right\rangle_{\Z/2} + 
\left\langle \mathcal{CB}_{\Pr_r^{\boxtimes}}^{(2)}\right\rangle_{\Z/2} \right ) = r+4$. 
Since $\dim_K(V/U) = \dim_K(V) - \dim_K(U)$ for finite-dimensional 
$K$-vectors spaces $U\subseteq V$, this implies \ref{349078364454938747474747329}. 
An entirely analogous argument 
proves \ref{9iu8u7yr5f4d3s44d5789u}.\ref{1364867134086734108413408746}. 

As to \ref{9iu8u7yr5f4d3s44d5789u}.\ref{83276268932879348793254328793}, this 
claim follows quickly from \ref{349078364454938747474747329}: it suffices to note 
that in $\Pr_r^{\boxminus}$ there actually exists a circuit of length $f_0(\cdot)-1$. 
Since $f_0(\Pr_r^{\boxminus})=f_0(\upM_r^{\boxminus})=r+4$ is even for even $r$, and 
since the support of the sum of two circuits of even length is an 
edge-disjoint union of circuits of even length, any circuit of length $f_0(\cdot)-1$ 
in $\Pr_r^{\boxminus}$  is not contained in 
$\left\langle \mathcal{CB}_{\Pr_r^{\boxtimes}}^{(1)}\right\rangle_{\Z/2} + 
\left\langle \mathcal{CB}_{\Pr_r^{\boxtimes}}^{(2)}\right\rangle_{\Z/2}$, hence after adding 
this circuit to the set 
$\mathcal{CB}_{\Pr_r^{\boxtimes}}^{(1)}\sqcup \mathcal{CB}_{\Pr_r^{\boxtimes}}^{(2)}$, 
the $\Z/2$-linear span has dimension 
$(r+4)+1 = r+5 = \dim_{\Z/2} \upZ_1(\Pr_r^{\boxminus};\Z/2)$, proving 
\ref{83276268932879348793254328793}, since that finite-dimensional vector 
spaces do not contain proper subspaces of the same dimension. An entirely analogous 
argumentation 
proves \ref{9iu8u7yr5f4d3s44d5789u}.\ref{328238346786839328743267237773}, this time 
using \ref{1364867134086734108413408746}.

We have now proved \ref{3w44345465566565565r56}--\ref{676776hy7h77h7767765656657}: 
property\ref{3w44345465566565565r56} follows from \ref{3287923478248979847238947238} 
(which is equivalent to $\Pr_r^{\boxtimes}\in\mathrm{cd}_0\mathcal{C}_{\{f_0(\cdot)\}}$), 
\ref{t6t67y8u899o8987g6g678jik90o} and Definition~\ref{def:monotonegraphpropertiesGlzetaandBGlzeta}.\ref{it:def:monotonegraphpropertiesGlzetaandBGlzeta:Glzeta}; 
property \ref{8767856454354e54565676767} follows 
from \ref{935872702348723487932893248934} 
(which is equivalent to $\upM_r^{\boxtimes}\in\mathrm{cd}_0\mathcal{C}_{\{f_0(\cdot)\}}$), 
\ref{89878675643545465r5556} and Definition~\ref{def:monotonegraphpropertiesGlzetaandBGlzeta}.\ref{it:def:monotonegraphpropertiesGlzetaandBGlzeta:Glzeta}; 
property \ref{7667545344edr65656554r54d} follows 
from \ref{349078364454938747474747329} 
(which is equivalent to $\Pr_r^{\boxminus}\in\mathrm{cd}_1\mathcal{C}_{\{f_0(\cdot)\}}$), 
\ref{32671163432867327682} and Definition~\ref{def:monotonegraphpropertiesGlzetaandBGlzeta}.\ref{it:def:monotonegraphpropertiesGlzetaandBGlzeta:Glzeta}; 
property \ref{6756667898887y7yy7765544544ew} follows 
from \ref{1364867134086734108413408746} 
(which is equivalent to $\upM_r^{\boxminus}\in\mathrm{cd}_1\mathcal{C}_{\{f_0(\cdot)\}}$), 
\ref{0o978986765565454454e} and Definition~\ref{def:monotonegraphpropertiesGlzetaandBGlzeta}.\ref{it:def:monotonegraphpropertiesGlzetaandBGlzeta:Glzeta}; 
property \ref{87876665r5r5r5665656565} follows 
from \ref{83276268932879348793254328793} (which is equivalent to 
$\Pr_r^{\boxminus}\in\mathrm{cd}_0\mathcal{C}_{\{f_0(\cdot)-1,f_0(\cdot)\}}$), 
\ref{32671163432867327682} and Definition~\ref{def:monotonegraphpropertiesGlzetaandBGlzeta}.\ref{it:def:monotonegraphpropertiesGlzetaandBGlzeta:Glzeta};  
property \ref{676776hy7h77h7767765656657} follows follows 
from \ref{328238346786839328743267237773} (which is equivalent to 
$\upM_r^{\boxminus}\in\mathrm{cd}_0\mathcal{C}_{\{f_0(\cdot)-1,f_0(\cdot)\}}$), 
\ref{0o978986765565454454e} and Definition~\ref{def:monotonegraphpropertiesGlzetaandBGlzeta}.\ref{it:def:monotonegraphpropertiesGlzetaandBGlzeta:Glzeta}. 

As to \ref{it:roughbandwidthstatements}, the bandwidth of any of $\upC_n^2$, $\CL_r$, 
$\Pr_r^{\boxtimes}$, $\Pr_r^{\boxminus}$, $\upM_r^{\boxtimes}$ and $\upM_r^{\boxminus}$ is 
constant, i.e. does not grow with $r$ or $n$. Therefore \ref{it:roughbandwidthstatements} 
is true in stronger form than is stated here. Since knowing the exact bandwidths 
would profit us nothing given the proof technology that is available at present, 
knowing the statement \ref{it:roughbandwidthstatements} is enough. To prove it, we 
employ a general characterization \cite[Theorem~8]{MR2644412} of 
low-bandwidth graphs due to 
J.~B{\"o}ttcher, K.~P.~Pruessmann, A.~Taraz and A.~W{\"u}rfl. This characterization 
allows us to prove the smallness of the bandwidth for each of the rather different 
graphs $\upC_n^2$, $\CL_r$, $\Pr_r^{\boxtimes}$, $\Pr_r^{\boxminus}$, $\upM_r^{\boxtimes}$ 
and $\upM_r^{\boxminus}$ without any close attention to the specifics of these 
graphs---simply by exhibiting small separators: in $C_n^2$ there does not exist any 
edge between the two sets 
$A := \{\overline{0},\overline{1},\dotsc,\overline{\lfloor\tfrac{n}{2}\rfloor-2}\}$ 
and $B:=\{\overline{\lfloor\tfrac{n}{2}\rfloor+1},\dotsc,\overline{n-3}\}$, and since 
both $\lvert A\rvert$ and $\lvert B\rvert$ are $\leq\tfrac23 f_0(C_n^2)$, the 
existence of the separator 
$S:=\{\overline{\lfloor\tfrac{n}{2}\rfloor-1}, \overline{\lfloor\tfrac{n}{2}\rfloor},
\overline{n-2}, \overline{n-1}\}$ implies that the separation number (in the sense 
of \cite[Definition~2]{MR2644412}) of $C_n^2$ is at most $4$. 
The claim \ref{it:roughbandwidthstatements} in the case of $X=\CL_r$ now follows 
by \cite[Theorem~8, equivalence (2) $\Leftrightarrow$ (4)]{MR2644412}. 
To prove the case $X=\CL_r$ of \ref{it:roughbandwidthstatements}, in the first 
sentence of this paragraph use 
`$A := \bigsqcup_{1\leq i\leq \lfloor \frac{r}{2} \rfloor - 1} \{ a_i, b_i \}$', 
`$B := \bigsqcup_{\lfloor \frac{r}{2} \rfloor + 1\leq i \leq r-1} \{ a_i, b_i \}$' 
and `$S:=\{a_0,b_0,a_{\lfloor \frac{r}{2} \rfloor},b_{\lfloor \frac{r}{2} \rfloor}\}$'. 
To prove the cases $X\in\{\Pr_r^{\boxtimes},\upM_r^{\boxtimes}\}$ of 
\ref{it:roughbandwidthstatements}, in the first sentence of this paragraph use 
`$A:=\{z\}\sqcup\bigsqcup_{1\leq i\leq \lfloor \frac{r}{2} \rfloor - 1} \{ x_i, y_i \}$', 
`$B := \bigsqcup_{\lfloor \frac{r}{2} \rfloor + 1\leq i \leq r-1} \{ x_i, y_i \}$' 
and `$S:=\{a_0,b_0,a_{\lfloor \frac{r}{2} \rfloor},b_{\lfloor \frac{r}{2} \rfloor}\}$'. 
To prove the cases $X\in\{\Pr_r^{\boxminus},\upM_r^{\boxminus}\}$ 
of \ref{it:roughbandwidthstatements}, use $B$ and $S$ as in the preceding sentence but 
`$A:=\{z',z''\}\sqcup\bigsqcup_{1\leq i\leq \lfloor \frac{r}{2} \rfloor - 1} \{ x_i, y_i \}$'. 
This proves the statement about the bandwidth in \ref{it:roughbandwidthstatements}, for every 
$Y \in\{C_n^2,\CL_r,\Pr_r^{\boxtimes},\Pr_r^{\boxminus},\upM_r^{\boxtimes},\upM_r^{\boxminus}\}$. 

As to the additional claims concerning 
$Y \in\{\Pr_r^{\boxtimes},\Pr_r^{\boxminus},\upM_r^{\boxtimes},\upM_r^{\boxminus}\}$, we 
explicitly give suitable maps $\upb_Y$ and $\uph_Y$ (thus for 
$\Pr_r^{\boxtimes}$, $\Pr_r^{\boxminus}$, $\upM_r^{\boxtimes}$, $\upM_r^{\boxminus}$ giving 
another proof of the small bandwidth). 

As to $Y=\Pr_r^{\boxtimes}$, for every even $r\geq 4$, the map $\upb_Y$ 
defined by $z\mapsto 1$, $x_0\mapsto 2$
$x_i\mapsto 4i$ for $1\leq i \leq\lfloor\tfrac{r}{2}\rfloor$, 
$x_i\mapsto 4(r-i) + 2$ for $\lfloor\tfrac{r}{2}\rfloor+1\leq i \leq r-1$, 
$y_0\mapsto 3$, 
$y_i\mapsto 4i + 1$ for $1\leq i \leq\lfloor\tfrac{r}{2}\rfloor$,  and 
$y_i\mapsto 4(r-i) + 3$ for $\lfloor\tfrac{r}{2}\rfloor+1\leq i \leq r-1$ is a 
bandwidth-$4$-labelling of $\Pr_r^{\boxtimes}$. Moreover, the map $\uph_Y$ defined by 
$z\mapsto 0$, $x_i\mapsto 1$ and $y_i\mapsto 2$ for even 
$0\leq i \leq r-1$, $x_i\mapsto 2$ and $y_i\mapsto 1$ for odd $0\leq i \leq r-1$, 
is a $3$-colouring of $\Pr_r^{\boxtimes}$ which for every $r$ large enough to have 
simultaneously $\beta f_0(Y) = \beta (2r+1) \geq 1 = \lvert \uph_Y^{-1}(0) \rvert$ and 
$8\cdot 2\cdot \beta\cdot f_0(Y) = 16\beta(2r+1)\geq 2$ obviously satisfies the 
requirement in Theorem~\ref{thm:BoettcherSchachtTaraz2009} of being 
$(8\cdot 2\cdot \beta\cdot f_0(Y),4\cdot 2\cdot \beta\cdot f_0(Y))$-zero-free 
w.r.t. $\upb_Y$ and having $\lvert\uph_Y^{-1}(0)\rvert \leq \beta f_0(Y)$. This proves 
\ref{it:roughbandwidthstatements} for $Y=\Pr_r^{\boxtimes}$. 

As to $Y=\upM_r^{\boxtimes}$, the same map $\upb_Y$ that was 
defined at the beginning of the preceding paragraph is (this being the reason for 
having used $\lfloor\cdot\rfloor$ despite even $r$) 
a bandwidth-$5$-labelling of $\upM_r^{\boxtimes}$ (which has bandwidth $4$, by the 
way), for every odd $r\geq 5$. Likewise, the same map $\uph_Y$ defined in the 
preceding paragraph is a $3$-colouring of $\upM_r^{\boxtimes}$ for which concerning 
$\lvert\uph_Y^{-1}(0)\rvert$ and zero-freeness w.r.t. $\upb_Y$ exactly the same can 
be said as in the previous paragraph. This proves \ref{it:roughbandwidthstatements} 
for $Y=\upM_r^{\boxtimes}$. 

As to $Y = \Pr_r^{\boxminus}$, for every even $r\geq 4$, the map $\upb_Y$ defined by 
$z'\mapsto 1$, $z''\mapsto 2$, $x_0\mapsto 3$, $y_0\mapsto 4$, 
$x_i\mapsto 4i+1$ and $y_i\mapsto 4i+2$ for $1\leq i\leq\lfloor\tfrac{r}{2}\rfloor$, 
$x_i\mapsto 4(r-i) + 3$ and $y_i\mapsto 4(r-i) + 4$ for 
$\lfloor\tfrac{r}{2}\rfloor + 1\leq i \leq r-1$ 
is a bandwidth-$5$-labelling of $\Pr_r^{\boxminus}$. Moreover, the map $\uph_Y$ defined 
by $z'\mapsto 2$, $z''\mapsto 0$, 
$x_0\mapsto 1$, $y_0\mapsto 2$, $x_1\mapsto 0$, $y_1\mapsto 1$, 
$x_i\mapsto 1$ and $y_i\mapsto 2$ for even $2\leq i \leq r-1$, and 
$x_i\mapsto 2$ and $y_i\mapsto 1$ for odd $2\leq i \leq r-1$ is a 
$3$-colouring of $\Pr_r^{\boxminus}$. In view of 
$\lvert\uph_Y^{-1}(0)\rvert =2$ and in particular in view of the fact that 
$\upb(\uph_Y^{-1}(0)) = \{2,5\}$ for \emph{every} even $r\geq 4$ (i.e. the distance 
along the bandwidth-$5$-labelling of the two $0$-labelled vertices is 
constantly $3$, i.e. independent of $f_0(Y)$), it is obvious that $\uph_Y$ is 
$(8\cdot 2\cdot \beta\cdot f_0(Y),4\cdot 2\cdot \beta\cdot f_0(Y))$-zero-free 
w.r.t. $\upb_Y$, provided that $r$ is large enough to have 
$4\cdot 2\cdot \beta\cdot f_0(Y) = 8\beta (2r+2)\geq 5$ (when testing the 
zero-freeness-property for the vertex $z'=\upb_Y^{-1}(1)$, we have to make five 
steps forward in order to have a zero-free interval ahead of us---but this is also 
the highest number of necessary repositioning steps we can encounter). If $r$ is 
large enough to have 
$\beta f_0(Y) = \beta (2r+2) \geq 2 = \lvert \uph_Y^{-1}(0) \rvert$, too, 
then both requirements about $\uph_Y$ are met. This completes the proof 
of \ref{it:roughbandwidthstatements} in the case $Y = \Pr_r^{\boxminus}$. 

As to $Y = \upM_r^{\boxminus}$, replace `$\upM_r^{\boxtimes}$' by `$\upM_r^{\boxminus}$' 
throughout the paragraph before the last (and delete the comment about bandwidth equal 
to $4$) in order to arrive at a proof of \ref{it:roughbandwidthstatements} in the 
case $Y = \upM_r^{\boxminus}$. 

Since $n_0$ can be chosen large enough to simultaneously 
satisfy the finitely many (and only $\beta$-dependent) requirements on $r$ 
encountered in the above cases, we have now 
proved \ref{it:roughbandwidthstatements} (where the $n_0$ is promised \emph{before} 
the choice $Y\in\{\upC_n^2,\CL_r,\Pr_r^{\boxtimes},\Pr_r^{\boxminus},\upM_r^{\boxtimes},\upM_r^{\boxminus}\}$ is made) in its entirety. 
\end{proof}

Let us close Section~\ref{t6t66t7y8778655454445e45} with two comments. Firstly, 
our proof of \ref{83276268932879348793254328793} shows that out of the generating 
set $\mathcal{C}_{\{f_0(\cdot)-1,f_0(\cdot)\}}(\Pr_r^{\boxminus})$ it suffices to 
use only \emph{one} circuit having the length $f_0(\cdot)-1$. The same is true 
for $\mathcal{C}_{\{f_0(\cdot)-1,f_0(\cdot)\}}(\Pr_r^{\boxminus})$. Since the 
monotonicity-argument used for proving Theorem~\ref{thm:mainresults} keeps adding 
Hamilton circuits to the current generating system---but never adds a circuit of 
length $f_0(\cdot)-1$ to it---this also implies that in 
Theorem~\ref{thm:mainresults}.\ref{thm:minDegOneHalfPlusGammaImpliesAllOneCanAskForWhenOrderIsEven}, a single circuit of length $f_0(\cdot)-1$ suffices in a generating set. 
Secondly, with $\Pr_r^{\boxminus,-} := \Pr_r^{\boxminus} -  x_0 z''$ and 
$\upM_r^{\boxminus,-} := \upM_r^{\boxminus} -x_0 z''$, the study of the special 
cases $r=4$ and $r=6$ strongly suggests that for every even $r\geq 4$,

{\scriptsize 
\begin{minipage}[b]{0.48\linewidth}
\begin{enumerate}[label={\rm($\boxminus,-$.(\arabic{*}))},start=0]
\item\label{39074298364454938747474747329} 
$\dim_{\Z/2}\bigl(\upZ_1(\Pr_r^{\boxminus,-};\Z/2) / \langle \mathcal{H}(\Pr_r^{\boxminus}) \rangle_{\Z/2}\bigr) = 2$ \; ,
\end{enumerate}
\end{minipage}
\begin{minipage}[b]{0.48\linewidth}
\begin{enumerate}[label={\rm($\boxminus,-$.(\arabic{*}))},start=1]
\item\label{1364867134086734108764408746} 
$\dim_{\Z/2}\bigl(\upZ_1(\upM_r^{\boxminus,-};\Z/2) / \langle \mathcal{H}(\upM_r^{\boxminus}) \rangle_{\Z/2}\bigr) = 2$ \; ,
\end{enumerate}
\end{minipage}
\newline
}
but we will not prove this in this paper. The statements 
\ref{39074298364454938747474747329} and \ref{1364867134086734108764408746}, if true 
in general, provide a justification for employing the symmetry-destroying edge 
$x_0 z''$: because of these two codimensions, the  graphs $\Pr_r^{\boxminus,-}$ and 
$\upM_r^{\boxminus,-}$---while spanning---are unsuitable as auxiliary substructures for 
proving \ref{thm:minDegOneHalfPlusGammaImpliesAllOneCanAskForWhenOrderIsEven} 
in Theorem~\ref{thm:mainresults}; for when adding an edge, the codimension of the 
span of Hamilton circuits within the cycle space can at most stay the same, never 
decrease.

\section{An alternative argumentation for step \ref{it:proofstrategy:structuralgraphtheorystep}}\label{sec:alternativeargumentation}

The entire Section~\ref{sec:alternativeargumentation} is logically superfluous 
for our proof of Theorem~\ref{thm:mainresults}. It is included here for two 
reasons. Firstly, to provide readers with an alternative way of arguing. 
Secondly, it seems conceivable that if there should ever exist graph-theoretical 
characterizations of the property of a Hamilton-generated cycle space, 
then non-separating induced circuits will play a role in them. 

The following theorem proved by A.~Kelmans will save us work in proving 
Lemma~\ref{nonsepindcircuitsPr}: 

\begin{theorem}[{cf. \cite[Theorem~4.5.2]{MR2159259} and \cite[p.~264]{MR625067}}]\label{thm:planaritycriterionintermsofnonseparatinginducedcircuits}
If $X$ is a $3$-connected graph, it is planar if and only if 
each $e\in\upE(X)$ is contained in at most two non-separating 
induced circuits of $X$. \hfill $\Box$
\end{theorem}

\begin{lemma}\label{nonsepindcircuitsPr}
For every $r\in \Z_{\geq 4}$, the set of nonseparating induced circuits 
in $\Pr_r$ equals 
\begin{equation}\label{specificationofallnonseparatinginducedcircuitsinPr}
 \Nsi_{\Pr_r}  := \{ \upC_{r,1} \} \sqcup \{ \upC_{r,2} \} \sqcup 
\bigsqcup_{0\leq i \leq r-1} \{ \upC_{4,i}\}
\end{equation}
where $\upC_{r,1}:= x_0 x_1 \dotsc x_{r-1} x_0$, 
$\upC_{r,2} := y_0 y_1 \dotsc y_{r-1} y_0$ and 
$\upC_{4,i}:= x_i x_{i+1} y_{i+1} y_i x_i$. In particular there are exactly $r+2$ 
non-separating induced circuits in $\Pr_r$.
\end{lemma}
\begin{proof}
Inclusion $\supseteq$ is easy to check. What we have to justify is that  
\eqref{specificationofallnonseparatinginducedcircuitsinPr} is the complete list 
of non-separating induced circuits in $\Pr_r$. 
This can be done by working directly from the definitions and distinguishing 
cases but we will take a shortcut via Kelmans' characterization of 
planar $3$-connected graphs:  let $C$ be an arbitrary non-separating induced 
circuit in $\Pr_r$. Suppose that $C$ is missing from 
\eqref{specificationofallnonseparatinginducedcircuitsinPr}. Let $e$ be an arbitrary 
edge of $C$. Since $\upE(C)\subseteq\upE(\Pr_r)$, 
Definition~\ref{def:squareofcircuitprismmobiusladder} implies that there 
is $0\leq j \leq r-1$ with 
$e\in \{ x_j x_{j+1},\ x_j y_j,\ y_j y_{j+1} \}$. By swapping the 
symbols $x$ and $y$ if necessary we may assume that there are only the 
two alternatives $e\in \{ x_j x_{j+1},\ x_j y_j \}$. 
If $e = x_j x_{j+1}$, then $\upC_{r,1}$ and $\upC_{4,j}$ are two distinct 
non-separating induced circuits in $\Pr_r$ which contain the edge $e$. Since 
by assumption $C$ does not appear in 
\eqref{specificationofallnonseparatinginducedcircuitsinPr}, $C$ is a \emph{third} 
non-separating induced circuit containing $e$. This is where Kelmans' theorem comes 
in: it is evident that $\Pr_r$ is planar and also (using Menger's theorem) that 
$\Pr_r$ is $3$-connected for every $r\in \Z_{\geq 3}$, and therefore 
Theorem~\ref{thm:planaritycriterionintermsofnonseparatinginducedcircuits} implies 
that every $e\in\upE(\Pr_r)$ lies in \emph{at most two} non-separating induced 
circuits of $\Pr_r$, a contradiction. Similarly, for the alternative 
$e = x_j y_j$, the circuits $\upC_{4,j-1}$ and $\upC_{4,j}$ are two 
distinct non-separating induced circuits in $\Pr_r$ containing $e$. 
Again, $C$ being a third one is a contradiction to 
Theorem~\ref{thm:planaritycriterionintermsofnonseparatinginducedcircuits}. This 
proves that none of the non-separating induced circuits of $\Pr_r$ has been 
forgotten in \eqref{specificationofallnonseparatinginducedcircuitsinPr}.
\end{proof}

\begin{lemma}\label{nonseparatinginducedcircuitsinPrExpressedAsSumsOfHamCircuitsIfrIsEven} 
Let $r\in \Z_{\geq 3}$ be \emph{even} and 
\begin{align}
H_{\upw,1} & := x_0x_1y_1y_2x_2x_3\dotsc x_{r-1}y_{r-1}y_0x_0 \quad , \label{defOfHw1} \\
H_{\upw,2} & := y_0y_1x_1x_2y_2y_3\dotsc y_{r-1}x_{r-1}x_0y_0 \qquad , \label{defOfHw2} \\
H_i & := x_ix_{i+1}\dotsc x_{i+r-1}y_{i+r-1}y_{i+r-2}\dotsc y_ix_i \quad . \label{defOfHi}
\end{align}
Then every $C\in \Nsi_{\Pr_r}$ can be expressed as a symmetric difference of some 
of the Hamilton circuits in 
$\{H_{\upw,1}\} \sqcup \{ H_{\upw,2}\} \sqcup \bigsqcup_{0\leq i \leq r-1} \{H_i\}$. One 
way to do this the following (for the definition of $\upC_{4,i}$, $\upC_{r,1}$ 
and $\upC_{r,2}$ see \eqref{specificationofallnonseparatinginducedcircuitsinPr}, 
and for the notation `$\upc_{X}$' see 
Section~\ref{section:presentationofmainresults}). Regardless of the value 
of $r\  \mathrm{mod}\  4$, for every $0\leq i \leq r-1$,
\begin{equation}\label{reprofC4i}
\upc_{\upC_{4,i}} = \upc_{H_{\upw,1}} + \upc_{H_{\upw,2}} + \upc_{H_{i+1}} \quad .
\end{equation}
Moreover, with the abbreviation 
$\Sigma:= \sum_{0\leq i \leq \frac{r}{2} - 1} \upc_{H_{2i}}$ , 
if $r\equiv\ 0\ (\mathrm{mod}\ 4)$, then 
\begin{equation}\label{repOfCr1andCr2ifrcongruent0mod4}
\upc_{\upC_{r,1}} = \upc_{H_{\upw,1}} + \Sigma \qquad\text{and}\qquad \upc_{\upC_{r,2}} = 
\upc_{H_{\upw,2}} + \Sigma\quad ,
\end{equation}
while if $r\equiv\ 2\ (\mathrm{mod}\ 4)$, then 
\begin{equation}\label{repOfCr1andCr2ifrcongruent2mod4}
\upc_{\upC_{r,1}} = \upc_{H_{\upw,2}} + \Sigma \qquad\text{and}\qquad \upc_{\upC_{r,2}} = 
\upc_{H_{\upw,1}} + \Sigma\quad .
\end{equation}
\end{lemma}
\begin{proof} 
Among all non-trivial coefficient rings, $\Z/2$ is the only one which has the 
convenient property that two chains are equal if and only if their supports are. 
We will make use of this without further mention. 
We first prove \eqref{reprofC4i} by showing 
$\Supp(\upc_{C_{4,i}}) = \upE(\upC_{4,i})$ $=$ 
$\Supp (\upc_{H_{\upw,1}} + \upc_{H_{\upw,2}} + \upc_{H_{i+1}})$. 

As to $\upE(\upC_{4,i}) \subseteq \Supp (\upc_{H_{\upw,1}} + \upc_{H_{\upw,2}} + \upc_{H_{i+1}})$ 
one may argue as follows. There are only three types of $e\in \upE(\upC_{4,i})$, 
namely $e=x_iy_i$, $e=x_ix_{i+1}$ and $e=y_iy_{i+1}$. Regardless of the parity of $i$, 
an $e=x_iy_i$ is simultaneously in $\upE(H_{\upw,1})$, 
in $\upE(H_{\upw,2})$ and in $\upE(H_{i+1})$. Thus, such an $e$ is in the support 
of each of the three summands $\upc_{H_{\upw,1}}$, $\upc_{H_{\upw,2}}$ and $\upc_{H_{i+1}}$ 
in \eqref{reprofC4i}, and this implies 
$e\in \Supp(\upc_{H_{\upw,1}} + \upc_{H_{\upw,2}} + \upc_{H_{i+1}})$. 
For the other two types of $e\in \upE(\upC_{4,i})$, we have to pay attention to 
the parity of $i$: if $i$ is even, then $x_ix_{i+1}\in \upE(H_{\upw,1})$, 
$x_i x_{i+1} \notin \upE(H_{\upw,2})$, $y_{i} y_{i+1}\notin \upE(H_{\upw,1})$ 
and $y_iy_{i+1}\in \upE(H_{\upw,2})$, while if $i$ is odd, the latter four 
statements are true with $\in$ and $\notin$ interchanged. This shows that, for 
whatever parity of $i$, both $x_i x_{i+1}$ and $y_i y_{i+1}$ are contained in the 
support of exactly one of the two summands $\upc_{H_{\upw,1}}$ and $\upc_{H_{\upw,2}}$. 
Concerning the third summand $\upc_{H_{i+1}}$ we see from its definition that 
regardless of the parity of $i$ the edges $x_i x_{i+1}$ and $y_i y_{i+1}$ are precisely 
those two edges of the two 
circuits $x_0x_1\dotsc x_{r-1}x_0$ and $y_0y_1\dotsc y_{r-1}y_0$ which are 
missing from $\upE(H_{i+1})$. Therefore, for both parities of $i$, for both 
$e\in\{ x_ix_{i+1},\ y_iy_{i+1}\}$ we know that $e$ is contained in exactly one 
support of the three summands $\upc_{H_{\upw,1}}$, $\upc_{H_{\upw,2}}$ and $\upc_{H_{i+1}}$ 
in \eqref{reprofC4i}, and therefore 
$e\in \Supp(\upc_{H_{\upw,1}} + \upc_{H_{\upw,2}} + \upc_{H_{i+1}})$. This completes 
the proof of 
$\upE(\upC_{4,i}) \subseteq \Supp (\upc_{H_{\upw,1}} + \upc_{H_{\upw,2}} + \upc_{H_{i+1}})$. 

As to 
$\upE(\upC_{4,i}) \supseteq \Supp (\upc_{H_{\upw,1}} + \upc_{H_{\upw,2}} + \upc_{H_{i+1}})$ we 
prove the equivalent inclusion 
$\upE(\Pr_r) \setminus \upE(\upC_{4,i}) \subseteq \upE(\Pr_r) \setminus \Supp(\upc_{H_{\upw,1}} + \upc_{H_{\upw,2}} + \upc_{H_{i+1}})$, 
thus taking advantage of a less complex description of the left-hand side of the 
inclusion: the set $\upE(\Pr_r) \setminus \upE(\upC_{4,i})$ can be classified into 
three type of edges, namely $x_\iota y_\iota$ for every 
$\iota\in [0,r-1]\setminus\{i,i+1\}$, and the types $x_\iota x_{\iota+1}$ and 
$y_\iota y_{\iota+1}$ for every $\iota\in [0,r-1]\setminus\{i,i+1\}$. As to the 
type $x_\iota y_\iota$, by definition of $H_{i+1}$ we have 
$x_\iota y_\iota\in \upE(H_{\upw,1}) \cap \upE(H_{\upw,2})$ but 
$x_\iota y_\iota \notin \upE(H_{i+1})$, and therefore 
$x_\iota y_\iota\notin\Supp ( \upc_{H_{\upw,1}} + \upc_{H_{\upw,2}} + \upc_{H_{i+1}})$, 
for every $\iota\in [0,r-1]\setminus\{i,i+1\}$. As to the types 
$x_\iota x_{\iota+1}$ and $y_\iota y_{\iota+1}$ our inspection of the 
$\iota\in[0,r-1]\setminus\{i\}$ has to pay attention to the parity of $\iota$: 
for every even $\iota\in[0,r-1]\setminus\{i\}$ we have 
$x_\iota x_{\iota+1} \in \upE(H_{\upw,1}) \cap \upE(H_{i+1})$ but 
$e\notin \upE(H_{\upw,2})$ and therefore 
$x_\iota x_{\iota+1}\notin \Supp(\upc_{H_{\upw,1}} + \upc_{H_{\upw,2}} + \upc_{H_{i+1}})$, 
while  $y_\iota y_{\iota+1} \in \upE(H_{\upw,2}) \cap \upE(H_{i+1})$ but 
$y_\iota y_{\iota+1} \notin \upE(H_{\upw,1})$ and therefore 
$y_\iota y_{\iota+1} \notin \Supp(\upc_{H_{\upw,1}} + \upc_{H_{\upw,2}} + \upc_{H_{i+1}})$; 
for every odd $\iota\in[0,r-1]\setminus\{i\}$ we have 
$x_\iota x_{\iota+1} \in \upE(H_{\upw,2}) \cap \upE(H_{i+1})$ but 
$x_\iota x_{\iota+1} \notin\upE(H_{\upw,1})$ and therefore
$x_\iota x_{\iota+1}\notin \Supp(\upc_{H_{\upw,1}} + \upc_{H_{\upw,2}} + \upc_{H_{i+1}})$, 
while 
$y_\iota y_{\iota+1} \in \upE(H_{\upw,1}) \cap \upE(H_{i+1})$ but 
$y_\iota y_{\iota+1}\notin \upE(H_{\upw,2})$ and therefore 
$y_\iota y_{\iota+1}\notin \Supp(\upc_{H_{\upw,1}} + \upc_{H_{\upw,2}} + \upc_{H_{i+1}})$. 
All told, none of the edges of the stated types is in 
 $\Supp(\upc_{H_{\upw,1}} + \upc_{H_{\upw,2}} + \upc_{H_{i+1}})$ and this completes the proof 
of 
$\upE(\upC_{4,i}) \supseteq \Supp (\upc_{H_{\upw,1}} + \upc_{H_{\upw,2}} + \upc_{H_{i+1}})$. 

By the two preceding paragraphs, $\Supp(\upc_{\upC_{4,i}})$ $=$ 
$\upE(\upC_{4,i}) = \Supp (\upc_{H_{\upw,1}} + \upc_{H_{\upw,2}} + \upc_{H_{i+1}})$. This 
completes the proof of \eqref{reprofC4i}. 

To prove \eqref{repOfCr1andCr2ifrcongruent0mod4} 
and \eqref{repOfCr1andCr2ifrcongruent2mod4} we will---since both equations 
involve this sum---first analyse the sum $\Sigma$ by itself, proving five claims 
which will combine to a proof of \eqref{repOfCr1andCr2ifrcongruent0mod4} 
and \eqref{repOfCr1andCr2ifrcongruent2mod4}.

\textsl{Claim~1:} For every even $r$ and every $i_0\in [0, r-1]$ the 
edge $x_{i_0} y_{i_0}$ lies in exactly one of the $\frac{r}{2}$ summands 
$\{ H_{2i}\colon 0\leq i \leq \frac{r}{2} - 1 \}$ comprising $\Sigma$. 
Proof of Claim~1:  Let $i_0\in [0, r-1]$ be given. For every 
$0\leq i \leq \frac{r}{2}-1$, there are exactly two edges of type $x_\iota y_\iota$ 
in  $H_{2i}$, namely $x_{2i} y_{2i}$ and $x_{2i+r-1} y_{2i+r-1}$. Since $2i$ is even 
and $r$ is even by assumption, $(2i+r-1)\ \mathrm{mod}\ r$ is odd. Therefore, 
if $i_0$ is even, only the edges $x_{2i} y_{2i}$ have the potential to be equal 
to $x_{i_0} y_{i_0}$ and for $i\in [0, r-1]$, there is the unique 
solution $i=\frac{\iota_0}{2}$ for the equation $x_{2i} y_{2i} = x_{\iota_0} y_{\iota_0}$, 
while if $\iota_0$ is odd, only the edges  $x_{2i+r-1} y_{2i+r-1}$ qualify and 
for $i\in [0, r-1]$, there is the unique solution $i = \frac{\iota_0+1}{2}$ for 
$x_{2i+r-1} y_{2i+r-1} = x_{\iota_0} y_{\iota_0}$. This proves Claim~1.

\textsl{Claim~2:} For every even $r\in \Z_{\geq 3}$ and every \emph{even} 
$\iota_0\in [0, r-1]$, both $x_{\iota_0} x_{\iota_0+1}$ and 
$y_{\iota_0} y_{\iota_0+1}$ lie in each of the 
$\{ H_{2i}\colon 0\leq i \leq \frac{r}{2}-1 \}$. In particular, they both lie 
in exactly $\frac{r}{2}$ (supports of) summands of $\Sigma$. In particular, if 
$r\equiv 0\ (\mathrm{mod}\ 4)$, then both $x_{\iota_0} x_{\iota_0+1}$ and 
$y_{\iota_0} y_{\iota_0+1}$ lie in an even number of supports, and 
if $r\equiv 2\ (\mathrm{mod}\ 4)$, they both lie in an odd number of supports. 
Proof of Claim~2: Let an even $\iota_0\in [0, r-1]$ be given. It has to be shown 
that both $x_{\iota_0} x_{\iota_0+1}$ and $y_{\iota_0} y_{\iota_0+1}$ are edges of $H_{2i}$ 
for every $i\in [0, \frac{r}{2}-1]$. 
Since $H_{2i} = x_{2i}x_{2i+1}\dotsc x_{2i+r-1}y_{2i+r-1}y_{2i+r-2}\dotsc y_{2i}x_{2i}$, it 
is evident that the only edge of type $x_{\iota} x_{\iota+1}$ which is missing 
from $\upE(H_{2i})$ is $x_{2i+r-1} x_{2i}$. If $x_{2i+r-1} x_{2i}$ were 
equal to $x_{\iota_0} x_{\iota_0+1}$, then the evenness of $\iota_0$ implies $2i=\iota_0$ 
and therefore $\iota_0 + 1 = 2i+r-1 = \iota_0 + r-1$ $\Leftrightarrow$ $1 = -1$, 
to be interpreted as an equation in the group $\Z/r$, which because of $r\geq 5$ 
is a contradiction. Therefore, indeed 
$x_{\iota_0} x_{\iota_0+1} \in \upE(H_{2i})$. An entirely analogous argument 
proves that $y_{\iota_0} y_{\iota_0+1} \in \upE(H_{2i})$. Since the two other 
statements in Claim~2 are mere specializations of the first, the proof of 
Claim~2 is complete. 

\textit{Claim~3:} For every even $r\in \Z_{\geq 3}$ and every \emph{odd} 
$\iota_0 \in [0, r-1]$, both 
$x_{\iota_0} x_{\iota_0+1}$ and $y_{\iota_0} y_{\iota_0+1}$ lie in each of 
$\{ H_{2i}\colon i\in [0,\frac{r}{2}-1]\setminus \{\frac{\iota_0+1}{2}\}\}$. 
However, for $i=\frac{\iota_0+1}{2}$ both 
$x_{\iota_0} x_{\iota_0+1}\notin\upE(H_{2i})$ and 
$y_{\iota_0} y_{\iota_0+1}\notin\upE(H_{2i})$. In particular, each of  
$x_{\iota_0} x_{\iota_0+1}$ and $y_{\iota_0} y_{\iota_0+1}$ lies in 
exactly $\frac{r}{2}-1$ (supports of) summands of $\Sigma$. In particular, if 
$r\equiv 0\ (\mathrm{mod}\ 4)$, then both 
$x_{\iota_0} x_{\iota_0+1}$ and $y_{\iota_0} y_{\iota_0+1}$ lie in an 
odd number of supports, and if $r\equiv 2\ (\mathrm{mod}\ 4)$, they both lie in an 
even number of supports. 
Proof of Claim~3: Retrace the steps of the proof of Claim~2. Now the equations 
$x_{2i+r-1} x_{2i} = x_{\iota_0} x_{\iota_0+1}$ and 
$y_{2i+r-1} y_{2i} = y_{\iota_0} y_{\iota_0+1}$ do have a (unique) solution 
$i = \tfrac{\iota_0+1}{2}$ and this fact is responsible 
for the exceptional case mentioned in the claim. Since again the other statements 
are merely specializations of the first, this proves Claim~3.

The motivation for formulating the following statements is that the summands on the 
right-hand sides have mutually disjoint supports. 

\textit{Claim~4.} For every even $r\in \Z_{\geq 3}$, 
if $r\equiv 0\ (\mathrm{mod}\ 4)$, then 
$\Sigma$ $=$ $\sum_{0\leq i \leq \frac{r}{2}-1} \upc_{\upC_{4,2i+1}}$, and if 
$r\equiv 2\ (\mathrm{mod}\ 4)$, then 
$\Sigma$ $=$ $\sum_{0\leq i \leq \frac{r}{2}-1} \upc_{\upC_{4,2i}}$.
Proof of Claim~4. It is evident from the definition of $\upC_{4,\iota}$ that 
the $\{ \upC_{4,2i+1}\colon 0\leq i \leq \frac{r}{2}-1\}$ have pairwise disjoint 
supports. Therefore the support of $\sum_{0\leq i \leq \frac{r}{2}-1} \upc_{\upC_{4,2i+1}}$ is 
the (disjoint) union of the supports of the $\upC_{4,2i+1}$. Analogously for 
$\sum_{0\leq i \leq \frac{r}{2}-1} \upc_{\upC_{4,2i}}$. Therefore, directly from the 
definition of $\upC_{4,\tilde{\imath}}$ in 
\eqref{specificationofallnonseparatinginducedcircuitsinPr} it follows that 
$\Supp \bigl ( \sum_{0\leq i \leq \frac{r}{2}-1} \upc_{\upC_{4,2i}} \bigr ) = S_{\mathrm{e}}$ and 
$\Supp \bigl( \sum_{0\leq i \leq \frac{r}{2}-1} \upc_{\upC_{4,2i+1}} \bigr ) = S_{\mathrm{o}}$ 
where
\begin{align}\label{proofOfClaim4supportOfSum}
S_{\mathrm{e}} & := \bigsqcup_{0\leq i \leq r-1} \{ x_i y_i \} \sqcup 
\bigsqcup_{0\leq i \leq \frac{r}{2}-1} \{ x_{2i+1} x_{2i+2},\ y_{2i+1} y_{2i+2}\}\quad ,\notag \\
S_{\mathrm{o}} & := \bigsqcup_{0\leq i \leq r-1} \{ x_i y_i\} \sqcup 
\bigsqcup_{0\leq i \leq \frac{r}{2}-1} \{ x_{2i} x_{2i+1},\  y_{2i} y_{2i+1}\} \quad .
\end{align}
Hence Claim~4 is equivalent to the statement that 
{\scriptsize
\begin{enumerate}[label={\rm(\arabic{*})}] 
\item\label{proofofexplicitrealizations:implication01} 
if $r\equiv 0\ (\mathrm{mod}\ 4)$, then 
$\Supp \bigl ( \sum_{0\leq i \leq\frac{r}{2}-1} \upc_{H_{2i}} \bigr ) = S_{\mathrm{e}}$\quad ,
\item\label{proofofexplicitrealizations:implication02} 
if $r\equiv 2\ (\mathrm{mod}\ 4)$, then 
$\Supp \bigl (  \sum_{0\leq i \leq\frac{r}{2}-1} \upc_{H_{2i}} \bigr ) = S_{\mathrm{o}}$\quad .
\end{enumerate}
}
That this is so can be deduced from the claims above as follows: 
As to \ref{proofofexplicitrealizations:implication01}, let $r\in\Z_{\geq 3}$ with 
$r\equiv 0\ (\mathrm{mod}\ 4)$. 
To prove $\supseteq$ in \ref{proofofexplicitrealizations:implication01}, note that 
by Claim~1 every $x_{\iota_0} y_{\iota_0}$ with $\iota_0 \in [0, r-1]$ is in 
exactly one of $\{ H_{2i}\colon 0\leq i \leq \frac{r}{2}-1 \}$, hence 
$x_{\iota_0} y_{\iota_0}$ $\in$ $\Supp\bigl(\sum_{0\leq j \leq\frac{r}{2}-1}\upc_{H_{2j}}\bigr)$. 
Moreover, Claim~3 invoked with $\iota_0 := 2i+1$ for $i \in [0,\frac{r}{2}-1]$ 
guarantees that for every $i \in [0,\frac{r}{2}-1]$ both $x_{2i+1} x_{2i+2}$ and 
$y_{2i+1} y_{2i+2}$ are in an odd number of 
$\{ H_{2i} \colon 0\leq i \leq \frac{r}{2}-1 \}$, hence 
$\{ x_{2i+1} x_{2i+2},\ y_{2i+1} y_{2i+2},\ \}$ $\subseteq$ 
$\Supp \bigl ( \sum_{0\leq i \leq\frac{r}{2}-1} \upc_{H_{2i}} \bigr )$ for every 
$i \in [0,\frac{r}{2}-1]$. This proves $\supseteq$ 
in \ref{proofofexplicitrealizations:implication01}. 
To prove $\subseteq$ in \ref{proofofexplicitrealizations:implication01}, we prove 
$\upE(\Pr_r)\setminus S_{\upe} \subseteq 
\upE(\Pr_r)\setminus \Supp \bigl (\sum_{0\leq j \leq\frac{r}{2}-1} \upc_{H_{2j}} \bigr )$. 
By definition of $S_{\upe}$ we have 
$\upE(\Pr_r)\setminus S_{\upe}$ $=$ 
$\{ x_{2i} x_{2i+1} \colon 0\leq i \leq \frac{r}{2}-1 \bigr \}$ $\sqcup$ 
$\{ y_{2i} y_{2i+1} \colon 0\leq i \leq \frac{r}{2}-1 \bigr \}$, so these 
are the types of edges whose inclusion in 
$\upE(\Pr_r)\setminus \Supp \bigl (\sum_{0\leq j \leq\frac{r}{2}-1} \upc_{H_{2j}} \bigr )$ 
we have to justify. Invoking Claim~2 successively with $\iota_0:=2i$ for 
$i \in [0, \frac{r}{2}-1]$ it follows that both $x_{2i} x_{2i+1}$ and 
$y_{2i} y_{2i+1}$ lie in an even number of 
$\{H_{2i} \colon 0\leq i \leq \frac{r}{2}-1\}$, hence both are missing from 
$\Supp \bigl ( \sum_{0\leq j \leq\frac{r}{2}-1} \upc_{H_{2j}} \bigr )$. Since this accounts 
for all the above mentioned types of edges in $\upE(\Pr_r)\setminus S_{\upe}$, 
we have proved $\subseteq$ in \ref{proofofexplicitrealizations:implication01}. 
This completes the proof of \ref{proofofexplicitrealizations:implication01}. 

As to \ref{proofofexplicitrealizations:implication02}, 
let $r\in\Z_{\geq 3}$ with $r\equiv 2\ (\mathrm{mod}\ 4)$. To prove $\supseteq$ 
in \ref{proofofexplicitrealizations:implication02}, note that Claim~1 guarantees 
that every $x_{\iota_0} y_{\iota_0}$ with $\iota_0\in [0, r-1]$ is in exactly 
one of $\{H_{2i}\colon 0\leq i \leq \frac{r}{2}-1 \}$, hence
$x_{\iota_0} y_{\iota_0}\in \Supp\bigl (\sum_{0\leq j \leq\frac{r}{2}-1} \upc_{H_{2j}} \bigr )$, 
too. Moreover, Claim~2 invoked successively with $\iota_0:=2i$ 
for $i\in [0, \frac{r}{2}-1]$ now guarantees that both $x_{2i} x_{2i+1}$ and 
$y_{2i} y_{2i+1}$ lie in an odd number 
of $\{ H_{2i} \colon 0\leq i \leq \frac{r}{2}-1 \}$, hence 
$\{ x_{2i} x_{2i+1},\ y_{2i}, y_{2i+1} \} \subseteq 
\Supp \bigl ( \sum_{0\leq j \leq \frac{r}{2}-1} \upc_{H_{2j}} \bigr )$ for every 
$i\in [0, \frac{r}{2}-1]$, proving $\supseteq$ 
in \ref{proofofexplicitrealizations:implication02}. In order to prove 
$\subseteq$ in \ref{proofofexplicitrealizations:implication02} we again resort to 
proving the equivalent inclusion 
$\upE(\Pr_r)\setminus S_{\upo} \subseteq \upE(\Pr_r)\setminus \Supp 
\bigl (\sum_{0\leq j \leq\frac{r}{2}-1} \upc_{H_{2j}} \bigr )$. The definition of $S_{\upo}$ 
shows that $\upE(\Pr_r)\setminus S_{\upo}$ $=$ 
$\{ x_{2i+1} x_{2i+2} \colon 0\leq i \leq\frac{r}{2}-1 \}$ $\sqcup$ 
$\{ y_{2i+1} y_{2i+2} \colon 0\leq i \leq\frac{r}{2}-1 \}$. Appealing 
to Claim~3 with  $\iota_0:=2i+1$ for every $i \in [0,\frac{r}{2}-1]$ we deduce that 
both $x_{2i+1} x_{2i+2}$ and $y_{2i+1} y_{2i+2}$ lie in an even number of the 
$\{ H_{2i} \colon 0\leq i \leq \frac{r}{2}-1 \}$, hence both are 
missing from $\Supp \bigl ( \sum_{0\leq j \leq \frac{r}{2}-1} \upc_{H_{2j}} \bigr )$. 
This accounts for every edge in $\upE(\Pr_r)\setminus S_{\upo}$ and therefore  
$\upE(\Pr_r)\setminus S_{\upo} \subseteq \upE(\Pr_r)\setminus \Supp 
\bigl (\sum_{0\leq j \leq\frac{r}{2}-1} \upc_{H_{2j}} \bigr )$ is true. This completes the 
proof of $\subseteq$ in \ref{proofofexplicitrealizations:implication02}. The proof 
of \ref{proofofexplicitrealizations:implication02} is now complete, as is the proof 
of Claim~4 as a whole. 

\textit{Claim~5:} For every even $r\in \Z_{\geq 3}$ we have 
$\{ x_i y_i \colon 0\leq i \leq r-1 \} \cap 
\Supp\bigl(\upc_{H_{\upw,1}} + \Sigma\bigr ) = \emptyset$  
and $\{ x_i y_i \colon 0\leq i \leq r-1 \bigr \} \cap 
\Supp\bigl (\upc_{H_{\upw,2}} + \Sigma \bigr ) = \emptyset$. 
Proof of Claim~5: By definition of $H_{\upw,1}$ and $H_{\upw,2}$, both 
$\Supp (\upc_{H_{\upw,1}})$ and $\Supp (\upc_{H_{\upw,2}})$ contain each of the edges 
$\{ x_iy_i\colon 0\leq i \leq r-1 \}$. It follows from the mutual 
edge-disjointness of the (supports of) the summands on the right-hand sides of 
the sums in Claim~4 that $\Supp(\Sigma)$ for both values of $r\ \mathrm{mod}\ 4$ 
contains all of these edges, too, and this proves the emptyness of the intersections 
in Claim~5.

We finally prove equations \eqref{repOfCr1andCr2ifrcongruent0mod4} and 
\eqref{repOfCr1andCr2ifrcongruent2mod4}. First note that 
Lemma~\ref{nonseparatinginducedcircuitsinPrExpressedAsSumsOfHamCircuitsIfrIsEven} 
demands $r$ to be even from the outset, hence all appeals to the claims above 
(all require even $r$) are valid. 

As to \eqref{repOfCr1andCr2ifrcongruent0mod4}, assume that 
$r\equiv 0\ (\mathrm{mod}\ 4)$, hence $\frac{r}{2}$ is even.  We first prove 
$\Supp(\upc_{\upC_{r,1}}) = \Supp(\upc_{H_{\upw,1}} + \Sigma)$. We begin 
with $\Supp(\upc_{\upC_{r,1}}) \subseteq \Supp(\upc_{H_{\upw,1}} + \Sigma)$. In 
$\upE(\upC_{r,1})$, there are only edges of the form $x_ix_{i+1}$. Of these, 
we distinguish the types of edges $x_ix_{i+1}$ with even $i$ from those 
with odd $i$ and argue as follows: for every \emph{even} $i\in [0, r-1]$, we know 
by Claim~2 that $x_i x_{i+1}$ lies in an even number of 
$\{ H_{2i}\colon 0\leq i \leq \frac{r}{2}-1 \}$, hence 
$x_i x_{i+1} \notin \Supp(\Sigma)$, while directly from the definition 
of $H_{\upw,1}$ we see that $x_i x_{i+1} \in \upE(H_{\upw,1})$, hence 
$x_i x_{i+1} \in \Supp(\upc_{H_{\upw,1}} + \Sigma)$. For every \emph{odd} $i\in [0,r-1]$, 
we know by Claim~3 that $x_i x_{i+1}$ lies in an odd number 
of $\{ H_{2i} \colon 0\leq i \leq \frac{r}{2}-1 \}$, hence 
$x_i x_{i+1} \in \Supp(\Sigma)$, while directly from the definition 
of $H_{\upw,1}$ we see that $x_i x_{i+1} \notin \upE(H_{\upw,1})$, hence again 
$x_i x_{i+1} \in \Supp(\upc_{H_{\upw,1}} + \Sigma)$. Since now all edges of 
$\upC_{r,1}$ have been found to lie in $\Supp(\upc_{H_{\upw,1}} + \Sigma)$, this 
proves $\Supp(\upc_{\upC_{r,1}}) \subseteq \Supp(\upc_{H_{\upw,1}} + \Sigma)$. 

We now prove $\Supp(\upc_{\upC_{r,1}}) \supseteq \Supp(\upc_{H_{\upw,1}} + \Sigma)$, yet 
again by proving the equivalent inclusion 
$\upE(\Pr_r)\setminus \Supp(\upc_{\upC_{r,1}}) \subseteq 
\upE(\Pr_r)\setminus\Supp(\upc_{H_{\upw,1}} + \Sigma)$. The only types of edges in 
$\upE(\Pr_r)\setminus \Supp(\upc_{\upC_{r,1}})$ are $y_i y_{i+1}$ and $x_i y_i$. As to 
the former, to justify why $y_i y_{i+1}\notin \Supp(\upc_{H_{\upw,1}} + \Sigma)$ for 
every $i\in [0,r-1]$, we may repeat the preceding paragraph verbatim except for 
interchanging $x$ and $y$ and changing 
`$x_i x_{i+1}\in\upE(H_{\upw,1})$' to `$y_i y_{i+1} \notin\upE(H_{\upw,1})$' 
and 
`$x_i x_{i+1} \notin\upE(H_{\upw,1})$' to `$y_i y_{i+1} \in\upE(H_{\upw,1})$' to 
find the parities work out as they should. As to the type $x_i y_i$, note 
that Claim~5 gives just what we need, namely 
$x_i y_i\notin \Supp(\upc_{H_{\upw,1}} + \Sigma)$ for every $i\in [0,r-1]$. 
The proof of $\Supp(\upc_{\upC_{r,1}}) \supseteq \Supp(\upc_{H_{\upw,1}} + \Sigma)$ is 
now complete, as is the proof of 
$\Supp(\upc_{\upC_{r,1}}) = \Supp(\upc_{H_{\upw,1}} + \Sigma)$.

To justify $\Supp(\upc_{\upC_{r,2}}) = \Supp(\upc_{H_{\upw,2}} + \Sigma)$ 
in \eqref{repOfCr1andCr2ifrcongruent0mod4} it suffices to change 
`$\upC_{r,1}$' into `$\upC_{r,2}$', `$H_{\upw,1}$' into `$H_{\upw,2}$' 
and `$x$' into `$y$' in the preceding two paragraphs. This completes the proof 
of \eqref{repOfCr1andCr2ifrcongruent0mod4}. 

As to \eqref{repOfCr1andCr2ifrcongruent2mod4}, assume that 
$r\equiv 2\ (\mathrm{mod}\ 4)$, hence $\frac{r}{2}$ is odd (which affects what 
Claims 2 and 3 will tell us about the parities of the number of containing 
supports). A proof for \eqref{repOfCr1andCr2ifrcongruent2mod4} can now be 
obtained by making obvious  modifications in the preceding three paragraphs. 
The proof of 
\ref{nonseparatinginducedcircuitsinPrExpressedAsSumsOfHamCircuitsIfrIsEven} 
is now complete.  
\end{proof}

\section{Concluding Remarks}\label{sec:concludingremarks}

\subsection{Two open questions and the state of contemporary knowledge}\label{7yt6r5657766tt6t66tt6t6} The formulation of Theorem~\ref{thm:mainresults} suggests further 
improvements (e.g. eliminating the lower bound on $f_0$, proving non-asymptotic 
minimum-degree thresholds, and finding an infinite set of counter-examples disproving 
the weakened implications for \emph{every} $f_0$, instead of only 
for $f_0=7$ and $f_0=12$ as was done in 
the sections \ref{78676r55r5655566567878888u} and \ref{7y6t65r44r5676756565} above).

In particular, the author does not know whether the threshold 
in \ref{thm:minDegOneHalfPlusGammaImpliesHamiltonGeneratedForOddOrder} 
can be lowered to the Dirac threshold itself. Two noteworthy open questions in 
that regard are: 

\begin{enumerate}[label={\rm(Q\arabic{*})}] 
\item\label{question:doesdiracthresholdalreadyimplyhamiltongenerated}
Let $X$  be a graph with $f_0(X)$ odd and 
$\delta(X)\geq \tfrac12 f_0(X)$. Does it follow that its cycle 
space is generated by its Hamilton circuits\quad ? 
\item\label{question:doesoneabovediracthresholdimplyguaranteeBCrasaspanningcopy}
Let $X$ be square bipartite with $\delta(X)\geq \tfrac14 f_0(X) + 1$. Does 
it follow that $\CL_{\frac12 f_0(X)} \hookrightarrow X$\quad ? 
\end{enumerate}

If \ref{question:doesoneabovediracthresholdimplyguaranteeBCrasaspanningcopy} has a 
positive answer then by arguments entirely analogous to those that were used to 
prove Theorem \ref{thm:mainresults}, it would follow immediately that 
\ref{thm:minDegOneFourthPlusGammaImpliesHamiltonGeneratednessInBipartiteGraphs} 
in Theorem~\ref{thm:mainresults} remains true when 
`$\delta(X)\geq (\tfrac14  + \gamma )f_0(X)$' is replaced by 
`$\delta(X)\geq \tfrac14 f_0(X) + 1$'. There is a theorem of 
A.~Czygrinow and H.~A.~Kierstead  \cite[Theorem~1]{MR1935733} which proves that 
if $X$ is  a sufficiently large square bipartite graph, then 
$\delta(X)\geq\tfrac14 f_0(X) + 1$ implies that $X$ contains a spanning copy of 
the \emph{non}-cyclic ladder $\NCL_r$ (i.e. $\CL_r$ with the two edges 
$\{ a_{r-1}, b_0 \}$ and $\{a_0,b_{r-1}\}$ removed). This small defect is enough to 
render this spanning subgraph unsuitable for serving as an auxiliary substructure 
in the same way $\CL_r$ has done in the present paper: while the non-cyclic ladder 
still is Hamilton-laceable, the loss of the two edges causes a drastic drop in the 
dimension of $\langle \mathcal{H}(\cdot) \rangle_{\Z/2}$~: whereas 
$\CL_r\in\mathrm{cd}_{0}\mathcal{C}_{\{f_0(\cdot)\}}$ 
by \ref{lem:it:bipartitecyclicladder:generator}, it can be checked 
that $\NCL_r$ contains only \emph{one} Hamilton circuit, hence 
$\NCL_r\in\mathrm{cd}_{\beta_1(\NCL_r)-1}\mathcal{C}_{\{f_0(\cdot)\}}$. 

In the pursuit of 
Question~\ref{question:doesdiracthresholdalreadyimplyhamiltongenerated} one should 
be aware of the following (probably known) implication (which without requiring 
$f_0(X)$ to be odd would be false):

\begin{lemma}[above the Dirac threshold, a graph with odd order is Hamilton 
connected]\label{lem:slightlyabovdiracthresholdagraphishamiltonconnected}
Every graph $X$ with $f_0(X)$ odd and $\delta(X) \geq  \tfrac12 f_0(X)$ is 
Hamilton-connected. 
\end{lemma}
\begin{proof}
This is an immediate corollary of a theorem of 
A.~S.~Asratian, O.~A.~Ambartsumian and G.~V.~Sarkisian \cite{MR1128859}\footnote{The 
author could not access this article and takes the statement of the theorem 
from \cite[Theorem~3]{MR1296952}.} which states that every connected 
graph $X$ with $f_0(X)\geq 3$ and the property 
$\lvert \upN_X(u) \rvert + \lvert \upN_X(u) \rvert \geq 
\lvert \upN_X(u) \cup \upN_X(v) \cup \upN_X(w) \rvert + 1$ for each of those 
$(u,v,w)\in \upV(X)^3$ which have $\emptyset\neq\upN_X(u)\cap\upN_X(v)\ni w$, 
is Hamilton-connected. It is evident that the present hypothesis of 
$\delta(X)\geq  \tfrac12f_0(X)$ and odd $f_0(X)$ makes the 
assumptions of this theorem true (in fact, our hypothesis makes them true 
for every $(u,v,w)\in \upV(X)^3$).  
\end{proof}

Question~\ref{question:doesdiracthresholdalreadyimplyhamiltongenerated} seems not 
to have been explicitly asked in the literature. There is, however, the 
aforementioned Conjecture~\ref{thm:BondyConjecture}, which according 
to \cite[Reference 1]{MR821540} \cite[Reference~3]{MR1174832} dates back to 1979 
and apparently is still open. For $n:=f_0(X)=2d$, 
Conjecture~\ref{thm:BondyConjecture} asks for a generating system consisting of 
Hamilton circuits together with all circuits shorter by one. For the case of 
even $n=2d$, these additional circuits are clearly necessary, but the point of 
Question~\ref{question:doesdiracthresholdalreadyimplyhamiltongenerated} is that 
for odd $n:=2d+1$ it seems quite possible to make do solely with Hamilton circuits 
(instead of the three lengths $2d-1$, $2d$ and $2d+1=f_0(X)$ allowed by Bondy's 
conjecture), all the more so as Theorem~\ref{thm:mainresults} of the present paper 
gives an asymptotic affirmative answer 
to \ref{question:doesdiracthresholdalreadyimplyhamiltongenerated}. 
The only papers explicitly addressing Bondy's conjecture apparently are \cite{MR725072} \cite{MR821540} \cite{MR815581} \cite{MR1174832} \cite{MR1760293} \cite{MR2171084} \cite{MR2279163}. We will briefly consider each of them. 
In \cite[p.~246]{MR725072}, Conjecture~\ref{thm:BondyConjecture} is merely 
mentioned at the end as a related open conjecture. 
In \cite[Theorem~2 and Corollary~4]{MR821540} it is proved that for every 
$d\in\Z$, if $X$ is a $3$-connected graph with $\delta(X)\geq d$ which is either 
non-hamiltonian or has $f_0(X)\geq 4d-5$, then $\upZ_1(X;\Z/2)$ is generated by its 
circuits of length at least $2d-1$ (note that if $f_0(X)\geq 4d-5$, the conclusion 
in Bondy's conjecture is far from generatedness by Hamilton circuits). The 
paper \cite{MR815581} does not have the cycle space as its main concern but 
announces the results of \cite{MR821540} at the very end. 
Moreover, the concern of \cite{MR1174832} is the question if and when there are 
inclusions $\mathcal{CO}_{\mathfrak{L}'}\subseteq \mathrm{cd}_0\mathcal{C}_{\mathfrak{L}''}$ 
for different sets of lengths $\mathfrak{L}'$ and $\mathfrak{L}''$; consequently 
the paper is not concerned with minimum-degree conditions and 
Conjecture~\ref{thm:BondyConjecture} is mentioned merely in 
passing \cite[p.~77]{MR1174832}. In \cite{MR1760293} the assumption about 
non-hamiltonicity appears in a 
different role, but it can be proved that the results of \cite{MR1760293} do not 
answer \ref{question:doesdiracthresholdalreadyimplyhamiltongenerated}:

\begin{theorem}[{Barovich--Locke \cite[Theorem~2.2]{MR1760293}}]\label{BarovichLockeTheorem2.2}
Let $d\in \Z$, let $X$ be a finite hamiltonian graph, let $X$ be $3$-connected, 
$\delta(X)\geq d$ and $f_0(X) \geq 2d-1$. If $f_0(X) \in \{9,\dotsc, 4d-8\}$, and 
if there exists at least one $v\in \upV(X)$ such that $X-v$ is 
\emph{not} hamiltonian, and if another condition holds (which to spell out 
would be irrelevant here), then $\upZ_1(X;\Z/2)$ is generated by the set of all 
circuits of length at least $2d-1$.
\end{theorem}
The point to be made is that if $f_0(X)$ is odd and 
$\delta(X)\geq \lceil \tfrac{ f_0(X)}{2} \rceil$, and if the theorem of 
Barovich--Locke is to yield generatedness by Hamilton circuits, then necessarily 
we must set $2d-1 = f_0(X)$. While this automatically makes the hypothesis 
$f_0(X) \in \{9,\dotsc, 4d-8\}$ true, and while 
$\delta(X)\geq \lceil \tfrac{f_0(X)}{2} \rceil$ ensures (by Dirac's theorem) 
that $X$ is hamiltonian and also that $X$ is $3$-connected, the remaining 
hypothesis of Theorem~\ref{BarovichLockeTheorem2.2} above 
\emph{cannot possibly be true} in the setting of 
Question~\ref{question:doesdiracthresholdalreadyimplyhamiltongenerated}: 
for every $v\in \upV(X)$ we have 
$\delta(X-v)$ $\geq$ $\delta(X) - 1$ $\geq$ (since $\delta(X)$ is an integer) 
$\geq$ $\lceil \tfrac12 f_0(X) \rceil - 1$ $=$ 
$\tfrac{f_0(X)}{2} - \tfrac12$ $=$ $\frac12 f_0(X-v)$, hence $X-v$ is still 
hamiltonian by Dirac's theorem. Hence 
Theorem~\ref{BarovichLockeTheorem2.2}, as it stands, does not answer 
Question~\ref{question:doesdiracthresholdalreadyimplyhamiltongenerated}. 
Furthermore, in \cite{MR2171084} the phrase ``in the presence of a long cycle 
every $k$-path-connected graph is $(k+1)$-generated'' \cite[Introduction, last 
paragraph]{MR2171084} cannot be construed so as to answer Question~\ref{question:doesdiracthresholdalreadyimplyhamiltongenerated}: each 
of the slightly different ways in which this phrase is made precise by the 
authors (cf. \cite[Corollary~5, Lemmas~9~and~10]{MR2171084}) involves additional 
assumptions one of which always is that there exists a circuit of 
length $2k-2$ or $2k-3$. The existence of such a circuit implies 
that `$(k+1)$-generated' is far from meaning `generated by Hamilton circuits'. 
Finally, \cite{MR2279163} is concerned with the same type of 
question as \cite{MR1174832} and Conjecure~\ref{thm:BondyConjecture} is again only 
mentioned in passing \cite[p.~12]{MR2279163}. 

\subsection{A positive example for 
Question~\ref{question:doesdiracthresholdalreadyimplyhamiltongenerated}}\label{87y7y6r5r5f5f5} 

We will now analyse a small yet relevant example which is a positive instance 
for Question~\ref{question:doesdiracthresholdalreadyimplyhamiltongenerated}. It 
provides an explicit illustration for how a minimum degree just barely 
satisfying the Dirac threshold can endow a \emph{non}-Cayley graph with the 
property of having its cycle space generated by its Hamilton circuits. 

\begin{definition}[{The graph $\upX$; this is the graph underlying 
Figure~\ref{7yt6yr52re34e34e4}.}]
Let $\upX$ be the graph defined by $\upV(\upX):=\{v_1,\dotsc,v_7\}$ 
and $\upE(\upX) := \bigl \{$ 
$v_1v_2,$ $v_1v_3,$ $v_1v_6,$ $v_1v_7,$ 
$v_2v_3,$ $v_2v_6,$ $v_2v_7,$ $v_3v_4,$ $v_3v_5,$ 
$v_4v_5,$ $v_4v_6,$ $v_4v_7,$ $v_5v_6,$ $v_5v_7$ $\bigr \}$. 
\end{definition}

Obviously $\upX$ satisfies the hypotheses 
in Question~\ref{question:doesdiracthresholdalreadyimplyhamiltongenerated} 
(but only barely so), and $\dim_{\Z/2}(\upX;\Z/2)$ $=$ $\beta_1(\upX)$ $=$ 
$f_1(\upX)$$-$ $f_0(\upX)$ $+$ $1$ $=$ $14-7+1$  $=$ $8$. 
Furthermore, because of the following fact we cannot prove that $\upX$ is a 
positive instance for 
Question~\ref{question:doesdiracthresholdalreadyimplyhamiltongenerated} just by  
appealing to 
Theorem~\ref{thm:AlspachLockeWitte}.\ref{thm:alspachlockewitte:numberofverticesodd}:
\begin{proposition}
The graph $\upX$ is not a Cayley graph.
\end{proposition}
\begin{proof}
The order $f_0(\upX)=7$ being prime, 
the only possible underlying group is $\Z/7$ with addition. Now suppose that 
$\upX$ were a Cayley graph on $\Z/7$. Since the spectrum of the adjacency matrix 
of $\upX$ is $(4,1,-1,-1,0,0,-3)\in\Z^7$, the graph $\upX$ would then be a 
quartic connected Cayley graph on an abelian group having only integer 
adjacency-eigenvalues. But this would contradict a classification theorem due to 
A.~Abdollahi and E.~Vatandoost \cite[Theorem~1.1]{ABDOLLAHIVATANDOOST_integralquarticcayleygraphsonabeliangroups} according to which the set of all orders of 
such graphs is a finite set which does not contain $7$. 
\end{proof}

\begin{figure} 
\begin{center}
\input{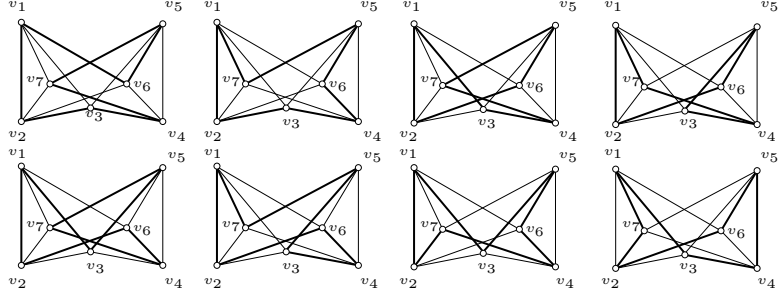}
\caption{An example of a $\Z/2$-basis for $\upZ_1(\upX;\Z/2)$ consisting only of 
Hamilton circuits in a situation where the underlying graph $\upX$ is not a 
Cayley graph and presumably owes its being Hamilton-generated to the 
Dirac condition (which it satisfies just barely).}
\label{7yt6yr52re34e34e4}
\end{center}
\end{figure}

\begin{figure} 
\begin{center}
\input{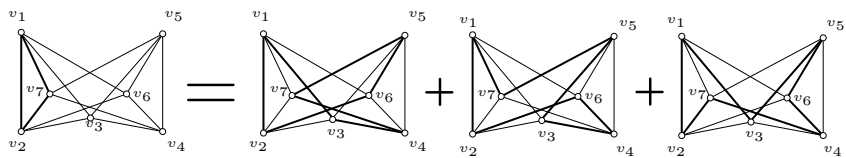}
\caption{An example of a realization of a $3$-circuit in terms of the 
Hamilton circuit basis from Figure~\ref{7yt6yr52re34e34e4}.}
\label{3d23sd35f6gh6h7}
\end{center}
\end{figure}
\begin{proposition}[{$\upX$ is Hamilton-generated}]
\label{125412412451327873483214687}
$\bigl \langle \mathcal{H}(\upX) \bigr\rangle_{\Z/2} = \upZ_1(\upX;\Z/2)$. 
\end{proposition}
\begin{proof}
Let us explicitly give a $\Z/2$-basis (shown in Figure~\ref{7yt6yr52re34e34e4}) for 
$\upZ_1(\upX;\Z/2)$ consisting of Hamilton circuits only (there is no particular 
reason why we choose this basis among several others). 
Let $C_1^{\upX} := v_1v_2v_3v_4v_7v_5v_6v_1$, 
$C_2^{\upX} :=  v_1v_2v_3v_4v_6v_5v_7v_1$, 
$C_3^{\upX} :=  v_1v_2v_6v_5v_7v_4v_3v_1$, 
$C_4^{\upX} :=  v_1v_2v_6v_5v_3v_4v_7v_1$, 
$C_5^{\upX} :=  v_1v_2v_6v_4v_7v_5v_3v_1$, 
$C_6^{\upX} :=  v_1v_2v_6v_4v_3v_5v_7v_1$, 
$C_7^{\upX} :=  v_1v_2v_7v_4v_6v_5v_3v_1$, 
$C_8^{\upX} :=  v_1v_7v_2v_6v_5v_4v_3v_1$.
All these circuits are Hamilton circuits of $\upX$. With respect to the standard 
basis of $\upC_1(\upX;\Z/2)$ the (chains of) the Hamilton circuits 
$C_1^{\upX}$, $\dotsc$, $C_8^{\upX}$ give rise to the matrix shown 
in \eqref{dr4e3s32ws344d45f24234234324}, which has $\Z/2$-rank 
equal to $8$ $=$ $\dim_{\Z/2}(\upZ_1(\upX;\Z/2))$. 
{\scriptsize
\begin{equation}\label{dr4e3s32ws344d45f24234234324}
\begin{smallmatrix}
& & C_1^{\upX} & C_2^{\upX} & C_3^{\upX} & C_4^{\upX} 
  & C_5^{\upX} & C_6^{\upX} & C_7^{\upX} & C_8^{\upX}  \\
& & & & & & & & &  \\
v_1\wedge v_2\  &     & 1 & 1 & 1 & 1 & 1 & 1 & 1 & 0 \\
v_1\wedge v_3\  &     & 0 & 0 & 1 & 0 & 1 & 0 & 1 & 1 \\
v_1\wedge v_6\  &     & 1 & 0 & 0 & 0 & 0 & 0 & 0 & 0 \\
v_1\wedge v_7\  &     & 0 & 1 & 0 & 1 & 0 & 1 & 0 & 1 \\
v_2\wedge v_3\  &     & 1 & 1 & 0 & 0 & 0 & 0 & 0 & 0 \\
v_2\wedge v_6\  &     & 0 & 0 & 1 & 1 & 1 & 1 & 0 & 1 \\
v_2\wedge v_7\  &     & 0 & 0 & 0 & 0 & 0 & 0 & 1 & 1 \\
v_3\wedge v_4\  &     & 1 & 1 & 1 & 1 & 0 & 1 & 0 & 1 \\
v_3\wedge v_5\  &     & 0 & 0 & 0 & 1 & 1 & 1 & 1 & 0 \\
v_4\wedge v_5\  &   & 0 & 0 & 0 & 0 & 0 & 0 & 0 & 1 \\
v_4\wedge v_6\  &   & 0 & 1 & 0 & 0 & 1 & 1 & 1 & 0 \\
v_4\wedge v_7\  &   & 1 & 0 & 1 & 1 & 1 & 0 & 1 & 0 \\
v_5\wedge v_6\  &   & 1 & 1 & 1 & 1 & 0 & 0 & 1 & 1 \\
v_5\wedge v_7\  &   & 1 & 1 & 1 & 0 & 1 & 1 & 0 & 0
\end{smallmatrix}
\end{equation}
}
Therefore the span of (the chains of) $C_1^{\upX}$, $\dotsc$, $C_8^{\upX}$ is 
an $8$-dimensional subspace of the $8$-dimensional $\Z/2$-vector 
space $\upZ_1(\upX;\Z/2)$, hence (this reasoning would not be valid over a 
general principal ideal domain) is \emph{equal} to $\upZ_1(\upX;\Z/2)$, 
completing the proof of Proposition~\ref{125412412451327873483214687}.
\end{proof}

\subsection{A group-theoretical question}\label{787y6767656555y88765756} Let us 
close by pointing out something else: the graph $\Pr_r$ can also be realized as 
a Cayley graph on the semi-direct product $\Z/2\ltimes \Z/r$ with $\Z/2$ acting 
on $\Z/r$ by inversion (this is the usual dihedral group). Therefore, $\Pr_r$ is 
an example of a graph which can \emph{simultaneously} be realized as a Cayley 
graph on an abelian and on a non-abelian group. There seems to be nothing 
known in general about such graphs, and it does not seem hopeless to attempt 
a classification: 
\emph{Which graphs are simultaneously Cayley graphs on a finite abelian group 
and on a finite non-abelian group?} \emph{And what can be deduced in general about 
the non-abelian groups which admit such a constellation?} While for Cayley graphs 
on infinite non-abelian groups the prospects of reaching a complete 
classification of those $2$-sets of (group,generator)-pairs with isomorphic 
Cayley graphs seem bleak (a point of departure to this topic can be 
\cite[Section~IV.A.9]{MR1786869}), the very strong assumption of requiring one 
of the two groups to be finite abelian might mean that a complete classification 
of such graphs and such groups can be found. 

\bibliographystyle{amsplain} 
\bibliography{HEINIGonprismsmoebiusladdersandthecyclespaceofdensegraphs_arXiv_version_20111221}

\providecommand{\bysame}{\leavevmode\hbox to3em{\hrulefill}\thinspace}
\providecommand{\MR}{\relax\ifhmode\unskip\space\fi MR }
\providecommand{\MRhref}[2]{%
  \href{http://www.ams.org/mathscinet-getitem?mr=#1}{#2}
}
\providecommand{\href}[2]{#2}
\begin{thebibliography}{10}

\bibitem{ABDOLLAHIVATANDOOST_integralquarticcayleygraphsonabeliangroups}
Alireza Abdollahi and Ebrahim Vatandoost, \emph{Integral quartic {C}ayley
  graphs on abelian groups}, Electron. J. Combin. \textbf{18} (2011),
  no.~{\#}P89.

\bibitem{MR2171084}
M.~Abreu, D.~Labbate, and Stephen~C. Locke, \emph{6-path-connectivity and
  6-generation}, Discrete Math. \textbf{301} (2005), no.~1, 20--27. \MR{2171084
  (2006f:05102)}

\bibitem{MR2279163}
M.~Abreu and Stephen~C. Locke, \emph{{$k$}-path-connectivity and
  {$mk$}-generation: an upper bound on {$m$}}, Graph theory in {P}aris, Trends
  Math., Birkh{\"a}user, Basel, 2007, pp.~11--19. \MR{2279163 (2007j:05125)}

\bibitem{MR2592513}
Brian Alspach, C.~C. Chen, and Matthew Dean, \emph{Hamilton paths in {C}ayley
  graphs on generalized dihedral groups}, Ars Math. Contemp. \textbf{3} (2010),
  no.~1, 29--47. \MR{2592513 (2011f:05135)}

\bibitem{MR1057481}
Brian Alspach, Stephen~C. Locke, and Dave Witte, \emph{The {H}amilton spaces of
  {C}ayley graphs on abelian groups}, Discrete Math. \textbf{82} (1990), no.~2,
  113--126. \MR{1057481 (91h:05058)}

\bibitem{MR1848324}
Brian Alspach and Yusheng Qin, \emph{Hamilton-connected {C}ayley graphs on
  {H}amiltonian groups}, European J. Combin. \textbf{22} (2001), no.~6,
  777--787. \MR{1848324 (2002g:05097)}

\bibitem{MR1296952}
A.~S. Asratian, \emph{A criterion for some {H}amiltonian graphs to be
  {H}amilton-connected}, Australas. J. Combin. \textbf{10} (1994), 193--198.
  \MR{1296952 (95m:05160)}

\bibitem{MR1128859}
A.~S. Asratyan, O.~A. Ambartsumyan, and G.~V. Sarkisyan, \emph{Some local
  conditions for the hamiltonicity and pancyclicity of a graph}, Akad. Nauk
  Armyan. SSR Dokl. \textbf{91} (1990), no.~1, 19--22. \MR{1128859 (92g:05124)}

\bibitem{MR527987}
L{\'a}szl{\'o} Babai, \emph{Almost all {$n$}-dimensional rectangular lattices
  are {H}amilton-laceable}, Proceedings of the {N}inth {S}outheastern
  {C}onference on {C}ombinatorics, {G}raph {T}heory, and {C}omputing ({F}lorida
  {A}tlantic {U}niv., {B}oca {R}aton, {F}la., 1978) (Winnipeg, Man.), Congress.
  Numer., XXI, Utilitas Math., 1978, pp.~649--661. \MR{527987 (81b:05075)}

\bibitem{MR1760293}
Mark~V. Barovich and Stephen~C. Locke, \emph{The cycle space of a 3-connected
  {H}amiltonian graph}, Discrete Math. \textbf{220} (2000), no.~1-3, 13--33.
  \MR{1760293 (2001d:05099)}

\bibitem{arXiv:1101.3099v1}
Sonny Ben-Shimon, Michael Krivelevich, and Benny Sudakov, \emph{On the
  resilience of hamiltonicity and optimal packing of {H}amilton cycles in
  random graphs}, SIAM J. Discrete Math. \textbf{25} (2011), no.~3, 1176--1193.

\bibitem{MR1943857}
Tom Bohman, Alan Frieze, and Ryan Martin, \emph{How many random edges make a
  dense graph {H}amiltonian?}, Random Structures Algorithms \textbf{22} (2003),
  no.~1, 33--42. \MR{1943857 (2004a:05143)}

\bibitem{MR0285424}
J.~A. Bondy, \emph{Pancyclic graphs. {I}}, J. Combinatorial Theory Ser. B
  \textbf{11} (1971), 80--84. \MR{0285424 (44 \#2642)}

\bibitem{MR1373656}
\bysame, \emph{Basic graph theory: paths and circuits}, Handbook of
  combinatorics, {V}ol.\ 1,\ 2, Elsevier, Amsterdam, 1995, pp.~3--110.
  \MR{1373656 (97a:05129)}

\bibitem{MR0414429}
J.~A. Bondy and V.~Chv{\'a}tal, \emph{A method in graph theory}, Discrete Math.
  \textbf{15} (1976), no.~2, 111--135. \MR{0414429 (54 \#2531)}

\bibitem{MR2735919}
Julia B{\"o}ttcher, Peter Heinig, and Anusch Taraz, \emph{Embedding into
  bipartite graphs}, SIAM J. Discrete Math. \textbf{24} (2010), no.~4,
  1215--1233. \MR{2735919}

\bibitem{MR2644412}
Julia B{\"o}ttcher, Klaas~P. Pruessmann, Anusch Taraz, and Andreas W{\"u}rfl,
  \emph{Bandwidth, expansion, treewidth, separators and universality for
  bounded-degree graphs}, European J. Combin. \textbf{31} (2010), no.~5,
  1217--1227. \MR{2644412}

\bibitem{MR2448444}
Julia B{\"o}ttcher, Mathias Schacht, and Anusch Taraz, \emph{Proof of the
  bandwidth conjecture of {B}ollob\'as and {K}oml\'os}, Math. Ann. \textbf{343}
  (2009), no.~1, 175--205. \MR{2448444 (2009g:05078)}

\bibitem{MR2443120}
Henning Bruhn and Xingxing Yu, \emph{Hamilton cycles in planar locally finite
  graphs}, SIAM J. Discrete Math. \textbf{22} (2008), no.~4, 1381--1392.
  \MR{2443120 (2009i:05136)}

\bibitem{CHAUDEBIASIOKIERSTEAD}
Phong Ch{\^ a}u, Louis DeBiasio, and H.~A. Kierstead, \emph{P{\'o}sa's
  conjecture for graphs of order at least {$2\times 10^8$}}, Random Structures
  Algorithms \textbf{39} (2011), no.~4, 507--525.

\bibitem{MR641233}
C.~C. Chen and Norman~F. Quimpo, \emph{On strongly {H}amiltonian abelian group
  graphs}, Combinatorial mathematics, {VIII} ({G}eelong, 1980), Lecture Notes
  in Math., vol. 884, Springer, Berlin, 1981, pp.~23--34. \MR{641233
  (83d:05051)}

\bibitem{arXiv:0908.4572v1}
Demetres Christofides, Daniela K{\"u}hn, and Deryk Osthus, \emph{Edge-disjoint
  {H}amilton cycles in graphs}, arXiv:0908.4572v1 [math.CO]. To appear in J.
  Combinatorial Theory Series B.

\bibitem{MR2520275}
Bill Cuckler and Jeff Kahn, \emph{Entropy bounds for perfect matchings and
  {H}amiltonian cycles}, Combinatorica \textbf{29} (2009), no.~3, 327--335.
  \MR{2520275 (2010j:05217)}

\bibitem{MR2520274}
\bysame, \emph{Hamiltonian cycles in {D}irac graphs}, Combinatorica \textbf{29}
  (2009), no.~3, 299--326. \MR{2520274 (2011a:05184)}

\bibitem{MR1935733}
A.~Czygrinow and H.~A. Kierstead, \emph{2-factors in dense bipartite graphs},
  Discrete Math. \textbf{257} (2002), no.~2-3, 357--369, Kleitman and
  combinatorics: a celebration (Cambridge, MA, 1999). \MR{1935733
  (2003h:05150)}

\bibitem{MR1786869}
Pierre de~la Harpe, \emph{Topics in geometric group theory}, Chicago Lectures
  in Mathematics, University of Chicago Press, Chicago, IL, 2000. \MR{1786869
  (2001i:20081)}

\bibitem{MR2080038}
Reinhard Diestel, \emph{On infinite cycles in graphs---or how to make graph
  homology interesting}, Amer. Math. Monthly \textbf{111} (2004), no.~7,
  559--571. \MR{2080038}

\bibitem{MR2128083}
\bysame, \emph{The cycle space of an infinite graph}, Combin. Probab. Comput.
  \textbf{14} (2005), no.~1-2, 59--79. \MR{2128083 (2005m:05123)}

\bibitem{MR2159259}
\bysame, \emph{Graph theory}, third ed., Graduate Texts in Mathematics, vol.
  173, Springer-Verlag, Berlin, 2005. \MR{2159259 (2006e:05001)}

\bibitem{MR2057684}
Reinhard Diestel and Daniela K{\"u}hn, \emph{On infinite cycles. {I}},
  Combinatorica \textbf{24} (2004), no.~1, 69--89. \MR{2057684 (2005d:05085a)}

\bibitem{MR2789733}
Reinhard Diestel and Philipp Spr{\"u}ssel, \emph{The homology of a locally
  finite graph with ends}, Combinatorica \textbf{30} (2010), no.~6, 681--714.
  \MR{2789733}

\bibitem{MR0047308}
G.~A. Dirac, \emph{Some theorems on abstract graphs}, Proc. London Math. Soc.
  (3) \textbf{2} (1952), 69--81. \MR{0047308 (13,856e)}

\bibitem{MR2430433}
Alan Frieze and Michael Krivelevich, \emph{On two {H}amilton cycle problems in
  random graphs}, Israel J. Math. \textbf{166} (2008), 221--234. \MR{2430433
  (2009g:05094)}

\bibitem{arXiv:1003.5115v3}
Agelos Georgakopoulos, \emph{Cycle decompositions: from graphs to continua},
  (2010), arXiv:1003.5115v3 [math.GN].

\bibitem{MR2330890}
Pierre~Antoine Grillet, \emph{Abstract algebra}, second ed., Graduate Texts in
  Mathematics, vol. 242, Springer, New York, 2007. \MR{2330890 (2008c:20001)}

\bibitem{MR0224499}
Richard~K. Guy and Frank Harary, \emph{On the {M}\"obius ladders}, Canad. Math.
  Bull. \textbf{10} (1967), 493--496. \MR{0224499 (37 \#98)}

\bibitem{MR725072}
Irith Ben-Arroyo Hartman, \emph{Long cycles generate the cycle space of a
  graph}, European J. Combin. \textbf{4} (1983), no.~3, 237--246. \MR{725072
  (85b:05113)}

\bibitem{MR2089014}
Pavol Hell and Jaroslav Ne{\v{s}}et{\v{r}}il, \emph{Graphs and homomorphisms},
  Oxford Lecture Series in Mathematics and its Applications, vol.~28, Oxford
  University Press, Oxford, 2004. \MR{2089014 (2005k:05002)}

\bibitem{MR1788124}
Wilfried Imrich and Sandi Klav{\v{z}}ar, \emph{Product graphs},
  Wiley-Interscience Series in Discrete Mathematics and Optimization,
  Wiley-Interscience, New York, 2000, Structure and recognition, With a
  foreword by Peter Winkler. \MR{1788124 (2001k:05001)}

\bibitem{MR527728}
Bill Jackson, \emph{Edge-disjoint {H}amilton cycles in regular graphs of large
  degree}, J. London Math. Soc. (2) \textbf{19} (1979), no.~1, 13--16.
  \MR{527728 (80k:05078)}

\bibitem{JAMSHEDembeddingspanningsubgraphsintolargedensegraphs}
Asif Jamshed, \emph{Embedding spanning subgraphs into large dense graphs},
  Ph.D. thesis, Rutgers, The State University of New Jersey, October 2010.

\bibitem{MR625067}
Alexander~K. Kelmans, \emph{A new planarity criterion for {$3$}-connected
  graphs}, J. Graph Theory \textbf{5} (1981), no.~3, 259--267. \MR{625067
  (83a:05058)}

\bibitem{arXiv:1104.4412v1}
Fiachra Knox, Daniela K{\"u}hn, and Deryk Osthus, \emph{Edge-disjoint
  {H}amilton cycles in random graphs}, arXiv:1104.4412v1 [math.CO].

\bibitem{arXiv:1006.1268v1}
\bysame, \emph{Approximate {H}amilton decompositions of random graphs},
  (2010), arXiv:1006.1268v1 [math.CO]. To appear in Random Structures {\&}
  Algorithms.

\bibitem{MR1611764}
J{\'a}nos Koml{\'o}s, G{\'a}bor~N. S{\'a}rk{\"o}zy, and Endre Szemer{\'e}di,
  \emph{On the square of a {H}amiltonian cycle in dense graphs}, Proceedings of
  the {S}eventh {I}nternational {C}onference on {R}andom {S}tructures and
  {A}lgorithms ({A}tlanta, {GA}, 1995), vol.~9, 1996, pp.~193--211. \MR{1611764
  (99f:05073)}

\bibitem{arXiv:1109.5341v1}
Michael Krivelevich and Wojciech Samotij, \emph{Optimal packings of {H}amilton
  cycles in sparse random graphs},  (2011), arXiv:1109.5341v1 [math.CO].

\bibitem{arXiv:1108.2502v1}
Choongbum Lee and Benny Sudakov, \emph{Dirac's theorem for random graphs},
  (2011), arXiv:1108.2502v1 [math.CO].

\bibitem{MR818599}
Stephen~C. Locke, \emph{A basis for the cycle space of a {$2$}-connected
  graph}, European J. Combin. \textbf{6} (1985), no.~3, 253--256. \MR{818599
  (87g:05138)}

\bibitem{MR821540}
\bysame, \emph{A basis for the cycle space of a {$3$}-connected graph}, Cycles
  in graphs ({B}urnaby, {B}.{C}., 1982), North-Holland Math. Stud., vol. 115,
  North-Holland, Amsterdam, 1985, pp.~381--397. \MR{821540 (87h:05136)}

\bibitem{MR815581}
\bysame, \emph{A generalization of {D}irac's theorem}, Combinatorica \textbf{5}
  (1985), no.~2, 149--159. \MR{815581 (87c:05077)}

\bibitem{MR1174832}
\bysame, \emph{Long paths and the cycle space of a graph}, Ars Combin.
  \textbf{33} (1992), 77--85. \MR{1174832 (94d:05076)}

\bibitem{MIKLAVICSPARLonextendabilityofcayleygraphs2009}
{\v S}tefko Miklavi{\v c} and Primo{\v z} {\v S}parl, \emph{On extendability of
  {C}ayley graphs}, Filomat \textbf{23} (2009), no.~3, 93--101.

\bibitem{MIKLAVICSPARLhamiltoncycleandhamiltonpathextendability2011}
\bysame, \emph{Hamilton cycle and {H}amilton path extendability of cayley
  graphs on abelian groups}, J. Graph Theory (2011), Published online but not
  included in an issue yet; doi:10.1002/jgt.

\bibitem{MR0299517}
C.~St. J.~A. Nash-Williams, \emph{Hamiltonian lines in graphs whose vertices
  have sufficiently large valencies}, Combinatorial theory and its
  applications, {III} ({P}roc. {C}olloq., {B}alatonf\"ured, 1969),
  North-Holland, Amsterdam, 1970, pp.~813--819. \MR{0299517 (45 \#8565)}

\bibitem{MR0284366}
\bysame, \emph{Edge-disjoint {H}amiltonian circuits in graphs with vertices of
  large valency}, Studies in {P}ure {M}athematics ({P}resented to {R}ichard
  {R}ado), Academic Press, London, 1971, pp.~157--183. \MR{0284366 (44 \#1594)}

\bibitem{MR2510568}
G{\'a}bor~N. S{\'a}rk{\"o}zy, \emph{A fast parallel algorithm for finding
  {H}amiltonian cycles in dense graphs}, Discrete Math. \textbf{309} (2009),
  no.~6, 1611--1622. \MR{2510568 (2010m:05314)}

\bibitem{MR1969376}
G{\'a}bor~N. S{\'a}rk{\"o}zy, Stanley~M. Selkow, and Endre Szemer{\'e}di,
  \emph{On the number of {H}amiltonian cycles in {D}irac graphs}, Discrete
  Math. \textbf{265} (2003), no.~1-3, 237--250. \MR{1969376 (2004a:05002)}

\bibitem{MR0158387}
W.~T. Tutte, \emph{How to draw a graph}, Proc. London Math. Soc. (3)
  \textbf{13} (1963), 743--767. \MR{0158387 (28 \#1610)}

\end{thebibliography}

\end{document}